\documentclass[11pt]{article}
\usepackage[margin=0.75in]{geometry}

\usepackage{amsthm}
\usepackage{amssymb}
\usepackage{amsmath}
\usepackage{comment}
\usepackage{thm-restate}
\usepackage{url} 
\usepackage{hyperref}
\usepackage[noabbrev,capitalise]{cleveref}

\usepackage[utf8]{inputenc}

\usepackage{color}
\usepackage[normalem]{ulem}

\newcounter{i}
\theoremstyle{plain}
\newtheorem{thm}{Theorem}[section]
\newtheorem{lem}[thm]{Lemma}
\newtheorem{proposition}[thm]{Proposition}
\newtheorem{fact}[thm]{Fact}
\newtheorem{cor}[thm]{Corollary}

\newtheorem{conj}[thm]{Conjecture}

\newenvironment{proofclaim}[1][]%
{\noindent \emph{Proof.} {}{#1}{}}{\hfill
	$\Diamond$\vspace{1em}}

\usepackage{amsthm,thmtools}

\declaretheorem[
  style=plain,
  name=Claim,
  within=theorem,
]{claim}

\newenvironment{lateproof}[1]
 {%
  \renewcommand{\theclaim}{\ref{#1}.\arabic{claim}}%
  \begin{proof}[Proof of~\cref{#1}]%
 }
 {\end{proof}}

\usepackage{etoolbox}
 \AtEndEnvironment{proof}{\setcounter{claim}{0}}

\theoremstyle{plain} 
\newcommand{\thistheoremname}{}
\newtheorem{genericthm}{\thistheoremname}

\theoremstyle{definition}
\newtheorem{definition}[thm]{Definition}

\newtheorem{remark}[thm]{Remark}

\newcommand{\Prob}[1]{\ensuremath{%
\mathbb P\left[#1\right]
}}

\newcommand{\Expect}[1]{\ensuremath{%
\mathbb E\left[#1\right]
}}

\title{Refined Absorption: A New Proof of the Existence Conjecture}

\author{
Michelle Delcourt
\thanks{Department of Mathematics, Toronto Metropolitan University (formerly named Ryerson University),
Toronto, Ontario M5B 2K3, Canada {\tt mdelcourt@torontomu.ca}. Research supported by NSERC under Discovery Grant No. 2019-04269.}
\and
Luke Postle
\thanks{Combinatorics and Optimization Department,
University of Waterloo, Waterloo, Ontario N2L 3G1, Canada {\tt lpostle@uwaterloo.ca}. Partially supported by NSERC
under Discovery Grant No. 2019-04304.}}
\date{February 29, 2024}

\begin{document}

\maketitle

\begin{abstract} 
The study of combinatorial designs has a rich history spanning nearly two centuries.  In a recent breakthrough, the notorious Existence Conjecture for Combinatorial Designs dating back to the 1800s was proved in full by Keevash via the method of randomized algebraic constructions. Subsequently Glock, K\"{u}hn, Lo, and Osthus provided an alternate purely combinatorial proof of the Existence Conjecture via the method of iterative absorption.  We introduce a novel method of \emph{refined absorption} for designs; here as our first application of the method we provide a new alternate proof of the Existence Conjecture (assuming the existence of $K_q^r$-absorbers by Glock, K\"{u}hn, Lo, and Osthus). 
\end{abstract}

\section{Introduction}

\subsection{The Existence Conjecture for Combinatorial Designs}

A \emph{Steiner system} with parameters $(n,q,r)$ is a set $S$ of $q$-subsets of an $n$-set $X$ such that every $r$-subset of $X$ belongs to exactly one element of $S$. More generally, a \emph{design} with parameters $(n,q,r,\lambda)$ is a set $S$ of $q$-subsets of an $n$-set $X$ such that every $r$-subset of $X$ belongs to exactly $\lambda$ elements of $S$. 

The notorious Existence Conjecture originating from the mid-1800's asserts that designs exist for large enough $n$ provided the obvious necessary divisibility conditions are satisfied as follows.

\begin{conj}[Existence Conjecture]\label{conj:Existence}
Let $q > r \ge 2$ and $\lambda \geq 1$ be integers. If $n$ is sufficiently large and $\binom{q-i}{r-i}~|~\lambda \binom{n-i}{r-i}$ for all $0\le i \le r-1$, then there exists a design with parameters $(n,q,r,\lambda)$.
\end{conj}

In 1847, Kirkman~\cite{K47} proved this when $q=3$, $r=2$ and $\lambda=1$. In the 1970s, Wilson~\cite{WI, WII, WIII} proved the Existence Conjecture for graphs, i.e.~when $r=2$ (for all $q$ and $\lambda$). In 1985, R\"{o}dl~\cite{R85} introduced his celebrated ``nibble method'' to prove that there exists a set $S$ of $q$-subsets of an $n$-set $X$ with $|S| = (1-o(1))\binom{n-r}{q-r}$ such that every $r$-subset is in at most one element of $S$, thereby settling the approximate version of the Existence Conjecture (known as the Erd\H{o}s-Hanani Conjecture~\cite{EH63}). Only in the last decade was the Existence Conjecture fully resolved as follows.

\begin{thm}[Keevash~\cite{K14}]\label{thm:Existence}
Conjecture~\ref{conj:Existence} is true.   
\end{thm}

Namely in 2014, Keevash~\cite{K14} proved the Existence Conjecture using \emph{randomized algebraic constructions}. Thereafter in 2016, Glock, K\"{u}hn, Lo, and Osthus~\cite{GKLO16} gave a purely combinatorial proof of the Existence Conjecture via \emph{iterative absorption}. Both approaches have different benefits and each has led to subsequent work using these approaches. We refer the reader to~\cite{W03, K14, GKLO16} for more history on the conjecture.

Our main purpose of this paper is to introduce our method of \emph{refined absorption} and then to provide a new alternate proof of the Existence Conjecture via said method.  We note that our proof assumes the existence of $K_q^r$-absorbers (as established by Glock, K\"{u}hn, Lo, and Osthus~\cite{GKLO16}) but sidesteps the use of iterative absorption. Our approach has some of the benefits of both of the previous proofs, namely we provide a purely combinatorial approach but one that utilizes single-step absorption and has the potential to be useful for a number of applications. For example, in a follow-up paper~\cite{DPII}, we use this new approach to prove the existence of high girth Steiner systems.

\subsection{Refined Absorbers}

A transformative concept that upended design theory and many similar exact structural decomposition problems is the \emph{absorbing method}, a method for transforming almost perfect decompositions into perfect ones, wherein a set of absorbers are constructed with the ability to `absorb' any particular uncovered `leftover' of a random greedy or nibble process into a larger decomposition. One of the original applications of the absorbing method by R\"odl, Ruci\'nski, and Szemer\'edi~\cite{RRS06} involved a single absorbing step via the so-called Absorbing Lemma, where a potential leftover is absorbed in one step. For embedding problems (e.g. perfect matching, $K_r$-factors, Hamilton cycles) where one wishes to decompose the vertices of a graph, the Absorbing Lemma has been utilized to provide absorbing proofs of old results (such as the Corr\'adi-Hajnal~\cite{CH63} and Hajnal-Szemer\'edi~\cite{HS70} Theorems) as well as many new ones. In cases of embedding problems where the Absorbing Lemma fails to apply, other absorbing techniques have been developed such as Montgomery's~\cite{M19b} method of ``absorber-templates''. Our refined absorbers (see below) may in some sense be viewed as ``absorber-templates'' for decomposition problems. 

The Existence Conjecture may be rephrased in graph theoretic terms. We refer the reader to the notation section (Section~\ref{ss:Notation} below) for the relevant background definitions and notation. For $K_q^r$ decompositions, the key definition of an absorber is as follows.

\begin{definition}[Absorber]
Let $L$ be a $K_q^r$-divisible hypergraph. A hypergraph $A$ is a \emph{$K_q^r$-absorber} for $L$ if $V(L)\subseteq V(A)$ is independent in $A$ and both $A$ and $L\cup A$ admit $K_q^r$ decompositions.
\end{definition}

Glock, K\"uhn, Lo, and Osthus~\cite{GKLO16} proved that a $K_q^r$-absorber $A_L$ exists for any $K_q^r$-divisible graph $L$ as follows.  

\begin{thm}[Glock, K\"uhn, Lo, and Osthus~\cite{GKLO16}]\label{thm:AbsorberExistence}
Let $q > r\ge 1$ be integers. If $L$ is a $K_q^r$-divisible hypergraph, then there exists a $K_q^r$-absorber for $L$.
\end{thm}

The main use of absorbers is to absorb the potential `leftover' $K_q^r$-divisible graph of a specific set/graph $X$ after a nibble/random-greedy process constructs a $K_q^r$-decomposition covering all edges outside of $X$. To that end, we make the following definition.

\begin{definition}[Omni-Absorber]
Let $q > r\ge 1$ be integers. Let $X$ be a hypergraph. We say a hypergraph $A$ is a \emph{$K_q^r$-omni-absorber} for $X$ with \emph{decomposition family} $\mathcal{H}$ and \emph{decomposition function} $\mathcal{Q}_A$ if $V(X)=V(A)$, $X$ and $A$ are edge-disjoint, $\mathcal{H}$ is a family of subgraphs of $X\cup A$ each isomorphic to $K_q^r$ such that $|H\cap X|\le 1$ for all $H\in\mathcal{H}$, and for every $K_q^r$-divisible subgraph $L$ of $X$, there exists $\mathcal{Q}_A(L)\subseteq \mathcal{H}$ that are pairwise edge-disjoint and such that $\bigcup \mathcal{Q}_A(L)=L\cup A$. 
\end{definition}

In particular note that since the empty subgraph is $K_q^r$-divisible, the definition above implies that $A$ admits a $K_q^r$-decomposition. Moreover, since an absorber $A_L$ exists for every $L$, we have that a $K_q^r$-omni-absorber for $X$ exists provided that $X$ has enough isolated vertices compared to its number of edges; the construction is simply taking $A$ to be the disjoint union of the $A_L$ for every $K_q^r$-divisible subgraph $L$ of $X$.

A natural question then is how efficient an omni-absorber can we build? Taking disjoint absorbers for every $L$ yields an omni-absorber $A$ with $e(A) = 2^{\Omega(e(X))}$. In their paper proving the existence of high girth Steiner triple systems, Kwan, Sah, Sawhney, and Simkin~\cite{KSSS22} constructed so-called \emph{high-efficiency absorbers} for triangle decompositions, which yields a $K_3^2$-omni-absorber for $X$ with $e(A) = O(e(X))^{15}$. This raises the question whether such high efficiency absorbers exist for $K_q^r$ decompositions and whether even more efficient omni-absorbers could exist. 

One main purpose of this paper is to show that such an extremely efficient $K_q^r$-omni-absorber $A$ for $X$ exists, in particular such that $\Delta(A)=O(\Delta(X))$ provided that $\Delta(X)$ is large enough (where $\Delta(A)$ denotes the $(r-1)$-degree of $A$, also written as $\Delta_{r-1}(A)$, that is defined as the maximum over all $(r-1)$-sets $S$ of the number of edges of $A$ containing $S$). 

Here is the key definition for proving their existence.

\begin{definition}[Refined Omni-Absorber]
Let $\mathcal{H}$ be a family of subgraphs of a hypergraph $G$. For $e\in G$, we define $\mathcal{H}(e):= \{H\in \mathcal{H}: e\in H\}$. 

Let $C\ge 1$ be real. We say a $K_q^r$-omni-absorber $A$ for a hypergraph $X$ with decomposition family $\mathcal{H}$ is \emph{$C$-refined} if $|\mathcal{H}(e)|\le C$ for every edge $e\in X\cup A$.
\end{definition}

Our main result for omni-absorbers is the following. 

\begin{thm}[Refined Omni-Absorber Theorem]\label{thm:Omni}
For all integers $q > r\ge 1$, there exist an integer $C\ge 1$ and real $\varepsilon\in (0,1)$ such that the following holds: Let $G$ be an $r$-uniform hypergraph on $n$ vertices with $\delta(G)\ge (1-\varepsilon)n$. If $X$ is a spanning subhypergraph of $G$ with $\Delta(X) \le \frac{n}{C}$ and we let $\Delta:= \max\left\{\Delta(X),~n^{1-\frac{1}{r}}\cdot \log n\right\}$, then there exists a $C$-refined $K_q^r$-omni-absorber $A\subseteq G$ for $X$ such that $\Delta(A)\le C \cdot \Delta$. 
\end{thm}

We note that Theorem~\ref{thm:Omni} with its extremely efficient omni-absorbers is the key to our new proof of the Existence Conjecture. Namely, we will take $X$ to be a random subset of the edges of $K_n^r$ (taken independently with some well-chosen small probability $p$), apply Theorem~\ref{thm:Omni} to find a $K_q^r$-omni-absorber $A$ for $X$, and then use the Boosting Lemma of Glock, K\"{u}hn, Lo, and Osthus~\cite{GKLO16} combined with the nibble method to find a $K_q^r$-packing of $K_n^r\setminus A$ that covers all edges in $K_n^r\setminus (X\cup A)$. By the definition of omni-absorber then, any leftover $L$ of $X$ can be absorbed into $L\cup A$, thereby completing the $K_q^r$ decomposition. See Section~\ref{s:ProofOverview} for a more formal proof overview. 

Phrased in other terms, we have replaced iterative absorption (and the use of a so-called vertex vortex) with a single-step absorption (and the use of a one layer edge vortex). See Section~\ref{ss:Vortex} for more discussion on this but crucially this strategy only works when we have near linear efficiency in omni-absorbers.

The inquiring reader may wonder then whether it is necessary to prove that the $K_q^r$-omni-absorber in Theorem~\ref{thm:Omni} is $C$-refined for our proof of Theorem~\ref{thm:Existence}. Technically, the refinedness is not needed and only the high efficiency ($\Delta(A)\le C\cdot \Delta$) is required. However, the property of being $C$-refined is key to the other applications in our series (see the discussion of the next paragraph); furthermore, it is necessary for our proof of Theorem~\ref{thm:Omni} which proceeds by induction on the uniformity and uses the $C$-refinedness inductively. Indeed, this refined concept is a key innovation driving the proof as will become evident from the introduction of \emph{refiners} in the next subsection, but also through the proof overview of Theorem~\ref{thm:Omni} and on to its use in the proof of the Refine Down Lemma.

The true benefit of our new proof of Theorem~\ref{thm:Existence} is in the robustness of the proof structure. One may use Theorem~\ref{thm:Omni} as a structural black box - a template that one can modify to prove various generalizations or variations of the Existence Conjecture. By embedding various gadgets on to the cliques in the decomposition family of the omni-absorber (what we generally call \emph{boosters}), we can build omni-absorbers suited to other settings. Note this is only possible since our omni-absorbers are $C$-refined and hence the number of cliques in the decomposition family is also small. Indeed, we proceed with some of these applications in the remainder of our series, namely to high girth designs and to graph designs in random settings. See our conclusion, Section~\ref{s:Conclusion}, for discussion of these applications.

The informed reader may wonder whether Theorem~\ref{thm:Omni} follows from known results about the Existence Conjecture. Keevash~\cite{K14} in fact proved a typicality generalization of the Existence Conjecture (Theorem 1.4 in~\cite{K14}, the main theorem of said paper); this in turn yields a weaker version of Theorem~\ref{thm:Omni}, namely by using a typical $K_q^r$-divisible graph $A$ much denser than $X$ (say at random but then deleting some edges to fix divisibility) and then arguing that $L\cup A$ is typical for any $K_q^r$-divisible subgraph $L$ of $X$ and hence admits a $K_q^r$-decomposition. This version would have worse parameters (specifically $\Delta := \max \left\{ \left(\frac{\Delta(X)}{n}\right)^{1/h}\cdot n,~n^{1-\varepsilon} \right\}$ for some large constants $h$ and $1/\varepsilon$, both at least exponential in $q$). More importantly, the decomposition family $\mathcal{H}$ might consist of all the cliques in $X\cup A$ since the various decompositions need not align and so every edge may be in $\Omega(n^{q-r})$ members of $\mathcal{H}$ as opposed to the constantly many of Theorem~\ref{thm:Omni}, a property which is crucial to the applications described above. 

\subsection{Refiners}\label{ss:Refiner}

In order to build extremely efficient and refined omni-absorbers (that is, prove Theorem~\ref{thm:Omni}), we introduce the following key definition which is analogous to omni-absorbers except it allows decomposing into $K_q^r$-divisible subgraphs instead of cliques.

\begin{definition}[Refiner]
Let $q > r\ge 1$ be integers. Let $X$ be an $r$-uniform hypergraph. We say an $r$-uniform hypergraph $R$ is a \emph{$K_q^r$-refiner} of $X$ with \emph{refinement family} $\mathcal{H}$ and \emph{refinement function} $\mathcal{Q}_R$ if $V(X)=V(R)$, $X$ and $R$ are edge-disjoint and $\mathcal{H}$ is a family of $K_q^r$-divisible subgraphs of $X\cup R$ such that $|H\cap X|\le 1$ for all $H\in\mathcal{H}$ and for every $K_q^r$-divisible subgraph $L$ of $X$, there exists $\mathcal{Q}_R(L)\subseteq \mathcal{H}$ that are pairwise edge-disjoint and such that $\bigcup \mathcal{Q}_R(L)=L\cup R$. 
\end{definition}

We note that it follows from the definition that a $K_q^r$-refiner is $K_q^r$-divisible since $R = \bigcup \mathcal{Q}_R(\emptyset)$. We also note that if every element of $\mathcal{H}$ is a $K_q^r$ then the refiner is in fact an omni-absorber. 

We also need an analogous notion of $C$-refined for refiners. Since the members of a refinement family need not be cliques, it is useful to impose that they have bounded size. Hence we introduce the following definition.

\begin{definition}[$C$-Refined Family]
Let $C\ge 1$ be real. Let $\mathcal{H}$ be a family of subgraphs of a hypergraph $G$.  We say $\mathcal{H}$ is \emph{$C$-refined} if $|\mathcal{H}(e)|\le C$ for every edge $e\in G$ and $\max\{v(H),e(H)\}\le C$ for every $H\in \mathcal{H}$.
\end{definition}

We note that the second condition that $\max\{v(H),e(H)\}\le C$ for every $H\in \mathcal{H}$ is trivially satisfied for a decomposition family $\mathcal{H}$ of an omni-absorber (at least for $C\ge \binom{q}{r}$). This is why we defined a $C$-refined omni-absorber as we did. Here then is our notion of $C$-refined for a refiner.

\begin{definition}[$C$-Refined Refiner]
Let $C\ge 1$ be a real number. We say a $K_q^r$-refiner $R$ with refinement family $\mathcal{H}$ is \emph{$C$-refined} if $\mathcal{H}$ is $C$-refined. Let $S\in \binom{V(R)}{r-1}$. The \emph{refinement degree} of $S$ in $\mathcal{H}$, denoted $\Delta_{\mathcal{H}}(S)$, is $|\{H\in \mathcal{H}: S\subseteq V(H)\}|$. The \emph{refinement degree} of $\mathcal{H}$, denoted $\Delta(\mathcal{H})$, is $\max \left\{ \Delta_{\mathcal{H}}(S): S\in \binom{V(R)}{r-1} \right\}$.
\end{definition}

If $R$ is a $K_q^r$-refiner of $X$, we note that only requiring the second condition of $C$-refined, or equivalently decomposing $L\cup R$ into bounded-size $K_q^r$-divisible graphs for every $K_q^r$-divisible subgraph $L$ of $X$, is enough to prove the existence of \emph{high-efficiency omni-absorbers}.  Namely taking $A$ to be the union of $R$ and edge-disjoint private absorbers for every $K_q^r$-divisible subgraph of $X\cup R$ on at most $C$ vertices would yield a $K_q^r$-omni-absorber $A$ of $X$ with $e(A) = O\left( ( v(X)^C\right)$ (where here of course one would have to add enough new isolated vertices to $X$ so as to embed these absorbers).  

However, to construct \emph{extremely efficient omni-absorbers}, that is with $\Delta(A) = O(\Delta(X))$, we require the first condition of $C$-refined, namely that every edge of $X\cup R$ is in at most $C$ of the $K_q^r$-divisible subgraphs in $\mathcal{H}$. Then the existence of a refined absorber will follow by embedding a private absorber for each element of $\mathcal{H}$ (in such a way that the absorbers are edge-disjoint and have low overall maximum degree). 

Here is our main result for refiners. Note it is a weaker version of Theorem~\ref{thm:Omni} which produces a refiner instead of omni-absorber, but from which we will derive Theorem~\ref{thm:Omni} as described in the paragraph above.

\begin{thm}[Refiner Theorem]\label{thm:Refiner}
 For all integers $q> r\ge 1$, there exist an integer $C\ge 1$ and real $\varepsilon \in (0,1)$ such that the following holds: Let $G$ be an $r$-uniform hypergraph on $n$ vertices with $\delta(G)\ge (1-\varepsilon)n$. If $X$ is a spanning subhypergraph of $G$ such that $\Delta(X) \le \frac{n}{C}$ and we let $\Delta:= \max\left\{\Delta(X),~n^{1-\frac{1}{r}}\cdot \log n\right\}$, then there exists a $C$-refined $K_q^r$-refiner $R\subseteq G$ of $X$ with refinement family $\mathcal{H}$ such that $\Delta(\mathcal{H}) \le C \cdot \Delta$.
\end{thm}

\subsection{Single-Step Absorption (and Edge Vortex) vs Iterative Absorption}\label{ss:Vortex}

The purely combinatorial proof of the Existence Conjecture by Glock, K\"{u}hn, Lo, and Osthus~\cite{GKLO16}, as well as many other recent breakthroughs in design theory and hypergraph decompositions, proceeds via \emph{iterative absorption}, in which the potential leftovers of a decomposition problem are iteratively reduced until one final leftover is completely absorbed. The method of iterative absorption was first introduced by Knox, K\"{u}hn, and Osthus in~\cite{KKO15} and first used for decompositions by K\"{u}hn and Osthus in~\cite{KO13}. There are two natural approaches to utilizing iterative absorption for graph designs as follows. 

The first approach that we mention (and arguably the simpler and more natural one for this setting) is to use a so-called \emph{edge vortex} $E(K_n^r)=E_0 \supset E_1 \supset E_2 \ldots \supset E_k = E(X)$, where $E_{i+1}$ is a small randomly chosen subset of $E_{i}$. Then a nibble or random-greedy process is iteratively applied to decompose almost all the edges of $E_i$ into $K_q^r$'s and the remaining edges of $E_i$ are then decomposed into $K_q^r$'s whose other edges lie only in $E_{i+1}$. The benefit of this approach is its simplicity.  However, the problem is that for the penultimate step to succeed, we require that $E_{k-1}$ contains $K_q^r$'s; even for triangle decompositions the probability that $E_{k-1}$ contains a triangle is effectively zero when $|E_{k-1}| \ll n^{3/2}$, and hence we require $|E_{k-1}| \ge n^{3/2}$. 
Therefore historically without access to extremely efficient omni-absorbers, this approach was not feasible.

A second approach is to use a so-called \emph{vertex vortex}, $V(K_n^r) = V_0 \supset V_1 \supset V_2 \ldots \supset V_k = V(X)$, where $V_{i+1}$ is a small random subset of $V_i$. Then a nibble or random-greedy process is employed to decompose almost all the edges of $G[V_i]\setminus G[V_{i+1}]$. The benefit here is that $v(X)$ may be small, even constant sized, and hence the use of inefficient omni-absorbers is permitted. The cost is that the remaining edges from $V_i\setminus V_{i+1}$ to $V_{i+1}$ must be decomposed and these edges must be decomposed via $K_{q-j}^{r-j}$ factors inside $V_{i+1}$ where $j$ is the number of vertices of the edge in $V_i\setminus V_{i+1}$; this is the technical heart of the Cover Down Lemma of Glock, K\"{u}hn, Lo, and Osthus~\cite{GKLO16}. One issue with the cover down approach is that it does not produce a black box theorem (as in Theorem~\ref{thm:Omni}) and so the whole proof and method must be modified for various variants. 

The first approach has not been generally used for designs due to only having access to inefficient omni-absorbers. Aided now by our efficient omni-absorbers, we use the first approach to provide our new proof of the Existence Conjecture. Indeed, since our omni-absorbers are so efficient, we only use one step; that is, we only need to choose one random subset $E_1$ of $E_0=E(K_n^r)$. 

Thus we no longer need to iterate the cover down. Indeed, instead of iteratively \emph{covering} down to $X$ what needs to be absorbed via random processes, our refiners (in their construction which also employs a vertex vortex of sorts) can be viewed as iteratively \emph{refining} down  what needs to be absorbed via structural methods.

\subsection{Outline of Paper}

In Section~\ref{s:ProofOverview}, we overview our new proof of the Existence Conjecture; we also provide more details on how our refiners are constructed (of special importance is the construction in the $r=1$ case) and also on what nibble result we need.  In Section~\ref{s:Refined}, we introduce multigraph versions of Theorems~\ref{thm:Omni} and~\ref{thm:Refiner}, namely Theorems~\ref{thm:MultiOmni} and~\ref{thm:MRT}. We then prove Theorem~\ref{thm:MRT} modulo the proof of the Refine Down Lemma (Lemma~\ref{lem:RefineDown}). In Section~\ref{s:RefineDown}, we prove Lemma~\ref{lem:RefineDown}. In Section~\ref{s:Embedding}, we embed fake edges and absorbers to prove Theorems~\ref{thm:MultiOmni},~\ref{thm:Refiner}, and~\ref{thm:Omni} in that order.  In Section~\ref{s:Existence}, we provide a formal proof of the Existence Conjecture using refined absorption following the outline we gave in Section~\ref{s:ProofOverview}.  Finally in Section~\ref{s:Conclusion}, we provide concluding remarks and discuss further directions.

\subsection{Notation and Standard Definitions}\label{ss:Notation}

A \emph{hypergraph} $H$ consists of a pair $(V, E)$ where $V$ is a set whose elements are called \emph{vertices} and $E$ is a
set of subsets of $V$ called \emph{edges}; we also write $H$ for its set of edges $E(H)$ for brevity. Similarly, we write $v(H)$ for the number of vertices of $H$ and either $e(H)$ or alternatively $|H|$ for the number of edges of $H$. A \emph{multi-hypergraph} is a pair $(V,E)$ where $V$ is a set and $E$ is a multi-set of subsets of $V$. (Note we do not allow sub-multisets here, i.e.~no `loops').

For an integer $r\ge 1$, a (multi-)hypergraph $H$ is said to be \emph{$r$-bounded} if every edge of $H$ has size at most $r$ and  \emph{$r$-uniform} if every edge has size exactly $r$; an \emph{$r$-uniform hypergraph} is called an \emph{$r$-graph} for short.

Let $F$ and $G$ be hypergraphs. An \emph{$F$ packing} of $G$ is a collection of copies of $F$ such that every edge of $G$ is in at most one element of $\mathcal{F}$.  An \emph{$F$ decomposition} of $G$ is an $F$ packing such that every edge of $G$ is in exactly one element of $\mathcal{F}$. For a (multi-)hypergraph $G$ and subset $S\subseteq V(G)$, $G(S)$ denotes the hypergraph with $V(G(S)):=V(G)\setminus S$ and $E(G(S))$ be the multiset $\{e\setminus S: e\in E(G), S\subseteq e\}$ (this is often called the \emph{link graph} of $S$). If $G$ is an $r$-graph and $|S|=r$, then $|G(S)|$ is called the \emph{multiplicity} of $S$. For $L \subseteq G(S)$, let $S \uplus L := \{S \cup e : e \in L\}$.

The complete $r$-graph on $q$ vertices, denoted $K_q^r$ is the $r$-graph with vertex set $[q]$ and edge set $\binom{[q]}{r}$. Note that a Steiner system with parameters $(n,q,r)$ is equivalent to a $K_q^r$ decomposition of $K_n^r$. A necessary condition for an $r$-graph $G$ to admit a $K_q^r$ decomposition is that $G$ is \emph{$K_q^r$-divisible}, that is, $\binom{q-i}{r-i}~|~|G(S)|$ for all $0\le i \le r-1$ and $S\subseteq V(G)$ with $|S|=i$. 

For our probabilistic arguments, we assume familiarity with the Chernoff bounds (e.g.~see Alon and Spencer~\cite{AS16}).

\section{Proof Overview}\label{s:ProofOverview}

In this section, we provide an overview of our proof of Theorem~\ref{thm:Existence}, with an eye towards discussing the main proof ideas and their novelty. First we provide a bird's-eye view of our new proof of the Existence Conjecture as follows.

Reformulating a $K_q^r$ decomposition problem in terms of a finding perfect matching of an auxiliary hypergraph is useful. To that end, we include the following definition.

\begin{definition}[Design Hypergraph]
Let $F$ be a hypergraph. If $G$ is a hypergraph, then the \emph{$F$-design hypergraph of $G$}, denoted ${\rm Design}_F(G)$, 
is the hypergraph $\mathcal{D}$ with $V(\mathcal{D})=E(G)$ and $E(\mathcal{D}) = \{S\subseteq E(G): S \text{ is isomorphic to } F\}$.
\end{definition}

\noindent
Here is a concise description of the proof steps of our new proof of the Existence Conjecture (where $G=K_n^r$ but later will be an $r$-uniform hypergraph of large minimum degree):
\begin{enumerate}
    \item[(1)] {\bf `Reserve' a random subset $X$ of $E(G)$ independently with some small probability $p$.}
    
    Namely, we use the following lemma (which we prove in Section~\ref{s:Existence} but follows easily from the Chernoff bounds):
    
    \begin{lem}\label{lem:RandomX}
    For all integers $q> r\ge 1$, there exists a real $\varepsilon \in (0,1)$ such that the following holds: Let $G$ be an $r$-uniform hypergraph with $\delta(G)\ge (1-\varepsilon)\cdot v(G)$ and let $p$ be a real number with $v(G)^{-\varepsilon} \le p \le 1$. If $v(G)\ge \frac{1}{\varepsilon}$, then there exists a spanning subhypergraph  $X\subseteq G$ such that $\Delta(X)\le 2p\cdot v(G)$ and for all $e\in G\setminus X$, there exist at least $\varepsilon\cdot p^{\binom{q}{r}-1}\cdot v(G)^{q-r}$ $K_q^r$'s in $X\cup \{e\}$ containing $e$.
    \end{lem}

    \item[(2)] {\bf Construct a refined omni-absorber $A$ of $X$}.
    
    Namely apply Theorem~\ref{thm:Omni} to $X$ and $G$ to find a refined omni-absorber $A$ for $X$ in $G$.
        
    \item[(3)] {\bf ``Regularity boost'' $G\setminus (X\cup A)$.}
    
    In particular, we apply the following special case of the ``Boost Lemma'' (Lemma 6.3 in~\cite{GKLO16}) of Glock, K\"{u}hn, Lo, and Osthus~\cite{GKLO16} to $J:= G\setminus (X\cup A)$.

\begin{lem}\label{lem:RegBoost}
For all integers $q> r \ge 1$, there exists a real $\varepsilon \in (0,1)$ and integer $C\ge 1$  such that the following holds: If $J$ is a $K_q^r$-divisible $r$-uniform hypergraph with $v(J)\ge C$ and $\delta(J) \ge (1-\varepsilon)\cdot v(J)$, then there exists a subhypergraph $H$ of ${\rm Design}_{K_q^r}(J)$ such that $d_H(v) = (1 \pm v(J)^{-(q-r)/3}) \cdot \frac{1}{2(q-r)!} \cdot v(J)^{q-r}$ for all $v\in V(H)$.
\end{lem}

    \item[(4)] {\bf Run ``nibble with reserves''} (i.e.~apply Theorem~\ref{thm:NibbleReserves} below) on $H$ to find a $K_q^r$ packing of $G\setminus A$ covering $G\setminus (X\cup A)$ where the \emph{reserves} are the $K_q^r$'s with one edge in $G-(X\cup A)$ and remaining edges in $X$; then for the resulting \emph{leftover} $L$ of $X$, we have that $L\cup A$ admits a $K_q^r$ decomposition by the definition of omni-absorber. 
\end{enumerate}

Together these steps complete our main result -  a new proof of the Existence Conjecture for Combinatorial Designs (Theorem~\ref{thm:Existence}); see Section~\ref{s:Existence} for a formal proof of Theorem~\ref{thm:Existence} that follows the outline above.

In the next subsections, we provide more details on how to build refiners and we also formally state the requisite ``nibble with reserves'' theorem.

\subsection{Building Refiners}

One key idea to proving Theorems~\ref{thm:Omni} and~\ref{thm:Refiner} is to prove them first in the $r=1$ case and then proceed by induction on the uniformity $r$. Note that a $K_q^1$-divisible $1$-uniform hypergraph always admits a $K_q^1$-decomposition (namely by partitioning the edges also known as `singletons' into sets of size $q$). Hence the empty graph is a $K_q^1$-omni-absorber for any $1$-uniform hypergraph $X$. However, the empty graph need not be a refined absorber.

Indeed, it is not hard to see that Theorem~\ref{thm:Omni} in the $r=1$ case is equivalent to the existence of a $q$-uniform hypergraph $H$ with $X\subseteq V(H)$, $\Delta(H)\le C$, and $v(H)\le C\cdot v(X)$ such that for every subset $L$ of $X$ with $q~|~|L|$, there exists a perfect matching of $L\cup (V(H)\setminus X)$. (Namely the equivalence follows by letting the edges of $H$ be the elements of the decomposition family $\mathcal{H}$ of $A$ and vice versa.) That is, Theorem~\ref{thm:Omni} in the $r=1$ case is equivalent to the existence of a \emph{robustly matchable $q$-uniform hypergraph} (RMH) with constant degree and only a constant proportion of new vertices. 

This is related to \emph{robustly matchable bipartite graphs} (RMBG) as constructed by Montgomery~\cite{M19b} except that the bipartite condition is much more restrictive and hence harder to construct. Thus our task is easier; this being said, we will require some additional properties of our RMH for the purposes of induction as follows.

Here is the key proposition.

\begin{proposition}[RMH Construction]\label{prop:RMHG}
Let $q\ge 1$ be an integer. If $X$ is a set of vertices, then there exists a $q$-uniform hypergraph $H$ and $X'\subseteq R:=V(H)\setminus X$ such that all of the following hold:
\begin{enumerate}
    \item[(1)] $X\subseteq V(H)$ and $|e\cap X|\le 1$ for all $e\in E(H)$,
    \item[(2)] $v(H)=(q+1)\cdot v(X)$,
    \item[(3)] $\Delta(H)\le 2q$ and hence $e(H)\le 2(q+1)\cdot v(X)$,
    \item[(4)] $|X'|=q-1$ and for every $L\subseteq X$, there exists a matching $M_L$ of $H$ such that $$(L\cup R)\setminus X'\subseteq V(M_L)\subseteq L\cup R.$$
We note that this implies that if $q\mid|L|$, then $M_L$ is a perfect matching of $L\cup R$.
\end{enumerate}
\end{proposition}
\begin{proof}
Let $m:= |X|$. Let $x_1,\ldots x_m$ be an enumeration of $X$. Let $y_1,\ldots, y_{(q+1)m}$ be an enumeration of the vertices of $H$ where we assume that $y_{(q+1)i}=x_i$ for all $i\in [m]$.

Now we define the edges of $H$ as
\begin{align*}
E(H):= &\Big\{ \big\{y_{i+j}:j\in \{0,1,\ldots, q-1\}\big\}: i\in [(q+1)m- (q-1)]\Big\}~~\cup\\
&\Big\{ \big\{y_{i+j}:j\in \{0,1,\ldots, q\},~(q+1)\nmid (i+j)\big\}: i\in [(q+1)m-q],~(q+1)\nmid i\Big\}.    
\end{align*}
Note that $H$ is a $q$-uniform hypergraph, $X\subseteq V(H)$ and $|e\cap X|\le 1$ for all $e\in E(H)$ and hence (1) holds. Note that $v(H)=(q+1)\cdot v(X)$ and hence (2) holds. Next note that $\Delta(H)\le 2q$ and hence since $H$ is $q$-uniform, we find that 
$$e(H) \le \frac{\Delta(H)\cdot v(H)}{q} \le \frac{2q\cdot ((q+1)\cdot v(X))}{q} = 2(q+1)\cdot v(X),$$ 
and hence (3) holds.

Let $$X':= \{y_{i}: i\in [q-1]\}.$$ 
Note that $|X'|=q-1$.

Fix $L\subseteq X$. Let   $s_1 < s_2 <\ldots < s_{qm+|L|}$ be integers such that $L\cup R=\{ y_{s_k}:k\in [qm+|L|]\}$. Then we define
$$M_L := \left\{ \big\{ y_{s_{|L\cup R|-q\cdot i-j}}: j\in \{0,\ldots, q-1\} \big\}: i\in \mathbb{Z},~0\le i \le \frac{|L\cup R|}{q} -1 \right\}.$$
Note that $M_L$ is a matching of $H$ such that $X'\cup V(M_L) = L\cup R$. Hence (4) holds.
\end{proof}

As mentioned before, we have the following immediate corollary of Proposition~\ref{prop:RMHG}.

\begin{cor}\label{cor:1Uuniform}
Theorem~\ref{thm:Omni} and Theorem~\ref{thm:Refiner} hold for $r=1$ with $C=2q$.
\end{cor}

To build our refined absorbers we employ the hypergraph in Proposition~\ref{prop:RMHG} in two cases: inductively to `link hypergraphs' as the $r=1$ case of these theorems, but also to sets of parallel hyperedges so as to refine them down to sets of bounded multiplicity.

The other key idea is to allow a refiner to only partially refine $X$ by permitting an `unrefined' \emph{remainder} $X'$ of $X\cup R$ (see Definition~\ref{def:Remainder}). We then prove the existence of refiners by gradually restricting the remainder to smaller and smaller vertex sets (similar to how the Cover Down Lemma is used to restrict uncovered edges; we call this restriction procedure the Refine Down Lemma). 

As mentioned in the outline of the paper, we actually prove the existence first of refiner and refined absorbers in the multi-hypergraph setting (which is easier since it permits multi-edges). Only afterwards do we build refiners by replacing the multi-edges by embedding them as gadgets that we call `fake edges' which have the same divisibility properties of an edge. Finally, to construct refined absorbers, we embed private absorbers for each element of the refinement family.

\subsection{Nibble with Reserves}

To finish the proof of Theorem~\ref{thm:Existence}, one uses ``nibble with reserves'', but first some  definitions.

\begin{definition}[$A$-perfect matching]
Let $G$ be a hypergraph and $A\subseteq V(G)$. We say that a matching of $G$ is \emph{$A$-perfect} if every vertex of $A$ is in an edge of the matching.
\end{definition}

\begin{definition}[Bipartite Hypergraph]
We say a hypergraph $G=(A,B)$ is \emph{bipartite with parts $A$ and $B$} if $V(G)=A\cup B$ and  every edge of $G$ contains exactly one vertex from $A$. 
\end{definition}

We are now ready to formally state ``nibble with reserves''. We note this version of nibble is stronger than the standard almost perfect matching version, the more general coloring and list coloring versions, and even our bipartite hypergraph formulation (also known as rainbow matching) from~\cite{DP22}.

\begin{thm}[Nibble with Reserves~\cite{DP22}]\label{thm:NibbleReserves}
For each integer $r \ge 2$ and real $\beta \in (0,1)$, there exist an integer $D_{\beta}\ge 0$ and real $\alpha \in (0,1)$ such that following holds for all $D\ge D_{\beta}$: 
\vskip.05in
Let $G$ be an $r$-uniform (multi)-hypergraph with codegrees at most $D^{1-\beta}$ such that $G$ is the edge-disjoint union of $G_1$ and $G_2$ where $G_2=(A,B)$ is a bipartite hypergraph such that every vertex of $B$ has degree at most $D$ in $G_2$ and every vertex of $A$ has degree at least $D^{1-\alpha}$ in $G_2$, and $G_1$ is a hypergraph with $V(G_1)\cap V(G_2) = A$ such that every vertex of $G_1$ has degree at most $D$ in $G_1$ and every vertex of $A$ has degree at least $D\left(1- D^{-\beta}\right)$ in $G_1$.

Then there exists an $A$-perfect matching of $G$.
\end{thm}

We note that this theorem has not appeared in the literature in such a general form (but mostly follows along the lines of Kahn's proof~\cite{K96} that the List Coloring Conjecture holds asymptotically for hypergraphs which Theorem~\ref{thm:NibbleReserves} also implies). Moreover, it is a special case of our new more general theorem about ``forbidden submatchings with reserves'' in~\cite{DP22}. We refer the reader to~\cite{DP22} for further history and proof details.

We also note that the `reserves' portion of the proof of the above theorem (that is the use of $G_2$ to extend an almost $A$-perfect matching to an $A$-perfect one) may be viewed for our proof of the Existence Conjecture as one Cover Down step (though in an edge vortex). Hence the proof of Glock, K\"{u}hn, Lo, and Osthus~\cite{GKLO16} of the Existence Conjecture may be viewed as {\bf Iterated (Vertex) Cover Down} while our new proof may be viewed as {\bf One Step of (Edge) Cover Down+Iterated Refine Down}.

\section{Multi-Refiners}\label{s:Refined}

In order to prove Theorems~\ref{thm:Omni} and~\ref{thm:Refiner}, it will be convenient to first prove analogous theorems in the multigraph setting. 

Here are the definitions for the multigraph versions of omni-absorber and refiner.

\begin{definition}[Multi-Omni-Absorber]
We say a multi-hypergraph $A$ is a \emph{$K_q^r$-multi-omni-absorber} if it meets the requirements of the definition of $K_q^r$-omni-absorber except perhaps that $X\cup A$ is not simple.
\end{definition}

\begin{definition}[Multi-Refiner]
We say a multi-hypergraph $R$ is a \emph{$K_q^r$-multi-refiner} of $X$ if it meets the requirements of the definition of $K_q^r$-refiner except perhaps that $X\cup R$ is not simple.
\end{definition}

Thus to prove our Refiner Theorem, Theorem~\ref{thm:Refiner}, we first build a $K_q^r$-multi-refiner with the desired properties and then replace each edge with a gadget we refer to as a `fake edge' which has the same divisibility properties as an edge. This embedding is done randomly so as to make the fake edges disjoint and such that the result has low maximum degree. We then derive our Refined Omni-Absorber Theorem, Theorem~\ref{thm:Omni}, from Theorem~\ref{thm:Refiner} by embedding private absorbers for each member of the refinement family (again so that the absorbers are edge-disjoint and the result has low maximum degree).

\begin{remark}
Proving the multigraph versions of these theorems is not strictly necessary as one could follow the multigraph proof and then embed fake-edges/absorbers as needed throughout the proof to keep the result simple. However, proving the multigraph version is more convenient as it then permits all those embedding steps to be performed afterwards in one single step.    
\end{remark}

\begin{remark}
We note for the reader that multigraph versions of these theorems - that is, where both $X$ and $A$ (or $R$) are allowed to be multi-hypergraphs - do not imply the original versions or vice versa. This is because allowing $A$ to be a multi-hypergraph makes the problem easier while allowing $X$ to be a multi-hypergraph makes the problem harder.
\end{remark}

It turns out we only need to prove multigraph versions where $X$ has bounded multiplicity. In particular here is the key number.

\begin{definition} For all integers $q>r\ge 1$, 
define 
$$M(q,r):= {\rm lcm}\left( \binom{q-i}{r-i}: i\in \{0,1,\ldots, r-1\}\right).$$    
\end{definition}

The key fact is that $M(q,r)$ edges on the same set of $r$ vertices is a $K_q^r$-divisible graph. Thus we will be able to apply our RMH construction from Proposition~\ref{prop:RMHG} to an arbitrary number of edges on the same set of $r$ vertices so as to reduce the multiplicity. See our Multiplicity Reduction Lemma~\ref{lem:MRL} in Section~\ref{ss:MRL} below.

For inductive purposes, we need to prove the multigraph theorems with a stronger property. Namely, in our proof of one step of the Refine Down Lemma (in particular in the proof of Claim~\ref{claim:Z} in Lemma~\ref{lem:RefineDownOneStep}) when we randomly embed omni-absorbers of lower uniformity, it is necessary to have a certain `flatness' property in order to have the desired expectations. This motivates the following definition.

\begin{definition}
Let $X$ be a hypergraph and $Y\subseteq V(X)$. For a nonnegative integer $i$, we say $X$ is \emph{$i$-flat to $Y$} if $|e\setminus Y|\le i$ for all $e\in E(X)$.
\end{definition}

\noindent Here is our multi-omni-absorber version of Theorem~\ref{thm:Omni}.

\begin{thm}[Refined Multi-Omni-Absorber Theorem]\label{thm:MultiOmni}
For all integers $q > r\ge 1$, there exists an integer $C\ge 1$ such that the following holds: Let $X$ be an $r$-uniform multi-hypergraph of multiplicity at most $M(q,r)$ with $v(X)\ge C$ and let $\Delta:= \max\left\{\Delta(X),~ v(X)^{1-\frac{1}{r}}\cdot \log v(X)\right\}$. If $Y\subseteq V(X)$ with $|Y| \ge v(X)/3$, then there exists a $C$-refined $K_q^r$-multi-omni-absorber $A$ for $X$ such that both of the following hold:
\begin{itemize}
    \item[(1)] $\Delta(A)\le C \cdot \Delta$, and
    \item[(2)] $A$ is $(r-1)$-flat to $Y$. 
\end{itemize}
\end{thm}

\noindent
Here is our multi-refiner version of Theorem~\ref{thm:Refiner}.

\begin{thm}[Multi-Refiner Theorem]\label{thm:MRT}
For all integers $q > r\ge 1$, there exists an integer $C \ge 1$ such that the following holds: Suppose Theorem~\ref{thm:MultiOmni} holds for all $r'\in [r-1]$. Let $X$ be an $r$-uniform multi-hypergraph of multiplicity at most $M(q,r)$ with $v(X)\ge C$ and let $\Delta:= \max\left\{\Delta(X),~v(X)^{1-\frac{1}{r}}\cdot \log v(X)\right\}$. If $Y\subseteq V(X)$ with $|Y| \ge v(X)/3$, then there exists a $C$-refined $K_q^r$-multi-refiner $R$ of $X$ with refinement family $\mathcal{H}$ such that both of the following hold:
\begin{enumerate}
    \item[(1)] $\Delta(\mathcal{H}) \le C \cdot \Delta$, and
    \item[(2)] $R$ is $(r-1)$-flat to $Y$. 
\end{enumerate}      
\end{thm}

\begin{remark}
Again note that Theorem~\ref{thm:MultiOmni} is Theorem~\ref{thm:MRT} with the additional assumption that every element of $H$ is a $K_q^r$. Strictly speaking it is not technically necessary to prove Theorem~\ref{thm:MultiOmni} as Theorem~\ref{thm:Refiner} (and hence Theorem~\ref{thm:Omni}) will be derived from Theorem~\ref{thm:MRT}. However, the proof of Theorem~\ref{thm:MRT} requires induction on the uniformity for the Refine Down steps and it is convenient to assume that the lower uniformity object is an omni-absorber instead of a refiner (in particular, we do this in the proof of Theorem~\ref{thm:MLRT} in Section~\ref{s:RefineDown}). It is possible (via an object we developed that we call a \emph{divisibility fixer}) to work around this assumption in the one place we used it, but it seemed simpler just to prove the multi-omni-absorber theorem instead since that only requires embedding private absorbers, for which we already developed machinery for the purpose of deriving Theorem~\ref{thm:Omni}.    
\end{remark}

In this section, we prove Theorem~\ref{thm:MRT} modulo the proof of the Refine Down Lemma (Lemma~\ref{lem:RefineDown}) which we provide in Section~\ref{s:RefineDown}. 

First we note that the $r=1$ cases of Theorems~\ref{thm:MultiOmni} and~\ref{thm:MRT} follow easily from our RMH construction of a hypergraph $H$ in Proposition~\ref{prop:RMHG}: namely by splitting the multiple copies of `singletons' of $X$ into separate vertices, applying the construction, and then identifying the copies back into single vertices. In particular, we embed the singletons of $A$ into $Y$ in such a way that the vertices of each edge are all distinct. Note we may assume that $C\ge 6q^2$ and hence $|Y|\ge v(X)/3 \ge 2q^2$. Then this embedding is possible via greedily embedding the singletons of $A$ so as to avoid  singletons (either in $X$ or from $A$ that were previously embedded) that are in edges of $H$ with the current singleton being embedded. This is possible since there are at most $2q^2-1$ such singletons as $H$ is $q$-uniform and $\Delta(H)\le 2q$.

\subsection{Multi-Refiners with Remainder}

For inductive purposes, it is useful to allow a refiner to only partially refine by permitting a `remainder' as follows.

\begin{definition}[Multi-Refiner with Remainder]\label{def:Remainder}
Let $q > r\ge 1$ be integers. Let $X$ be an $r$-uniform multi-hypergraph. We say an $r$-uniform multi-hypergraph $R$ is a \emph{$K_q^r$-multi-refiner} of $X$ with \emph{remainder} $X' \subseteq X\cup R$, \emph{refinement family} $\mathcal{H}$ and \emph{refinement function} $\mathcal{Q}_R$, if $X$ and $R$ are edge-disjoint, $R$ is $K_q^r$-divisible, and $\mathcal{H}$ is a family of $K_q^r$-divisible subgraphs of $X\cup R$ such that $|H\cap X|\le 1$ and $H\cap X\cap X'=\emptyset$ for all $H\in\mathcal{H}$, and for every $K_q^r$-divisible subgraph $L$ of $X$, $\mathcal{Q}_R(L) \subseteq \mathcal{H}$ such that $(L\cup R)\setminus X' \subseteq \bigcup \mathcal{Q}_R(L) \subseteq L\cup R$. We call $L':= (L\cup R) \setminus \bigcup \mathcal{Q}_R(L)$ the \emph{$L$-leftover} of the refinement.

Let $C\ge 1$ be a real.  We say $R$ is \emph{$C$-refined} if $\mathcal{H}$ is $C$-refined.

Let $S\in \binom{V(R)}{r-1}$. The \emph{refinement degree} of $S$ in $\mathcal{H}$, denoted $\Delta_{\mathcal{H}}(S)$, is $|\{H\in \mathcal{H}: S\subseteq V(H)\}$. The \emph{refinement degree} of $\mathcal{H}$, denoted $\Delta(\mathcal{H})$, is $\max \left\{ \Delta_{\mathcal{H}}(S): S\in \binom{V(R)}{r-1} \right\}$.
\end{definition}

Note that when $X'=\emptyset$, this reduces to the definition of $K_q^r$-multi-refiner. We remark that $L':= (L\cup R) \setminus \bigcup \mathcal{Q}_R(L)$ (the $L$-leftover) is also $K_q^r$-divisible since $L$ is by assumption, $R$ is by definition, $\bigcup \mathcal{Q}_R(L)$ is by construction and that $\bigcup \mathcal{Q}_R(L) \subseteq L\cup R$. 
Also note that $L'\subseteq X'$.

Here is a useful proposition which shows that refiners with remainder concatenate as follows.

\begin{proposition}[Concatenating Refiners]\label{prop:Concatenate}
Let $q > r \ge 1$ be integers. Let $X_1, X_2, X_3$ be $r$-uniform hypergraphs. For each $i\in \{1,2\}$, suppose that $R_i$ is a $C_i$-refined $K_q^r$-multi-refiner of $X_i$ with remainder $X_{i+1}$,  refinement family $\mathcal{H}_i$ and refinement function $\mathcal{Q}_{R_i}$. If $R_1$ and $R_2$ are edge-disjoint, then $R:= R_1\cup R_2$ is a $(C_1+C_2)$-refined $K_q^r$-multi-refiner of $X_1$ with remainder $X_3$, refinement family $\mathcal{H}:= \mathcal{H}_1\cup \mathcal{H}_2$, refinement function $$\mathcal{Q}_R(L):= \mathcal{Q}_{R_1}(L) \cup \mathcal{Q}_{R_2}\left((L\cup R_1)\setminus \bigcup \mathcal{Q}_{R_1}(L)\right)$$ 
and such that $\Delta(\mathcal{H}) \le \Delta(\mathcal{H}_1) + \Delta(\mathcal{H}_2)$.     
\end{proposition}
\begin{proof}
Since $R_1$ and $R_2$ are $K_q^r$-divisible by definition, we find that $R$ is $K_q^r$-divisible. Clearly $R$ is $(C_1+C_2)$-refined and $\Delta(\mathcal{H}) \le \Delta(\mathcal{H}_1) + \Delta(\mathcal{H}_2)$.     

By definition of multi-refiner, we have that $|H_1\cap X_1|\le 1$ for all $H_1\in \mathcal{H}_1$ and that $|H_2\cap X_2|\le 1$ for all $H_2\in \mathcal{H}_2$. Hence it follows that $|H\cap X_1|\le 1$ for all $H\in \mathcal{H}$. Similarly $H_1\cap X_1\cap X_2 =\emptyset$ for all $H_1\in \mathcal{H}_1$ and $H_2\cap X_2\cap X_3=\emptyset$ for all $H_2\in \mathcal{H}_2$. Since $X_1\cap X_3\subseteq X_2$, we find that $X_1\cap X_3=X_1\cap X_2\cap X_3$. It then follows that $H\cap X_1\cap X_3=\emptyset$ for all $H\in \mathcal{H}$.

Now let $L$ be a $K_q^r$-divisible subgraph of $X_1$. By definition, $(L\cup R_1)\setminus X_2 \subseteq \bigcup \mathcal{Q}_{R_1}(L)\subseteq L\cup R_1$. Let $L_1:= (L\cup R_1)\setminus \bigcup \mathcal{Q}_{R_1}(L)$. As remarked above, it follows from the definition that $L_1$ is also $K_q^r$-divisible and that $L_1\subseteq X_2$. Hence by definition, $(L_1\cup R_2)\setminus X_3 \subseteq \bigcup \mathcal{Q}_{R_2}(L_1)\subseteq L_1\cup R_2$. 

Note that $\mathcal{Q}_R(L) = \mathcal{Q}_{R_1}(L)\cup \mathcal{Q}_{R_2}(L_1)$ by definition. First we have that $\bigcup \mathcal{Q}_R(L) \subseteq (L\cup R_1) \cup (L_1\cup R_2)$. Since $L_1\subseteq L\cup R_1$, it follows that $\bigcup \mathcal{Q}_R(L) \subseteq L\cup R_1 \cup R_2 = L\cup R$ as desired. Second we have that $\bigcup \mathcal{Q}_R(L) \supseteq  \bigcup \mathcal{Q}_{R_1}(L) \cup ((L_1\cup R_2)\setminus X_3)$. Since $L_1= (L\cup R_1)\setminus \bigcup \mathcal{Q}_{R_1}(L)$, it follows that $\bigcup \mathcal{Q}_R(L) \supseteq (L\cup R_1 \cup R_2)\setminus X_3 = (L\cup R)\setminus X_3$ as desired.

Finally, since $R_1$ and $R_2$ are edge-disjoint, it follows that $(\bigcup \mathcal{Q}_{R_1}(L))\cap (\bigcup \mathcal{Q}_{R_2}(L_1)) = \emptyset$ as desired. 
\end{proof}

\subsection{Proof Overview of Multi-Refiner Theorem}

Thus Theorem~\ref{thm:MRT} will be proved via repeated applications of the following two lemmas until the remainder is constant sized (in which case it is easy to build a refiner of constant refinement).

First we have the Edge Sparsification Lemma. 

\begin{lem}[Edge Sparsification Lemma]\label{lem:Sparsify}
For all integers $q > r\ge 2$, there exists an integer $C \ge 1$ such that the following holds: Let $X$ be an $r$-uniform multi-hypergraph of multiplicity at most $M(q,r)$ with $v(X)\ge C$ and suppose $\Delta(X)\ge v(X)^{1-\frac{1}{r}}\cdot \log v(X)$. If $Y\subseteq V(X)$ with $|Y|\ge v(X)/3$, then there exists a $C$-refined $K_q^r$-multi-refiner $R$ of $X$ with remainder $X'\subseteq R$ and refinement family $\mathcal{H}$ such that all of the following hold:

\begin{enumerate}
    \item[(1)] $\Delta(\mathcal{H}) \le C \cdot  \Delta(X)$,
    \item[(2)] $R$ is $(r-1)$-flat to $Y$,
    \item[(3)] $\Delta(X')\le v(X)^{1-\frac{1}{r}}$, and
    \item[(4)] $X'$ is an $r$-uniform multi-hypergraph with multiplicity at most $M(q,r)$.
\end{enumerate}      
\end{lem}

The above lemma sparsifies the maximum degree of the remainder provided the original maximum degree is dense enough. The key idea of its proof is to have each edge receive its own private clique in the refiner but crucially to have many edges each choose the same set of vertices for said clique - this `focuses' the remainder into a small set of edges, whence we apply the Multiplicity Reduction Lemma (Lemma~\ref{lem:MRL}) to obtain a remainder of small maximum degree. The Edge Sparsification Lemma is proved in Section~\ref{ss:EdgeSparsification}.

Second we have the Refine Down Lemma. 

\begin{lem}[Refine Down Lemma]\label{lem:RefineDown}

For all integers $q > r\ge 2$, there exists an integer $C \ge 1$ such that the following holds: Suppose Theorem~\ref{thm:MultiOmni} holds for all $r'\in [r-1]$.  Let $X$ be an $r$-uniform multi-hypergraph of multiplicity at most $M(q,r)$ and let $\Delta:= \max \{\Delta(X),~ v(X)^{1-\frac{1}{r}}\cdot \log v(X)\}$. If $Y\subseteq V(X)$ with $|Y|\ge C$, then there exists a $C$-refined $K_q^r$-multi-refiner $R$ of $X$ with remainder $X'$ and refinement family $\mathcal{H}$ such that all of the following hold:

\begin{enumerate}
    \item[(1)] $\Delta(\mathcal{H}) \le C \cdot  \left(\frac{v(X)}{|Y|}\right)^{\binom{r}{2}} \cdot \Delta$,
    \item[(2)] $R$ is $(r-1)$-flat to $Y$,
    \item[(3)] $\Delta(X')\le C \cdot  \left(\frac{v(X)}{|Y|}\right)^{\binom{r}{2}} \cdot \Delta$, and
    \item[(4)] $X'$ is an $r$-uniform multi-hypergraph with multiplicity at most $M(q,r)$  and $V(X') \subseteq Y$.
\end{enumerate}      
\end{lem}

The above lemma sparsifies the maximum degree of the remainder by restricting the remainder to a subset $Y$ of the vertices (and applying the Multiplicity Reduction Lemma). This comes at the cost of increasing the maximum degree of the refiner by some power of $\frac{v(X)}{|Y|}$. This lemma is proved in Section~\ref{s:RefineDown}.

We note that the statement (and proof) of the Refine Down Lemma have similarities to the Cover Down Lemma from~\cite{GKLO16}. Here we are \emph{refining} the edges outside of $Y$ via means of a small well-chosen subgraph (the refiner) as opposed to \emph{covering} all edges outside of $Y$ as in the Cover Down Lemma. This is a structural approach which builds an extremely efficient refined absorber in iterative layers.

The two lemmas above combine to make the following Better Refine Down Lemma. Namely, we first apply the Edge Sparsification Lemma and then the Refine Down Lemma - crucially the decrease in maximum degree from the first offsets the increase in maximum degree from the second.

\begin{lem}[Better Refine Down Lemma]\label{lem:BetterRefineDown}

For all integers $q > r\ge 2$, there exists an integer $C \ge 1$ such that the following holds: Suppose Theorem~\ref{thm:MultiOmni} holds for all $r'\in [r-1]$.  Let $X$ be an $r$-uniform multi-hypergraph of multiplicity at most $M(q,r)$ with $v(X)\ge C$ and let $\Delta:= \max \{\Delta(X),~C\cdot v(X)^{1-\frac{1}{r}}\cdot \log v(X)\}$. If $Y\subseteq V(X)$ with $|Y|\ge v(X)/3$, then there exists a $C$-refined $K_q^r$-multi-refiner $R$ of $X$ with remainder $X'$ and refinement family $\mathcal{H}$ such that all of the following hold:

\begin{enumerate}
    \item[(1)] $\Delta(\mathcal{H}) \le C \cdot  \Delta$,
    \item[(2)] $R$ is $(r-1)$-flat to $Y$,
    \item[(3)] $\Delta(X')\le \frac{\Delta}{2}$, and
    \item[(4)] $X'$ is an $r$-uniform multi-hypergraph with multiplicity at most $M(q,r)$  and $V(X') \subseteq Y$.
\end{enumerate}      
\end{lem}
\begin{proof}
Let $C_1$ be as $C$ in Lemma~\ref{lem:Sparsify} and let $C_2$ be as $C$ in Lemma~\ref{lem:RefineDown}. Let $\Delta_1 := \max \{\Delta(X),~v(X)^{1-\frac{1}{r}}\cdot \log v(X)\}$. 

First suppose that $\Delta(X) < \Delta_1$. Then by the Refine Down Lemma (Lemma~\ref{lem:RefineDown}) applied to $X$ and $Y$, there exists a $C_2$-refined $K_q^r$-multi-refiner $R$ of $X$ with remainder $X'$ and refinement family $\mathcal{H}$ satisfying Lemma~\ref{lem:RefineDown}(1)-(4). Since $|Y|\ge v(X)/3$, we find that 
$$\Delta(\mathcal{H}) \le C_2 \cdot  \left(\frac{v(X)}{|Y|}\right)^{\binom{r}{2}} \cdot \Delta_1 \le C_2 \cdot 3^{\binom{r}{2}} \cdot \Delta_1 \le \frac{\Delta}{2} \le C\cdot \Delta$$
since $\Delta \ge 2\cdot C_2\cdot 3^{\binom{r}{2}}\cdot \Delta_1$ since $C\ge 2\cdot 3^{\binom{r}{2}}\cdot C_2$ since $C$ is large enough. Hence (1) holds for $\mathcal{H}$. Similarly we find that $\Delta(X')\le \frac{\Delta}{2}$. Hence (3) holds for $X'$. Then all of (1)-(4) hold for $R,\mathcal{H},$ and $X'$ as desired.

So we assume that $\Delta_1 = \Delta(X)$. Hence $\Delta(X)\ge v(X)^{1-\frac{1}{r}}\cdot \log v(X)$. Since $C\ge C_1$, we have by the Edge Sparsification Lemma (Lemma~\ref{lem:Sparsify}) that there exists a $C_1$-refined $K_q^r$-multi-refiner $R_1$ of $X$ with remainder $X_2\subseteq R_1$ and refinement family $\mathcal{H}_1$ satisfying Lemma~\ref{lem:Sparsify}(1)-(4). By Lemma~\ref{lem:Sparsify}(3), we have that $\Delta(X_2)\le v(X)^{1-\frac{1}{r}}$. 

Let $\Delta_2:= \max \{\Delta(X_2),~v(X)^{1-\frac{1}{r}}\cdot \log v(X)\}$. Since $C_2\ge 1$ and $v(X)\ge C \ge e$, we find that $\Delta_2 > \Delta(X_2)$. Hence $\Delta_2 = v(X)^{1-\frac{1}{r}}\cdot \log v(X)$. By the Refine Down Lemma (Lemma~\ref{lem:RefineDown}) applied to $X_2$ and $Y$, there exists a $C_2$-refined $K_q^r$-multi-refiner $R_2$ of $X_2$ with remainder $X_3$ and refinement family $\mathcal{H_2}$ satisfying Lemma~\ref{lem:RefineDown}(1)-(4).

Let $R:= R_1\cup R_2$ and $\mathcal{H}:=\mathcal{H}_1\cup \mathcal{H}_2$. By Proposition~\ref{prop:Concatenate}, $R$ is a $(C_1+C_2)$-refined $K_q^r$-multi-refiner of $X$ with remainder $X_3$ and refinement family $\mathcal{H}$. Since $C\ge C_1+C_2$, we have that $R$ is $C$-refined. By Lemma~\ref{lem:Sparsify}(1), we have that 
$$\Delta(\mathcal{H}_1)\le C_1\cdot \Delta(X) \le \frac{C}{2}\cdot \Delta,$$
where we used that $\Delta \ge \Delta(X)$. By Lemma~\ref{lem:RefineDown}(1), we have that
$$\Delta(\mathcal{H}_2) \le C_2 \cdot  \left(\frac{v(X)}{|Y|}\right)^{\binom{r}{2}} \cdot \Delta_2 \le C_2 \cdot 3^{\binom{r}{2}} \cdot \Delta_2 \le 3^{\binom{r}{2}} \cdot C_2^2 \cdot v(X)^{1-\frac{1}{r}}\cdot \log v(X) \le \frac{\Delta}{2} \le \frac{C}{2}\cdot \Delta,$$
where we used that $|Y|\ge v(X)/3$, $C\ge 2\cdot 3^{\binom{r}{2}} \cdot C_2^2$ and $\Delta \ge C \cdot v(X)^{1-\frac{1}{r}}\cdot \log v(X)$. Combining, we find that
$$\Delta(\mathcal{H})\le \Delta(\mathcal{H}_1)+\Delta(\mathcal{H}_2) \le \frac{C}{2}\cdot \Delta + \frac{C}{2}\cdot \Delta \le C\cdot \Delta$$
and hence (1) holds for $\mathcal{H}$. Similarly we find that 
$$\Delta(X_3) \le C_2 \cdot  \left(\frac{v(X)}{|Y|}\right)^{\binom{r}{2}} \cdot \Delta_2 \le \frac{\Delta}{2},$$
and hence (3) holds for $X_3$.

By Lemma~\ref{lem:Sparsify}(2), we have that $R_1$ is $(r-1)$-flat to $Y$. By Lemma~\ref{lem:RefineDown}(2), we have that $R_2$ is $(r-1)$-flat to $Y$. Hence $R$ is $(r-1)$-flat to $Y$ and (2) holds for $R$. Finally, by Lemma~\ref{lem:RefineDown}(4), $X_3$ is an $r$-uniform multi-hypergraph with multiplicity at most $M(q,r)$ and $V(X_3)\subseteq Y$. Thus all of (1)-(4) hold for $R$, $\mathcal{H}$ and $X_3$ as desired. 
\end{proof}

\begin{remark}
We remark that strictly speaking the Edge Sparsification Lemma is not necessary to the proof of the Existence Conjecture; indeed, one may simply use the Refine Down Lemma (albeit with worse bounds by induction) to prove versions of our main refiner/omni-absorber theorems with worse bounds, say $\Delta(A)= O\left( \left(\frac{\Delta(X)}{v(X)}\right)^{1/r^2} \cdot v(X)\right)$, which would still suffice for our proof of the Existence Conjecture.  However, since the Edge Sparsification Lemma only requires a couple of pages to prove and yields essentially optimal omni-absorbers, we opted to include it with an eye towards potential future applications.    
\end{remark}

We now are prepared to prove the Multi-Refiner Theorem (Theorem~\ref{thm:MRT}) assuming the Better Refine Down Lemma as follows. Namely, we simply iterate the Better Refine Down Lemma until the remainder has constant size; however to ensure the induction works we require a slightly stronger inductive assumption on the refinement of edges of $X$ so as to ensure that overall the refiner is $C$-refined.

\begin{lateproof}{thm:MRT}
Let $C_0:=\max\left\{6\cdot C',~2\cdot M(q,r)\cdot (C')^{r},~2^{M(q,r)\cdot (C')^{r}} + 1\right\}$ where $C'$ is as in Lemma~\ref{lem:BetterRefineDown} for $q$ and $r$. Let $\Delta_0 := \max \{\Delta(X),~C'\cdot v(X)^{1-\frac{1}{r}}\cdot \log v(X)\}$.

We prove the stronger theorem that if $v(X)\ge C'$ and the other hypotheses of Theorem~\ref{thm:MRT} hold, then there exists a $C_0$-refined $K_q^r$-multi-refiner $R$ of $X$ with refinement family $\mathcal{H}$ such that all of the following hold:
\begin{itemize}
    \item[(1)] $\Delta(\mathcal{H})\le C_0\cdot \Delta_0$,
    \item[(2)] $R$ is $(r-1)$-flat to $Y$.
    \item[(3)] $|\{H\in \mathcal{H}: e\in H\}|\le \frac{C_0}{2}$ for all $e\in E(X)$.
\end{itemize}
Note this then still implies the theorem by letting $C=C_0\cdot C'$ since $\Delta_0\le C'\cdot \Delta$ and $C'\ge 1$.

We proceed by induction on $v(X)$. Note we assume that $\Delta(X)\ge 1$ as otherwise there is nothing to show (i.e.~$X$ is empty and hence the empty refiner satisfies (1)-(3) trivially). This implies that $\Delta_0\ge 1$.

Let $Y' \subseteq Y \subseteq V(X)$ with $|Y'| = \lceil v(X)/3 \rceil$. Note that $|Y'|\le v(X)/2$ since $v(X)\ge C_0 \ge 6$. 

By Lemma~\ref{lem:BetterRefineDown}, there exists a $C'$-refined $K_q^r$-multi-refiner $R_1$ of $X$ with remainder $X'$ and refinement family $\mathcal{H}_1$ such that all of Lemma~\ref{lem:BetterRefineDown}(1)-(4) hold.

By Lemma~\ref{lem:BetterRefineDown}(4) then, $X'$ is an $r$-uniform multi-hypergraph with multiplicity at most $M(q,r)$  and $V(X') \subseteq Y'$. We assume without loss of generality that $V(X')=Y'$. Let $\Delta':= \max\{\Delta(X'),~C'\cdot v(X')^{1-\frac{1}{r}}\cdot \log v(X')\}$. 

\begin{claim}
There exists a $C_0$-refined $K_q^r$-multi-refiner $R_2$ of $X'$ with refinement family $\mathcal{H}_2$ such that all of (1)-(3) hold for $R_2$, $\mathcal{H}_2$, $X'$, $Y'$ and $\Delta'$.
\end{claim}
\begin{proofclaim}
We assume that $\Delta(X') \ge 1$ as otherwise there is nothing to show (since then $X'$ is empty and the empty refiner satisfies (1)-(3) trivially). 

First suppose that $v(X')\ge C'$. Since $v(X)=|Y'| \le v(X)/2 < v(X)$, the claim follows by the induction of Theorem~\ref{thm:MRT} applied to $X'$ and $Y'$.

So we assume $v(X') < C'$. Let $R_2$ be the $r$-uniform multi-hypergraph edge-disjoint from $X'$ with $V(X)=V(R_2)$ and where $|R_2(S)|= M(q,r)\cdot |X'(S)|$ for all $S\in \binom{V(X')}{r}$. For each $e\in X'$, we choose $R_e$ to be a set of $M(q,r)$ edges of $R_2$ where $V(f)=V(e)$ for all $f\in R_e$ and in such a way that $R_e\cap R_{e'}=\emptyset$ for all $e\ne e'\in X$. For each $e\in X'$, choose one element of $R_e$ and denote it by $e^*$ and let $R'_e := R_e\setminus \{e^*\}$.

Now we define the refinement family 
$$\mathcal{H}_2 := \{R_e:e\in X'\}~\cup~\{e\cup R'_e: e\in X'\}~\cup~\left\{ \bigcup_{e\in L} e^*: L \text{ is a $K_q^r$-divisible subgraph of $X'$ }\right\}.$$
Moreover by construction, we have $|H\cap X'|\le 1$ for all $H\in \mathcal{H}_2$. Thus, we have that $R_2$ is a $K_q^r$-multi-refiner for $X'$ with refinement family $\mathcal{H}$ and refinement function
$$\mathcal{Q}_R(L) := \{R_e: e\in X'\setminus L\} \cup \{e\cup R'_e: e\in L\} \cup \left\{\bigcup_{e\in L} e^*\right\}.$$
First, we calculate that $\Delta(\mathcal{H}_2)\le 2^{M(q,r)\cdot \binom{v(X')}{r}} + 1 \le 2^{M(q,r)\cdot (C')^{r}} + 1 \le C_0$ and hence (1) holds.

Since $R_2\subseteq Y'$, we have that  $R_2$ is $0$-flat to $Y'$ and hence $R_2$ is $(r-1)$-flat to $Y$ and thus (2) holds.

Every element of $\mathcal{H}_2$ has at most $v(X')$ vertices and at most $M(q,r)\cdot \binom{v(X')}{r}$ edges. Hence $R_2$ is a $M(q,r)\cdot v(X')^r$-refined $K_q^r$-multi-refiner of $X'$ and hence is $\frac{C_0}{2}$-refined (since $C_0\ge 2\cdot M(q,r)\cdot (C')^{r}$) and hence (3) holds as desired.  
\end{proofclaim}

Let $R:=R_1\cup R_2$. Let $\mathcal{H}:= \mathcal{H}_1 \cup \mathcal{H}_{2}$. We claim that $R$ and $\mathcal{H}$ are as desired.

Recall that $\Delta_0= \max \{\Delta(X),~C'\cdot v(X)^{1-\frac{1}{r}}\cdot \log v(X)\}$. By Lemma~\ref{lem:BetterRefineDown}(3), we have that $\Delta(X')\le \frac{\Delta_0}{2}$. Since $v(X')\le |Y'|\le v(X)/2$ and $r\ge 2$, we find that 
$$\Delta'\le \left(\frac{1}{2}\right)^{1-\frac{1}{r}}\cdot \Delta_0$$

Now we find by Lemma~\ref{lem:BetterRefineDown}(1) for $\mathcal{H}_1$ and by (1) for $\mathcal{H}_2$ that
\begin{align*}
\Delta(\mathcal{H}) &\le \Delta(\mathcal{H}_1)+\Delta(\mathcal{H}_2) \le C' \cdot \Delta_0 + C_0 \cdot \Delta' \\
&\le \left(C'+C_0\cdot \left(\frac{1}{2}\right)^{1-\frac{1}{r}}\right)\cdot \Delta_0 \le C_0\cdot \Delta_0,
\end{align*}
where we used that $C_0\ge C'\cdot \frac{1}{1-\left(\frac{1}{2}\right)^{1-\frac{1}{r}}}$ since $C_0\ge 6\cdot C'$. Hence (1) holds for $\mathcal{H}$. 

Since $R_1$ is $(r-1)$-flat to $Y'$ by Lemma~\ref{lem:BetterRefineDown}(2); we have that $R_1$ is $(r-1)$-flat to $Y$ since $Y'\subseteq Y$. Similarly $R_2\subset Y'\subseteq Y$; hence $R_2$ is $0$-flat to $Y$ and so is also $(r-1)$-flat to $Y$. Thus we find that $R$ is $(r-1)$-flat to $Y$ and hence (2) holds for $R$. 

Let $e\in X$. If $e\in X\setminus X'$, then $$|\mathcal{H}(e)| = |\mathcal{H}_1(e)| \le C'\le \frac{C_0}{2},$$
since $R_1$ is $C'$-refined and $C_0\ge 2C'$. If $e\in X\cap X'$, then 
$$|\mathcal{H}(e)| = |\mathcal{H}_2(e)| \le \frac{C_0}{2},$$
where we used (3) for $R_2$ and that for all $H\in \mathcal{H}_1$, we have that $H\cap X\cap X'=\emptyset$ by definition of $K_q^r$-multi-refiner. Combining, we find that (3) holds for $\mathcal{H}$. 

Let $e\in R_1$. Then since $R_1$ is $C'$-refined and by (3) for $R_2$ and $X'$, we find that
$$|\mathcal{H}(e)| = |\mathcal{H}_1(e)| + |\mathcal{H}_2(e)| \le C' + \frac{C_0}{2} \le C_0,$$
since $C_0\ge 2C'$. 

Let $e\in R_2$. Then since $R_2$ is $C_0$-refined, we find that
$$|\mathcal{H}(e)| = |\mathcal{H}_2(e)| \le C_0.$$

Thus we find that $|\mathcal{H}(e)|\le C_0$ for all $e\in X\cup R$. Note that $\max\{e(H),v(H)\} \le \max\{C',C_0\} = C_0$ for all $H\in \mathcal{H}$. Hence we have that $R$ is $C_0$-refined.

Finally by Proposition~\ref{prop:Concatenate}, we find that $R$ is a $K_q^r$-multi-refiner with refinement family $\mathcal{H}$. Since all of (1)-(3) hold for $R$ and $\mathcal{H}$, we have that $R$ and $\mathcal{H}$ are as desired.
\end{lateproof}

\subsection{Multiplicity Reduction}\label{ss:MRL}

It is useful for proving the above lemmas and theorems to be able to reduce the multiplicity of a remainder to bounded size via the following lemma (which uses our RMH construction from Proposition~\ref{prop:RMHG}).

\begin{lem}[Multiplicity Reduction Lemma]\label{lem:MRL}
Let $q > r\ge 1$ be integers. Let $X$ be an $r$-uniform multi-hypergraph. Let $R$ be the $r$-uniform multi-hypergraph edge-disjoint from $X$ with $V(X)=V(R)$ and where $|R(S)|= M(q,r)\cdot |X(S)|$ for all $S\in \binom{V(X)}{r}$. Let $X'\subseteq R$ where $V(X')=V(R)$ and for each $S\in \binom{V(X)}{r}$, we have $|X'(S)| = M(q,r)-1$ if $|R(S)|\ne 0$ and $0$ otherwise. Then $R$ is a $2\cdot M(q,r)$-refined $K_q^r$-multi-refiner of $X$ with remainder $X'$ and refinement family $\mathcal{H}$ and  $\Delta(\mathcal{H}) \le 2\cdot M(q,r)\cdot \Delta(X)$. 
\end{lem}
\begin{proof}
Let $M:= M(q,r)$ for brevity. For each $S\in \binom{V(X)}{r}$ with $|X(S)|\ne 0$, let $X_S :=S\uplus X(S)$, $R_S:= S\uplus R(S)$, $X'_S := S\uplus X'(S)$.

By Proposition~\ref{prop:RMHG}, for each $S\in \binom{V(X)}{r}$, there exists an $M$-uniform hypergraph $H_S$ with vertex set $X_S\cup R_S$ such that all of Proposition~\ref{prop:RMHG}(1)-(4) hold for $X_S, R_S, H_S, X'_S$. For each $T\in E(H_S)$, we let $\mathcal{H}_{S,T} := \bigcup_{e\in V(T)} S\uplus e$ (that is, the subgraph of $S\uplus (X\cup R)(S)$ that corresponds to $T$). We let
$$\mathcal{H}_S := \{ \mathcal{H}_{S,T}: T\in E(H_S)\},$$
and
$$\mathcal{H} := \bigcup_{S \in \binom{V(X)}{r}} \mathcal{H}_S. $$

By Proposition~\ref{prop:RMHG}(1), we have that for each $H\in\mathcal{H}$, $|H\cap X|\le 1$. Moreover since $X'\subseteq R$ and $X$ and $R$ are edge-disjoint, it follows that $x\cap X'=\emptyset$ and hence $H\cap (X\cap X')=\emptyset$ for all $H\in \mathcal{H}$. By Proposition~\ref{prop:RMHG}(3), we find that for each $e\in X$, we have that $d\cdot |\mathcal{H}(e)|\le  2\cdot M$. Since $H$ is $M$-uniform, we find that $e(H)\le M$ for all $H \in \mathcal{H}$ (and of course $v(H)=r$ by construction). 

Let $L$ be a $K_q^r$-divisible subgraph of $X$. We define
$$\mathcal{H}_L:= \bigcup_{S\in \binom{V(X)}{r}:~X(S)\ne \emptyset} \{ \mathcal{H}_{S,T}: T\in M_{S,L}\}$$
where $M_{S,L}$ is the matching in $H_S$ for $L(S)$ guaranteed by Proposition~\ref{prop:RMHG}(4).
Note that elements of $\mathcal{H}_L$ are pairwise edge disjoint. Letting $H_L := \bigcup_{H\in \mathcal{H}_L} H$, we find by Proposition~\ref{prop:RMHG}(4) that $(L\cup R)\setminus X' \subseteq H_L \subseteq L\cup R$. Hence $R$ is a $2M$-refined $K_q^r$-multi-refiner of $X$ with remainder $X'$.
\end{proof}

\subsection{Proof of Edge Sparsification Lemma}\label{ss:EdgeSparsification}

In this subsection, we prove Lemma~\ref{lem:Sparsify}. Here is a brief overview. First we (quasi)-randomly partition $X$ into $k\approx v(X)^{(r-1)/r}$ parts. See Proposition~\ref{prop:QuasiRandomPartition}. Then for each $r$-multi-subset of parts, we choose a special set of $q-r$ vertices disjoint from those parts (in a small degree manner). See Lemma~\ref{lem:SpecialSets}. Finally we prove Lemma~\ref{lem:Sparsify} by letting $R$ be obtained by first for each edge $e\in X$ embedding $M(q,r)$ copies of a private clique on $e\cup S$ where $S$ is the special set corresponding to the parts of the vertices of $e$ and then second applying the Multiplicity Reduction Lemma to reduce the multiplicity of the remainder.

First we need the following proposition about quasi-random partitions of the vertex set of $X$.

\begin{proposition}\label{prop:QuasiRandomPartition}
Let $X$ be an $r$-uniform hypergraph. If $k$ is a positive integer such that $\Delta(X) \ge 3k(\log (2k)+(r-1)\log v(X))$, then there exists a partition $V_1,\ldots, V_k$ of $V(X)$ such that $|V_i|\le 2\cdot \frac{v(X)}{k}$ and for every $(r-1)$-subset $S$ of $V(X)$, $|X(S)\cap V_i|\le 2\cdot \frac{\Delta(X)}{k}$ for all $i\in [k]$.
\end{proposition}
\begin{proof}
For each vertex $v$ of $X$, put $v$ in a part $V_i$ independently uniformly at random. For each $i\in[k]$, let $A_i$ be the event that $|V_i| > 2\cdot \frac{v(X)}{k}$. For each $(r-1)$-subset $S$ of $X$ and $i\in [k]$, let $A_{S,i}$ be the event that $|X(S)\cap V_i| > 2\cdot \frac{\Delta(X)}{k}$.

Note that $\Expect{|V_i|} = \frac{v(X)}{k}$ for all $i\in [k]$. Hence by the Chernoff bound~\cite{AS16}, we find that
$$\Prob{A_i} \le e^{-\frac{v(X)}{3k}} < \frac{1}{2k},$$
since $v(X) \ge \Delta(X) \ge 3k\cdot \log (2k)$. Similarly we note that $\Expect{|X(S)\cap V_i|} \le \frac{\Delta(X)}{k}$. Hence by the Chernoff bound, we find that
$$\Prob{A_{S,i}} \le e^{-\frac{\Delta(X)}{3k}} < \frac{1}{2k\cdot v(X)^{r-1}},$$
since $\Delta(X)\ge 3k(\log (2k)+(r-1)\log v(X))$.
Hence by the union bound, we find that
$$\bigcup_{i\in [k]} \Prob{A_i}~\cup \bigcup_{S\in \binom{V(X)}{r-1}, i\in [k]} A_{S,i} < \frac{1}{2} + \frac{1}{2} < 1.$$
Thus there exists a partition $V_1,\ldots, V_k$ as desired. 
\end{proof}

Next we need a lemma that assigns to each $r$-multi-subset of $[k]$, a special set of $q-r$ vertices of $X$ without too much overlap. But first we need the following easy proposition.

\begin{proposition}\label{prop:NonUniformTuran}
Let $q\ge 1$ be an integer. Let $Z$ be a $q$-bounded hypergraph with at least $q$ vertices. For each $i\in [q]$, let $\alpha_i := \frac{|\{e\in Z: |e|=i\}|}{\binom{v(Z)}{i}}$.
Let $\alpha := \sum_{i\in [q]} \alpha_i \cdot \binom{q}{i}$.
Then the number of $q$-subsets of $V(Z)$ that span an edge of $Z$ is at most $\alpha \cdot \binom{v(Z)}{q}$ and hence if $\alpha < 1$, then there exists a $q$-subset of $V(Z)$ that does not span an edge of $Z$.
\end{proposition}
\begin{proof}
For each $i\in [q]$, let $A_i:= \{ S\in \binom{V(Z)}{q}: \exists~e\in Z,~|e|=i,~e\subseteq S\}$. Thus we have that
$$|A_i| \le |\{e\in Z: |e|=i\}|\cdot \binom{v(Z)-i}{q-i}= \alpha_i \cdot \binom{v(Z)}{i} \cdot \binom{v(Z)-i}{q-i} = \alpha_i \cdot \binom{q}{i} \cdot \binom{v(Z)}{q}.$$
Combining, we find that the number of $q$-subsets of $Z$ that span an edge of $Z$ is at most $\sum_{i\in [q]} |A_i| \le \alpha\cdot \binom{v(Z)}{q}$ as desired. And hence if $\alpha < 1$, there exists a $q$-subset of $V(Z)$ that does not span an edge of $Z$ as desired.
\end{proof}

\begin{lem}\label{lem:SpecialSets}
For all integers $q>r\ge 1$, there exists an integer $C\ge 1$ such that the following holds: Let $X$ be an $r$-uniform hypergraph with $v(X)\ge C$. Let $k$ be an integer satisfying $24r\le k \le \frac{1}{C}\cdot v(X)^{(r-1)/r}$ and let $\mathcal{I}$ be the set of all $r$-multi-subsets of $[k]$. If $V_1,\ldots,V_k$ is a partition of $V(X)$ as in Proposition~\ref{prop:QuasiRandomPartition} and $Y\subseteq V(X)$ with $|Y|\ge v(X)/3$, then there exists a collection $\mathcal{S} = (S_I \subseteq V(X): I \in \mathcal{I})$ such that all of the following hold:
\begin{enumerate}
    \item[(1)] for all $I\in \mathcal{I}$, $|S_I|=q-r$ and $S_I\subseteq Y\setminus \bigcup_{i\in I} V_i$,
    \item[(2)] for all distinct $I_1, I_2 \in \mathcal{I}$, $|S_{I_1}\cap S_{I_2}| \le r-2$, and
    \item[(3)] for every multi-subset $J$ of $[k]$ with $|J|\in [r-1]$ and $T\subseteq V(X)$ with $|T|\ = r-1-|J|$, we have that $|\{S_I \in \mathcal{S}: T\subseteq S_I,~J\subseteq I\}| \le \lceil C\cdot \frac{k^{|T|+1}}{v(X)^{|T|}} \rceil$.
\end{enumerate}
\end{lem}
\begin{proof}
Let $C$ be chosen large enough as needed throughout the proof. For $\ell\in [r-2]$, we let $m_{\ell}:= \left\lceil C\cdot \frac{k^{\ell+1}}{v(X)^{\ell}} \right\rceil$. Let $\mathcal{I}'\subseteq \mathcal{I}$ and $\mathcal{S}' = (S_I\subseteq V(X): I \in \mathcal{I}')$ such that (1)-(3) hold for $\mathcal{I}'$ and $\mathcal{S'}$ and subject to that $|\mathcal{I}'|$ is maximized. Note such a choice exists since $\mathcal{I}':=\emptyset$ satisfies (1)-(3) trivially. If $|\mathcal{I}'|=|\mathcal{I}|$, then $\mathcal{S}:=\mathcal{S}'$ is as desired.

So we assume that $|\mathcal{I}'| < |\mathcal{I}|$. Hence there exists $I\in \mathcal{I}\setminus \mathcal{I}'$. Let $Y':= Y\setminus \bigcup_{i\in I} V_i$. Note that since $|V_i| \le 2\cdot \frac{v(X)}{k}$ and $k\ge 24r$, we find that $|\bigcup_{i\in I} V_i|\le \frac{v(X)}{12}$. Hence $|Y'|\ge |Y|-\frac{v(X)}{12} \ge \frac{v(X)}{4}$ since $|Y|\ge \frac{v(X)}{3}$. 

Let $B_{r-1} := \bigcup_{I'\in \mathcal{I}'} \binom{S_{I'}\cap Y'}{r-1}$. Since $|\mathcal{I}'| \le |\mathcal{I}|\le k^r$, we find that 
$$|B_{r-1}| \le \binom{q-r}{r-1} \cdot |\mathcal{I}'| \le \binom{q-r}{r-1}\cdot k^r \le \binom{q-r}{r-1} \cdot \frac{v(X)^{r-1}}{C^r} \le \frac{q^r\cdot 8^r}{C^r} \binom{|Y'|}{r-1},$$
where we used that $|Y'|-r\ge \frac{v(X)}{4} - r \ge \frac{v(X)}{8}$ since $v(X)\ge C \ge 8r$. For brevity, we let $d_{\mathcal{S}'}(J,T):= |\{S_I \in \mathcal{S}': T\subseteq S_I,~J\subseteq I\}|$ for $J\subseteq I$ and $T\subseteq Y$. For $\ell \in [r-2]$, let
$$B_{\ell} := \left\{ T\in \binom{Y'}{\ell}: \exists~J \subseteq I \text{ with } |J|=r-1-\ell \text{ such that } d_{\mathcal{S}'}(J,T) = m_{\ell} \right\}.$$
Note that for $J\subseteq I$ with $|J|=r-1-\ell$,
\begin{align*}
\left|\left\{ T\in \binom{Y'}{\ell}: d_{\mathcal{S}'}(J,T)= m_{\ell}\right\}\right| &\le \frac{|\{I' \in \mathcal{I}': J\subseteq I'\}| \cdot \binom{q-r}{\ell}}{m_{\ell}} \le \frac{k^{\ell+1}\cdot \binom{q-r}{\ell}}{m_{\ell}} \\
&\le \frac{\binom{q-r}{\ell}}{C} \cdot v(X)^{\ell} \le \frac{q^r\cdot 8^r}{C}\cdot \binom{|Y|}{\ell},
\end{align*}
where we used that $|Y|\ge \frac{v(X)}{4}-r\ge \frac{v(X)}{8}$ since $v(X)\ge C \ge 8r$. Note there are at most $\binom{r}{r-1-\ell}$ multi-$\ell$-subsets $J$ of $I$. 

Note that $\sum_{\ell \in [r-1]} \frac{1}{C}\cdot \binom{r}{r-1-\ell} \cdot q^r\cdot 8^r \cdot \binom{q-r}{\ell} < 1$ since $C$ is large enough. Hence by Proposition~\ref{prop:NonUniformTuran}, there exists $S_{I}\subseteq Y'$ with $|S_{I}| = q-r$ such that $S_{I}$ contains no element of $\bigcup_{\ell \in [r-1]} B_{\ell}$. Let $\mathcal{I}'':=\mathcal{I}'\cup \{I\}$ and $\mathcal{S}'':=(S_I: I\in \mathcal{I}'')$. Note that $|S_I|=q-r$ and $S_I\subseteq Y'$ and hence (1) holds for $\mathcal{I}''$ and $\mathcal{S}''$. Since $S_I$ does not contain an element of $B_{r-1}$, it follows that (2) holds for $\mathcal{I}''$ and $\mathcal{S}''$. Finally note that for a multi-subset $J$ of $[k]$ with $|J|\in [r-1]$ and $T\subseteq V(X)$ with $|T|=\ell$ and $|J|=r-1-|T|$, we have that $d_{\mathcal{S}''}(J,T) \le d_{\mathcal{S}'}(J,T)+1$ and that equality holds only if $J\subseteq I$ and $T\subseteq S_I$. Since $S_I$ contains no element of $\bigcup_{\ell\in [r-2]} B_{\ell}$, it follows that $d_{\mathcal{S}'}(J,T)\le m_{\ell}$ for all such $J$ and $T$ and hence (3) holds for $\mathcal{I}''$ and $\mathcal{S}''$. Thus $\mathcal{I}''$ and $\mathcal{S}''$ satisfy outcomes (1)-(3), contradicting the choice of $\mathcal{I}'$ and $\mathcal{S}'$.
\end{proof}

We note for the reader's sake that in Lemma~\ref{lem:SpecialSets} the ceiling in outcome (3) is important as we must allow a value of $1$ even if the number inside the ceiling is some small fraction.  In that sense, the lemma will be more useful for bounds when $k$ is near the large end of its allowed range (where the values inside the ceiling are at least $1$ and hence the ceiling is mostly negligible).

\begin{remark}
We note that the proof above is deterministic and similar in spirit to the proof of Lemma 5.20 in~\cite{GKLO16} by Glock, K\"{u}hn, Lo, and Osthus. We use this style of proof of `deterministically avoiding the bad sets' in the proofs of two of our other embedding results (Theorem~\ref{thm:MultiOmni} and Lemma~\ref{lem:EmbedMinDegree}). We note though that other proof approaches also yield these types of results such as: a random greedy process, random sparsification, or using the Lov\'asz Local Lemma to find an $A$-perfect matching of some appropriate auxiliary bipartite hypergraph. We chose the above method for our current paper for its simplicity.    
\end{remark}

We are now prepared to prove Lemma~\ref{lem:Sparsify} as follows.

\begin{proof}[Proof of Edge Sparsification Lemma (Lemma~\ref{lem:Sparsify}).]
Let $C$ be chosen large enough as needed throughout the proof. Let $$k:=\left\lfloor \frac{1}{8q\cdot 2^r\cdot M(q,r)\cdot C'}\cdot v(X)^{1-\frac{1}{r}}\right\rfloor$$ 
where $C'$ is as in Lemma~\ref{lem:SpecialSets}. Since $\Delta(X)\ge v(X)^{(r-1)/r}\cdot \log v(X)$ by assumption, we find that $\Delta(X) \ge 3k(\log(2k)+(r-1)\log v(X))$. Hence by Proposition~\ref{prop:QuasiRandomPartition}, there exists a partition $V_1,\ldots, V_k$ of $V(X)$ as in Proposition~\ref{prop:QuasiRandomPartition}. 

Since $v(X)\ge C$ by assumption and $C$ is large enough, we find that $k \ge 24r$. By Lemma~\ref{lem:SpecialSets} applied to $X$ and $Y$, there exists a collection $\mathcal{S}=(S_I \subseteq V(X): I \text { is an $r$-multi-subset of $[k]$ })$ as in Lemma~\ref{lem:SpecialSets}. For a subset $U=\{v_1,\ldots, v_b\}$ of $X$ with $b\le r$, let $\phi(U)$ denote the $b$-multi-set $\{i: v_j\in V_i,~j \in [b]\}$. Let $R_1$ be obtained as follows: For each edge $e = \{v_1,\ldots v_r\}\in X$, add $M(q,r)$ edge-disjoint copies $P_{1,e},\ldots P_{M(q,r),e}$ of $\binom{e\cup S_{\phi(e)}}{r} \setminus e$.  

Note that $R_1$ is $K_q^r$-divisible. Let $X_2:=R_1$, let $\mathcal{H}_1$ be the family $\{ e\cup P_{1,e}: e\in X\}$ and let $\mathcal{Q}_{R_1}$ be the function $\mathcal{Q}_{R_1}(L) := \{ e\cup P_{1,e}: e\in L\}$. Let $C_1:= \max\{\binom{q}{r},q\}$. Now $R_1$ is a $C_1$-refined $K_q^r$-multi-refiner of $X$ with refinement family $\mathcal{H}_1$ and refinement function $\mathcal{Q}_{R_1}$.

Let $R_2$ be the $r$-uniform multi-hypergraph edge-disjoint from $X_2\cup X$ with $V(X_2)=V(X)=V(R_2)$ and where $|R_2(S)|= M(q,r)\cdot |X_2(S)|$ for all $S\in \binom{V(X)}{r}$. Let $X_3\subseteq R$ where $V(X_3)=V(R_2)$ and for each $S\in \binom{V(X)}{r}$, we have $|X_3(S)| = M(q,r)-1$ if $|R_2(S)|\ne 0$ and $0$ otherwise. Let $C_2:= 2\cdot M(q,r)$. By the Multiplicity Reduction Lemma (Lemma~\ref{lem:MRL}), we have that $R_2$ is a $C_2$-refined $K_q^r$-multi-refiner of $X$ with remainder $X_3$ and refinement family $\mathcal{H}_2$ and $\Delta(\mathcal{H}_2) \le 2\cdot M(q,r)\cdot \Delta(X_2)$.

Let $R=R_1\cup R_2$. By Proposition~\ref{prop:Concatenate}, $R$ is a $C$-refined $K_q^r$-multi-refiner of $X$ with remainder $X_3$ and refinement family $\mathcal{H}:=\mathcal{H}_1\cup \mathcal{H}_2$ (since $C\ge C_1+C_2$ as $C$ is large enough).

Now we prove that (1)-(4) hold for $\mathcal{H}$ and $X'$. First let us define some useful constants. For $i\in [r-1]_0$, let $m_i:= C'\cdot \frac{k^{i+1}}{v(X)^{i}}$. Since $k\le v(X)$, we have that $m_i\ge m_{r-2}$ for all $i\in [r-1]_0$. Using that $v(X)\ge C$ and $C$ is large enough, it follows that $m_i\ge 1$ for all $i\in [r-2]_0$. Meanwhile, using the definition of $k$ and that $C$ is large enough, it follows that $m_{r-1}\ge \frac{1}{\sqrt{C}}$. 

For $S\in \binom{V(X)}{r-1}$ and $T\subseteq S$, define 
$$\mathcal{I}_{S,T} := \{I: T\subseteq S_I,~\phi(S\setminus T)\subseteq I\}.$$
It follows from Lemma~\ref{lem:SpecialSets}(2) that if $|T|=r-1$, then $|\mathcal{I}_{S,T}|\le 1 \le \sqrt{C}\cdot m_{|T|}$. By Lemma~\ref{lem:SpecialSets}(3), we have that if $|T| < r-1$, then $|\mathcal{I}_{S,T}|\le \left\lceil m_{|T|} \right\rceil \le 2\cdot m_{|T|}$ since $m_{|T|}\ge 1$ for $|T|\in [r-2]_0$. In either case then since $C$ is large enough, we have that $|\mathcal{I}_{S,T}|\le \sqrt{C}\cdot m_{|T|}$.

Now we calculate $\Delta(R_1)$ as follows. Fix $S\in \binom{V(X)}{r-1}$. 
Then
$$|R_1(S)| \le \sum_{T\subseteq S} |\{e\in X: S\setminus T\subseteq e,~\phi(e)\in \mathcal{I}_{S,T}\}|.$$
Now we calculate the terms on the right side as follows. For $I\in \mathcal{I}_{S,T}$, let $I=\phi(S\setminus T)\cup \{i_1,\ldots, i_{|T|+1}\}$ and we calculate using the properties of the partition $V_1,\ldots V_k$ that
\begin{align*}
|\{e\in X: S\setminus T\subseteq e,~\phi(e)\in I\}| &\le \sum_{U\in \binom{V(X)\setminus S}{|T|}: \phi(U)= \{i_1,\ldots, i_{|T|}\}} |X(U)\cap V_{i_{|T|+1}}|,\\
&\le \left(\prod_{a=1}^{|T|} |V_{i_a}|\right) \cdot \left(2\cdot \frac{\Delta(X)}{k}\right) \le \left(2\cdot \frac{v(X)}{k}\right)^{|T|} \cdot \left(2\cdot \frac{\Delta(X)}{k}\right)\\
&= 2^{|T|+1}\cdot \frac{v(X)^{|T|}}{k^{|T|+1}} \cdot \Delta(X) = 2^{|T|+1}\cdot \frac{C'}{m_{|T|}} \cdot \Delta(X).
\end{align*}
Hence we find that
\begin{align*}
|\{e\in X: S\setminus T\subseteq e,~\phi(e)\in \mathcal{I}_{S,T}\}| &\le 2^{|T|+1}\cdot \frac{C'}{m_{|T|}}\cdot \Delta(X)\cdot |\mathcal{I}_{T}|\\
&\le 2^{|T|+1}\cdot C'\cdot \sqrt{C} \cdot \Delta(X),
\end{align*}
where we used that $|\mathcal{I}_{S,T}|\le \sqrt{C}\cdot m_{|T|}$. Combining we find that
$$|R_1(S)| \le \sum_{T\subseteq S}  |\{e\in X: S\setminus T\subseteq e,~\phi(e)\in \mathcal{I}_{S,T}\}| \le 4^{r}\cdot C'\cdot \sqrt{C}\cdot \Delta(X),$$
where we used that $|T|\le r-1$ and $|\{T: T\subseteq S\}|\le 2^r$. Hence 
$$\Delta(R_1) \le 4^{r}\cdot C'\cdot \sqrt{C}\cdot \Delta(X).$$

We note that $\Delta(\mathcal{H}_1)\le \Delta(R_1)+\Delta(X)$. Moreover, $\Delta(\mathcal{H}_2)\le 2\cdot M(q,r) \cdot \Delta(X_2) = 2\cdot M(q,r)\cdot \Delta(R_1)$. Hence
$$\Delta(\mathcal{H}) \le \Delta(\mathcal{H}_1)+\Delta(\mathcal{H}_2)\le (2\cdot M(q,r)+1)\cdot \Delta(R_1)+\Delta(X) \le C\cdot \Delta(X),$$
since $C$ is large enough. Hence (1) holds.

By Lemma~\ref{lem:SpecialSets}, we have that $S_I\subseteq Y\setminus \bigcup_{i\in I} V_i$ for all $r$-multi-subsets $I$ of $[k]$. For every edge $e\in X$ and $i\in [M(q,r)]$, we have that every edge of $P_{i,e}$ contains a vertex of $S_{\phi(e)}$ by construction and hence a vertex of $Y$. It follows that $R_1$ is $(r-1)$-flat to $Y$. Since $|R_2(S)|\ne 0$ only if $|R_1(S)|\ne 0$ for $S\in \binom{V(X)}{r}$, it then follows that $R_2$ is also $(r-1)$-flat to $Y$. Combining, we find that $R$ is $(r-1)$-flat to $Y$ and hence (2) holds. 

Next we calculate $\Delta(X_3)$ as follows. Fix $S\in \binom{V(X)}{r-1}$. Now using that $X_3$ has multiplicity at most $M(q,r)$, we calculate that
\begin{align*}
|X_3(S)|\le M(q,r)\cdot \bigg( &\sum_{T\subseteq S: T\ne \emptyset} |\{v\in V(X)\setminus S: \exists I\text{ such that } T\subseteq S_I,~\phi(\{v\}\cup (S\setminus T)) \subseteq I\}| \\
&+ \sum_{T\subseteq S} |\{v\in V(X)\setminus S: \exists I\text{ such that } T\cup \{v\}\subseteq S_I,~\phi(S\setminus T) \subseteq I\}| \bigg).
\end{align*}
Note crucially that the case $T=\emptyset$ in the first term is excluded since for every $e\in X$ and $i\in [M(q,r)]$, $P_{i,e}$ contains at least one vertex of $S_{\phi(e)}$ by construction.

Using the properties of the partition $V_1,\ldots,V_k$, we find that
\begin{align*}
|X_3(S)| &\le M(q,r)\cdot \bigg( \sum_{T\subseteq S: T\ne\emptyset}~~\sum_{I\in \mathcal{I}_{S,T}}~~\sum_{a\in I\setminus \phi(S\setminus T)} |V_a|~~+~~\sum_{T\subseteq S}~~\sum_{I\in \mathcal{I}_{S,T}} (q-r-|T|)\bigg)\\
&\le M(q,r)\cdot \bigg( \sum_{T\subseteq S: T\ne\emptyset}~~\sum_{I\in \mathcal{I}_{S,T}}~ (|T|+1)\cdot  \left(2\cdot\frac{v(X)}{k}\right)~~+~~\sum_{T\subseteq S}~~q\cdot |\mathcal{I}_{S,T}|\bigg)\\
&\le 2q\cdot M(q,r)\cdot \bigg( \sum_{T\subseteq S: T\ne\emptyset}~~|\mathcal{I}_{S,T}|\cdot\frac{v(X)}{k}~~+~~\sum_{T\subseteq S}~~|\mathcal{I}_{S,T}|\bigg)\\
\end{align*}
Recall that $|\mathcal{I}_{S,T}|\le 1$ if $|T|=r-1$ and $|\mathcal{I}_{S,T}| \le 2\cdot m_{|T|}$ if $|T|\in [r-2]_0$. Also note that since $k\le v(X)$, we have $m_i \le m_1\le m_0$ for all $i\in [r-2]$. Further note that $m_1\ge 1$. Hence we find that
\begin{align*}
|X_3(S)|&\le 2q\cdot M(q,r)\cdot \bigg( \sum_{T\subseteq S: T\ne\emptyset}~~2\cdot m_1\cdot\frac{v(X)}{k}~~+~~\sum_{T\subseteq S}~~2\cdot m_0\bigg)\\
&\le 4q\cdot 2^r \cdot M(q,r) \cdot \left( m_1\cdot \frac{v(X)}{k} + m_0\right).
\end{align*}
However by definition, we have that $m_1\cdot \frac{v(X)}{k} = m_0 = C'\cdot k$. Substituting, we find that
\begin{align*}
|X_3(S)|&\le 8q\cdot 2^r \cdot M(q,r) \cdot C'\cdot k \le v(X)^{1-\frac{1}{r}}, 
\end{align*}
where we used the definition of $k$. Hence (3) holds.

Finally we note that $X_3$ is an $r$-uniform multi-hypergraph with multiplicity at most $M(q,r)-1$. Hence (4) holds.
\end{proof}

\section{Proof of the Refine Down Lemma}\label{s:RefineDown}

In this section, we prove Lemma~\ref{lem:RefineDown}. Here we provide a brief overview. It is useful to break up the proof of Lemma~\ref{lem:RefineDown} further into steps where each step we only refine edges with a fixed intersection size with $Y$ (see Lemma~\ref{lem:RefineDownOneStep}). Again, this is similar to both the proof of the Cover Down Lemma (but refining instead of covering) and to the construction of Transformers achieved in~\cite{GKLO16} (but going one direction instead of symmetrically).

Then it is useful to further break up the proof by refining only the edges containing a fixed $i$-set $S$ of $V(X)\setminus Y$. To that end, we define a \emph{local refiner} which is a refiner (with remainder) for `link' hypergraphs (see Definition~\ref{def:Local}). Crucially any link hypergraph of a divisible hypergraph is also divisible with lower uniformity (see Proposition~\ref{prop:LinkDiv}). Hence we prove a Local Multi-Refiner Theorem (Theorem~\ref{thm:MLRT}) whose proof uses the inductive assumption that Theorem~\ref{thm:MultiOmni} holds for smaller values of $r$. 

Lemma~\ref{lem:RefineDownOneStep} will then follow by randomly embedding the local multi-refiners from Theorem~\ref{thm:MLRT} so as to have small degree. Note to ensure the expected degree is small, this part requires a set of new vertices for embedding the rest of each local refiner. This is the key place where we require the flatness assumption(s) that we have to carry throughout the other lemmas. Lemma~\ref{lem:RefineDown} will then follow by repeated applications of Lemma~\ref{lem:RefineDownOneStep} - where we need larger and larger sets of new vertices for each step of refine down.

\subsection{Proof Overview for Refine Down Lemma}

Here is one step of our Refine Down Lemma.

\begin{lem}[Refine Down Lemma - One Step]\label{lem:RefineDownOneStep}
For all integers $q > r\ge 2$, there exists an integer $C > 1$ such that the following holds: Suppose that Theorem~\ref{thm:MultiOmni} holds for $r'\in [r-1]$. Let $X$ be an $r$-uniform multi-hypergraph of multiplicity at most $M(q,r)$ and let $\Delta:=\max\{\Delta(X),~v(X)^{1-\frac{1}{r}}\cdot \log v(X)\}$. Let $Y\subseteq V(X)$ with $|Y|\ge C$ and $i\in [r]$ such that $X$ is $i$-flat to $Y$. If $Z\subseteq V(X)$, with $Y\subseteq Z$ and $|Z|\ge 2\cdot C\cdot |Y|$ then there exists a $C$-refined $K_q^r$-multi-refiner $R$ of $X$ with remainder $X'$ and refinement family $\mathcal{H}$ such that all of the following hold:

\begin{enumerate}
    \item[(1)] $\Delta_{\mathcal{H}}(T) \le C \cdot \Delta$ if $T\in \binom{V(X)}{r-1}$ with $|T\setminus Y|=i$,
    \item[(2)] $\Delta_{\mathcal{H}}(T) \le C \cdot \max\left\{\left(\frac{v(X)}{|Y|}\right)^{i-1} \cdot \Delta,~\left(\frac{v(X)}{|Y|}\right)^{\min\{i,r-1\}}\right\}$ if $T\in \binom{V(X)}{r-1}$ with $|T\setminus Y|<i$,
    \item[(3)] $R$ is $(r-1)$-flat to $Y$ and for each $H\in \mathcal{H}$, we have that $|V(H)\setminus Z| \le i$, 
    \item[(4)] $X'$ is $(i-1)$-flat to $Z$,
    \item[(5)] $\Delta(X') \le \Delta(X) + C \cdot \max\left\{\left(\frac{v(X)}{|Y|}\right)^{i-1} \cdot \Delta,~\left(\frac{v(X)}{|Y|}\right)^{\min\{i,r-1\}}\right\}$,
    \item[(6)] $X'$ is an $r$-uniform multi-hypergraph with multiplicity at most $M(q,r)$,
\end{enumerate}          
\end{lem}

Here is the definition of local multi-refiner with remainder.

\begin{definition}[Local Multi-Refiner with Remainder]\label{def:Local}
Let $q > r > i\ge 1$ be integers. Let $X$ be an $r$-uniform hypergraph and $S\subseteq V(X)$ with $|S|=i$ such that $S\subseteq e$ for all $e\in E(X)$.

We say an $r$-uniform hypergraph $R$ is a \emph{local $K_q^r$--multi-refiner} of $X$ at $S$ with \emph{remainder} $X' \subseteq X\cup R$ and \emph{refinement family} $\mathcal{H}$, if $X$ and $R$ are edge-disjoint, $R$ is $K_q^r$-divisible, and $\mathcal{H}$ is a family of $K_q^r$-divisible subgraphs of $X\cup R$ such that $|H\cap X|\le 1$ and $H\cap X\cap X'=\emptyset$ for all $H\in\mathcal{H}$, and for every subgraph $L$ of $X$ where $L(S)$ is $K_{q-i}^{r-i}$-divisible, there exists a collection of edge-disjoint subgraphs $H_1,\ldots, H_k \in \mathcal{H}$ such that letting $H_L := \bigcup_{i\in [k]} H_i$, we have $(L\cup R)\setminus X' \subseteq H_L \subseteq L\cup R$. We call $L':= (L\cup R) \setminus H_L$ the \emph{leftover} of the refinement.

Let $C\ge 1$ be a real. We say $R$ is \emph{$C$-refined} if $\mathcal{H}$ is $C$-refined.
\end{definition}

Here is an easy proposition about the divisibility of link hypergraphs.

\begin{proposition}\label{prop:LinkDiv}
Let $L$ be a $K_q^r$-divisible $r$-uniform (multi)-hypergraph. If $i\in [r-1]$ and $S \in \binom{V(L)}{i}$, then $L(S)$ is a $K_{q-i}^{r-i}$-divisible (multi)-hypergraph.   
\end{proposition}
\begin{proof}
Let $j\in [r-1-i]$ and let $T\in \binom{V(L)\setminus S}{j}$. Since $L$ is a $K_q^r$-divisible, we find that $\binom{q-(i+j)}{r-(i+j)}\mid d_L(S\cup T)$. Since $d_L(S\cup T) = |L(S\cup T)| = d_{L(S)}(T)$, it follows that $\binom{(q-i)-j}{(r-i)-j}~|~d_{L(S)}(T)$ and hence $L(S)$ is a $K_{q-i}^{r-i}$-divisible hypergraph as desired.    
\end{proof}

Here then is a Local Multi-Refiner Theorem.

\begin{thm}[Local Multi-Refiner Theorem]\label{thm:MLRT}
For all integers $q > r \ge 2$, there exists an integer $C > 1$ such that the following holds: Let $i\in [r-1]$ and suppose that Theorem~\ref{thm:MRT} holds for $r':=r-i$. Let $X$ be an $r$-uniform multi-hypergraph of multiplicity at most $M(q,r)$ and let $\Delta:= \max\{\Delta(X),~v(X)^{1-\frac{1}{r}}\cdot \log v(X)\}$. Suppose that there exists $S\subseteq V(X)$ with $|S|=i$ such that $S\subseteq e$ for all $e\in X$. If there exists $Y\subseteq V(X)$ with $|Y|\le  v(X)/C$ such that $e\subseteq Y$ for all $e\in X$, then there exists a $C$-refined local $K_q^r$-multi-refiner $R$ of $X$ at $S$ with remainder $X'\subseteq X\cup R$ and refinement family $\mathcal{H}$ such that all of the following hold: 
\begin{enumerate}
    \item[(1)] $\Delta_{\mathcal{H}}(T) \le C \cdot \Delta$ if $S\subseteq T$ and $T\in \binom{V(X)}{r-1}$,
    \item[(2)] $\Delta_{\mathcal{H}}(T) \le C$ if $S\setminus T\ne \emptyset$ and $T\in \binom{V(X)}{r-1}$,
    \item[(3)] $R$ is $(r-1)$-flat to $V(R)\setminus (S\cup Y)$; $R$ is $(r-i)$-flat to $V(R)\setminus Y$; for all $e\in E(R)$, if $|e\cap Y|=r-i$, then $(e\cap Y)\cup S = V(f)$ for some $f\in X$; 
    \item[(4)] $X'$ is $(i-1)$-flat to $V(R)\setminus S$,
    \item[(5)] $\Delta(X')\le C$.
\end{enumerate}      
\end{thm}

Theorem~\ref{thm:MLRT} will follow by extending a refined $K_{q-i}^{r-i}$-multi-omni-absorber of $X(S)$ to a local $K_q^r$-refiner by adding more edges (not containing $S$) so as to make the subgraphs used by the refiner $K_q^r$-divisible (in particular they will be $K_q^r$'s). Again, this construction is where we require the inductive assumption that Theorem~\ref{thm:MultiOmni} holds for smaller values of $r$. 

\subsection{Proof of Local Multi-Refiner Theorem}

Now we proceed to prove Lemma~\ref{lem:RefineDown} by proving the various intermediate lemmas in reverse. To that end, we first prove the Local Multi-Refiner Theorem.

\begin{proof}[Proof of Local Multi-Refiner Theorem (Theorem~\ref{thm:MLRT})]
Let $C:= \max\{2^{r-1},~M(q,r)\cdot \binom{q}{r}\cdot C'\} \cdot C'$ where $C'$ is the maximum over all $i\in[r-1]$ of the value of $C_{\ref{thm:MultiOmni}}$ for $q-i, r-i$. 

Let $\Delta':=\max\{\Delta(X(S)),~v(X(S))^{1-\frac{1}{r}}\cdot \log v(X(S))\}$. Note $\Delta(X(S))=\Delta(X)$ and $v(X(S)) = v(X)-|S| \le v(X)$; hence $\Delta' \le \Delta$. By assumption, Theorem~\ref{thm:MultiOmni} holds for $r' := r-i$. Hence, there exists a $C'$-refined $K_{q-i}^{r-i}$-multi-omni-absorber $A$ of $X(S)$ with decomposition family $\mathcal{H}'$ such that $\Delta(A)\le C' \cdot \Delta' \le C'\cdot \Delta$, and $A$ is $(r-1-i)$-flat to $V(X)\setminus (S\cup Y)$. Note that for each $H\in\mathcal{H}'$, we have $|H\cap X(S)|\le 1$ by the definition of multi-omni-absorber. 

Since the empty set is $K_{q-i}^{r-i}$-divisible, it follows that there exists a $K_{q-i}^{r-i}$-decomposition $\mathcal{Q}$ of $A$ such that $\mathcal{Q}\subseteq \mathcal{H}'$.

Let $R$ be the $r$-uniform multi-hypergraph obtained from $A$ by adding the following edges: For each $H\in \mathcal{Q}$ and $T \in \binom{V(H)\cup S}{r}$ where $S\setminus T\ne \emptyset$, add a set $R_{H,T}$ of $M(q,r)+1$ edges with vertex set $T$. For each $H\in \mathcal{H}'\setminus \mathcal{Q}$ and $T \in \binom{V(H)\cup S}{r}$ where $S\setminus T\ne \emptyset$, add a set $R_{H,T}$ of $M(q,r)$ edges with vertex set $T$. 

Let $X':= R\setminus A$. For each $H\in\mathcal{H}'$, let $\phi(H)$ be obtained from $H$ by adding the first edge in $R_{H,T}$ for all $T \in \binom{V(H)\cup S}{r}$ with $S\setminus T\ne \emptyset$. Let $\mathcal{H} := \{\phi(H): H\in \mathcal{H}'\}$. 

We prove that $R$ is a $C$-refined local $K_q^r$-multi-refiner of $X$ as $S$ with remainder $X'$ and refinement family $\mathcal{H}$ such that all of (1)-(5) hold for $R$, $\mathcal{H}$ and $X'$.

First we show $R$ is a $C$-refined. Since each $H\in\mathcal{H}$ is a $K_q^r$, we have that $\max\{v(H),e(H)\} \le \max\{q,\binom{q}{r}\} \le \binom{q}{r} \le C$. Let $e\in X\cup R$. If $e\in X$, then $|\mathcal{H}(e)| = |\mathcal{H}'(e)|$ which is at most $C'$ since $A$ is $C'$-refined. If $e\in R$, then $|\mathcal{H}(e)|\le 1$ by construction. Hence in either case, $|\mathcal{H}(e)|\le C' \le C$ since $C\ge C'$. Thus $R$ is $C$-refined.

Next we show that $R$ is a local $K_q^r$-multi-refiner of $X$ at $S$. First note that $R$ is $K_q^r$-divisible by construction. Similarly by construction, we have that $|\phi(H')\cap X|\le 1$ for all $H'\in \mathcal{H}'$ and hence $|H\cap X|\le 1$ for all $H\in \mathcal{H}$. Also note that $X'\subseteq R$; hence since $X$ and $R$ are edge-disjoint, we find that $X\cap X'=\emptyset$ and thus $H\cap X\cap X'=\emptyset$ for all $H\in \mathcal{H}$.

Now let $L$ be a subgraph of $X$ such that $L(S)$ is $K_{q-i}^{r-i}$-divisible. By definition of refined multi-omni-absorber, there exists a $K_{q-i}^{r-i}$ decomposition $\mathcal{H}'_L \subseteq \mathcal{H}'$ of $L\cup A$. We let $\mathcal{H}_L := \{ \phi(H): H\in \mathcal{H}'_L\}$. Let $H_L := \bigcup_{H\in \mathcal{H}_L} H$. Since $\phi(H)\supseteq H$ for all $H\in\mathcal{H}'$, it follows that $H_L \supseteq L\cup A = (L\cup R)\setminus X'$ since $X' = R\setminus A$. Furthermore, $H_L \subseteq L\cup R$ by construction. Hence $R$ is a local $K_q^r$-multi-refiner of $X$ at $S$.

Now we proceed to show that all of (1)-(5) hold for $R$, $\mathcal{H}$ and $X'$. Let $T\in \binom{V(X)}{r-1}$ such that $S\subseteq T$. Then 
$\Delta_{\mathcal{H}}(T) = \Delta_{\mathcal{H'}}(T\setminus S)$. Yet 
$$\Delta_{\mathcal{H'}}(T\setminus S) \le \Delta(\mathcal{H}')\le C'\cdot M(q,r)\cdot \Delta(A) \le (C')^2 \cdot M(q,r)\cdot \Delta,$$ 
since $X$ has multiplicity at most $M(q,r)$ by assumption, every $H\in\mathcal{H}'$ is a $K_{q-i}^{r-i}$ and every $e\in X(S)$ is in at most $C'$ elements of $\mathcal{H}'$ since $A$ is $C'$-refined. Thus we find that
$$\Delta_{\mathcal{H}}(T) \le (C')^2 \cdot M(q,r)\cdot \Delta \le C\cdot \Delta$$
since $C\ge (C')^2\cdot M(q,r)$. Hence (1) holds. 

Let $T\in \binom{V(X)}{r-1}$ such that $S\setminus T \ne \emptyset$. Then $\Delta_{\mathcal{H}}(T) = |\{H\in\mathcal{H}': V(T\setminus S) \subseteq V(H)\}|$. However, since every $H\in\mathcal{H}'$ is a $K_{q-i}^{r-i}$ and $|T\setminus S|\ge r-i$ as $S\setminus T\ne \emptyset$, we have that  
$$\Delta_{\mathcal{H}}(T)\le \sum_{e\in X(S): e \subseteq V(T\setminus S)}  |\mathcal{H}'(e)|.$$ 
Since $A$ is $C'$-refined, we have that $|\mathcal{H}'(e)|\le C'$ for all $e\in X(S)$. Hence we find that
$$\Delta_{\mathcal{H}}(T) \le \binom{|T\setminus S|}{r-i} \cdot C' \le \binom{r-1}{r-i}\cdot C'\le C$$
since $C \ge 2^{r-1}\cdot C'$. Hence (2) holds.

For $H\in \mathcal{H}'$, we have that $H$ is a $K_{q-i}^{r-i}$ with $|H\cap X'|\le 1$. Since $A$ is $(r-1-i)$-flat to $V(X)\setminus (S\cup Y)$, it follows that $|V(H)\cap Y|\le r-i$ with equality only if $V(H)\cap Y = V(f)$ for some $f\in X(S)$. Hence, for $T \in \binom{V(H)\cup S}{r}$ where $S\setminus T\ne \emptyset$, we find that $|T\cap Y|\le r-i$ with equality only if $T\cap Y = V(f)$ for some $f\in X(S)$. Hence, we find $R$ is $(r-i)$-flat to $V(R)\setminus Y$. Moreover, $R$ is also $(r-1)$-flat to $V(X)\setminus (S\cup Y)$ by construction. Hence (3) holds. 

Note that for all $e\in X'$, we have that $S\setminus e\ne \emptyset$ by construction. Hence (4) holds. 

Since $X'$ is $(i-1)$-flat to $V(R)\setminus S$ by (4), it follows that $|X'(T)|=0$ for all $T\in \binom{V(X)}{r-1}$ with $S\subseteq T$. Meanwhile for $T\in \binom{V(X)}{r-1}$ with $S\setminus T\ne \emptyset$, we find that $|X'(T)| \le (q-r-1)\cdot \Delta_{\mathcal{H}}(T)$ since every element of $\mathcal{H}$ is a $K_q^r$. Since $\Delta_{\mathcal{H}}(T)\le \binom{r-1}{r-i}\cdot C'$ as shown above, we find that $|X'(T)| \le (q-r-1)\cdot \binom{r-1}{r-i}\cdot C' \le C$ as $C$ is large enough. Thus $\Delta(X')\le C$ and hence (5) holds.

Hence, we have that $R$ is a $C$-refined $K_q^r$-multi-refiner of $X$ with remainder $X'$ and refinement family $\mathcal{H}$ and that all of (1)-(5) hold as desired.



\end{proof}

\subsection{Proof of One Step Refine Down Lemma}

We are now prepared to prove the One Step Refine Down Lemma as follows.

\begin{lateproof}{lem:RefineDownOneStep}
Let $C:= \max\{2(C')^{r},r\}$ where $C'$ is as $C_{\ref{thm:MLRT}}$ in Theorem~\ref{thm:MLRT} for $q$ and $r$.

It follows from the Multiplicity Reduction Lemma (Lemma~\ref{lem:MRL}) that it suffices to show only that (1)-(5) hold and not (6) (namely, this would follow by applying the Multiplicity Reduction Lemma to the remainder $X'$ to obtain a refiner $R'$ of $X'$ with a new remainder $X''$ of bounded multiplicity and taking the union of $R\cup R'$ and checking that the conditions (1)-(6) still hold albeit with a worse constant).

Let $S \subseteq X\setminus Z$ with $|S|=i$. Since $X$ is $i$-flat to $Y$, it follows that for each $e\in X$ with $S\subseteq e$, we have the $e\setminus S \subseteq Y$. Since $|Y| \le \frac{|Z|}{C} \le \frac{|Z|}{C'}$, we have by Theorem~\ref{thm:MLRT} that there exists a $C'$-refined local $K_q^r$-multi-refiner $R_S$ of $X_S:= S\uplus X(S)$ at $S$ with remainder $X'_S\subseteq X_S\cup R_S$ and refinement family $\mathcal{H}_S$ such that $V(R_S)\subseteq S\cup Z$ and all of Theorem~\ref{thm:MLRT}(1)-(5) hold.  

For a permutation $\sigma$ of $Z\setminus Y$, let $H_{S,\sigma}$ be the graph obtained from $R_S$ by permuting the vertices of $Z\setminus Y$ by $\sigma$. Note that for any permutation $\sigma$, $R_{S,\sigma}$ is still a local $K_q^r$-multi-refiner of $X_S$ at $S$ with remainder $X'_{S,\sigma_S}$ and refinement family $\mathcal{H}_{S,\sigma_S}$ (obtained by permuting $X'$ and elements of $\mathcal{H}_S$ respectively by $\sigma_S$) and such that all of Theorem~\ref{thm:MLRT}(1)-(5) hold.

For each $i$-set $S$ of $X\setminus Z$, choose a random permutation $\sigma_S$ of $Z\setminus Y$. Let 
$$R:= \bigcup_{S\in \binom{X\setminus Z}{i}} R_{S,\sigma_S},~~\mathcal{H}:= \bigcup_{S\in \binom{X\setminus Z}{i}} \mathcal{H}_{S,\sigma_S}, ~~ X':= \left\{U\in X: |U\setminus Z| < i\right\}~\cup \bigcup_{S\in \binom{X\setminus Z}{i}} X'_{S,\sigma_S}.$$ 
Since each $H_{S,\sigma_S}$ is $K_q^r$-divisible and the $H_{S,\sigma_S}$ are pairwise edge-disjoint by construction, we find that $R$ is $K_q^r$-divisible.

We prove that $R$ is a $C$-refined $K_q^r$-multi-refiner of $X$ with remainder $X'$ and refinement family $\mathcal{H}$ such that all of (1)-(5) hold for $R$, $\mathcal{H}$ and $X'$.

First we show $R$ is a $C$-refined $K_q^r$-multi-refiner of $X$. Note that since the $X_S\cup R_S$ are edge-disjoint, it follows that $R$ is $C$-refined since each $R_S$ is.

Now we show $R$ is a $K_q^r$-multi-refiner of $X$. First note that for every $H\in \mathcal{H}$, we have that $|H\cap X|\le 1$ and  $H\cap (X\cap X')=\emptyset$ since these properties hold for $H_{S,\sigma_S}$ as $R_{S,\sigma_S}$ is a local $K_{q}^{r}$-multi-refiner of $X_S$ at $S$ (and the $X_S\cup R_S$ are all edge-disjoint).

Now let $L$ be a $K_q^r$-divisible subgraph of $X$. By Proposition~\ref{prop:LinkDiv}, we find that $L(S)$ is $K_{q-i}^{r-i}$-divisible for each $i$-set $S$ in $X\setminus Z$. By definition of local multi-refiner, we have that there exists a collection of edge-disjoint subgraphs $\mathcal{H}_{L,S} \subseteq \mathcal{H}_{S,\sigma}$ such that letting $H_{L,S} := \bigcup_{H\in \mathcal{H}_{L,S}} H$, we have that $(L(S)\cup R_S)\setminus X'_S \subseteq H_{L,S} \subseteq L(S)\cup R_S$.

Now let $\mathcal{H}_L := \bigcup_S \mathcal{H}_{L,S}$ and note this is a collection of edge-disjoint subgraphs of $\mathcal{H}$ since the $X_S\cup R_S$ are edge-disjoint. Similarly let $H_L := \bigcup_{H\in \mathcal{H}_{L}} H$ and note this equals $\bigcup_S H_{L,S}$ since the $X_S\cup R_S$ are edge-disjoint. Thus $H_L \subseteq \bigcup_S (L(S)\cup R(S)) = L\cup R$. Similarly we have that $H_L \supseteq \bigcup_S  (L(S)\cup R_S)\setminus X'_S = (L\cup R)\setminus X'$. Hence $R$ is a $K_q^r$-multi-refiner of $X$ with remainder $X'$ and refinement family $\mathcal{H}'$.

Next we show that all of (1)-(5) hold for $R$, $\mathcal{H}$ and $X'$. To do that, we first show that (1), (3) and (4) hold for any choice of the $\sigma_S$, and then we show that (2) and (5) hold for some choice of the $\sigma_S$.

Let $T\in \binom{V(X)}{r-1}$ with $|T\setminus Y|=i$. Let $S:= T\setminus Y$. Let $\Delta_S := \max\{\Delta(X(S)),~|Z|^{1-\frac{1}{r}}\cdot \log |Z|\}$. Since $|Z|\le v(X)$ and $\Delta(X(S))\le \Delta(X)$, we find that $\Delta_S\le \Delta$. By Theorem~\ref{thm:MLRT}(1), we have that 
$$\Delta_{\mathcal{H}_{S,\sigma_{S}}}(T) \le C' \cdot \Delta_S \le C\cdot \Delta$$
since $C'\le C$ and $\Delta_S \le \Delta$. Since $|T\setminus Y|=i$, it follows that
$$\Delta_{\mathcal{H}}(T) = \Delta_{\mathcal{H}_{S,\sigma_{S}}}(T)\le C\cdot \Delta,$$
and hence (1) holds. 

Note that by construction, we have that $R$ is $(r-1)$-flat to $Y$ and for all $H\in\mathcal{H}$, we have that $|V(H)\setminus Z|\le i$. Hence (3) holds. 

Also, $X'$ is an $r$-uniform multi-hypergraph and $X'$ is $(i-1)$-flat to $Z$ by construction. Hence (4) holds. Note this implies that $|X'(T)|=0$ for $T\in\binom{V(X)}{r-1}$ with $|T\setminus Z|\ge i$.

Thus we have shown (1),(3),(4) hold for all choice of $\sigma_S$ (and no longer need to show (6)), it remains to show that (2) and (5) hold for some choice of permutations $\sigma_S$.

Let $T\in \binom{V(X)}{r-1}$ with $|T\setminus Y|<i$ and let $j:= |T\setminus Z|$. Let 
$$Z_T:= \left\{S\in \binom{X\setminus Z}{i}: \Delta_{H_{S,\sigma_S}}(T\setminus S)\ne 0\right\}.$$
For an $i$-set $S$ in $X\setminus Z$, since $|T\setminus Y| < i$, we find that $|T\setminus S|\ge r-i$ and hence by Theorem~\ref{thm:MLRT}(3), $\Delta_{\mathcal{H}_{S,\sigma_S}}(T\setminus S)\le C'$. Thus we have that
$$\Delta_{\mathcal{H}}(T) \le C'\cdot |Z_T|.$$
Similarly by Theorem~\ref{thm:MLRT}(5), it follows that
$$|X'(T)|\le |X(T)|+C'\cdot |Z_T|.$$

Next we need the following claim.

\begin{claim}\label{claim:Z}
$$\Expect{|Z_T|} \le \mu:= (C')^{r-1} \cdot \max\left\{\left(\frac{v(X)}{|Y|}\right)^{i-1} \cdot \Delta,~\left(\frac{v(X)}{|Y|}\right)^{\min\{i,r-1\}}\right\}.$$
\end{claim}
\begin{proofclaim}
We consider two cases as follows. First suppose $|T\cap Y|=r-i$. Then $\Delta_{\mathcal{H}_{S,\sigma_S}}(T\setminus S)=0$ unless $T\setminus Z \subseteq S$ and $(T\cap Y)\cup S=V(f)$ for some $f\in X$. It follows that the number of $S$ for which this value is nonzero is at most $v(X)^{i-j-1} \cdot \Delta(X)$. Furthermore, for this to be nonzero, we also need $T\cap (Z\setminus Y)$ (which has size $i-j-1$) to be the image of the at most $C'$ many elements counted by $\Delta_{\mathcal{H}_S}(T\cap Y)$. 

Thus we find that
$$\Expect{|Z_T|} \le v(X)^{i-j-1} \cdot \Delta(X)\cdot \left(\frac{C'}{|Z|-|Y|}\right)^{i-j-1} \le (C')^{i-1} \cdot \left(\frac{v(X)}{|Y|}\right)^{i-1} \cdot \Delta(X),$$
as desired, where we used that $j\ge 0$ and $|Z|\ge 2 |Y|$.

So we assume that $|T\cap Y| < r-i$ since $\Delta_{\mathcal{H}}(T)=0$ if $|T\cap Y|>r-i$ since $|V(H)\cap Y|\le r-i$ for all $H\in\mathcal{H}$. Note then we assume that $i<r$ as otherwise there is nothing to show. The number of $i$-sets $S$ in $X\setminus Z$ with $T\setminus Z \subseteq S$ is at most $v(X)^{i-j}$. Once again for $\Delta_{\mathcal{H}_{S,\sigma_S}}(T\setminus S)$ to be nonzero, we also need $T\cap (Z\setminus Y)$ (which has size at least $i-j$) to be the image of the at most $C'$ many elements counted by $\Delta_{\mathcal{H}_S}(T\cap Y)$. 

Thus we find that
$$\Expect{\Delta_{\mathcal{H}}(T)} \le v(X)^{i-j} \cdot \left(\frac{C'}{|Z|-|Y|}\right)^{i-j} \le (C')^{r-1} \cdot \left(\frac{v(X)}{|Y|}\right)^{\min\{i,r-1\}},$$
as desired, where we used that $i\le r-1$, $j\ge 0$ and $|Z|\ge 2 |Y|$.
\end{proofclaim}

We note that $|Z_T|$ is the sum of independent Bernoulli random variables. Hence by the Chernoff bounds, we find that

$$\Prob{|Z_T| > 2\mu} \le e^{-\frac{\mu}{3}} < \frac{1}{2\cdot v(X)^{r-1}},$$
since $\mu \ge C'\cdot \Delta \ge 6(r-1)\log v(X)$ as $C'\ge 1$, $r\ge 2$ and $v(X)\ge |Y|$.

Let $A_T$ be the event that $\Delta_{\mathcal{H}}(T) > C \cdot \max\left\{\left(\frac{v(X)}{|Y|}\right)^{i-1} \cdot \Delta,~\left(\frac{v(X)}{|Y|}\right)^{\min\{i,r-1\}}\right\}$. 
Using that $C\ge 2(C')^{r}$, 
we find that $\Prob{A_T} \le \Prob{|Z_T|>2\mu} < \frac{1}{2\cdot v(X)^{r-1}}$. 

Similarly let $B_T$ be the event that $|X'(T)| > C \cdot \max\left\{\left(\frac{v(X)}{|Y|}\right)^{i-1} \cdot \Delta,~\left(\frac{v(X)}{|Y|}\right)^{\min\{i,r-1\}}\right\}+|X(T)|$. 
Using that $C\ge 2(C')^{r}$, 
we find that $\Prob{B_T} \le \Prob{|Z_T|>2\mu} < \frac{1}{2\cdot v(X)^{r-1}}$. 

Since the number of $T$ is at most $\binom{v(X)}{r-1} \le v(X)^{r-1}$, we find by the union bound that the probability that none of the $A_T$ and none of the $B_T$ happen is positive. Hence there exists a choice of permutations $\sigma_S$ such that (2) and (5) hold as desired.
\end{lateproof}

\subsection{Proof of Refine Down Lemma}

Now we prove the Refine Down Lemma.

\begin{lateproof}{lem:RefineDown}
Let $C:= \left(3\cdot C'\right)^{3^r\cdot (r+2)}$ where $C'$ is as $C_{\ref{lem:RefineDownOneStep}}$ in Lemma~\ref{lem:RefineDownOneStep} for $q$ and $r$. 
 For $i\in [r]_0$, let $C_i := \left(3\cdot C'\right)^{3^{r-i}\cdot (r+2)}$. 

Choose subsets $Y_r\subseteq Y_{r-1} \subseteq Y_{r-2} \subseteq \ldots \subseteq Y_0\subseteq Y$ such that $|Y_r| \ge C'$, $|Y_0| \ge \frac{|Y|}{2C'}$ and $|Y_i|= 2\cdot C'\cdot |Y_{i+1}|$ for all $i\in [r-1]_0$. Note that such a choice exists since $|Y|\ge C \ge (2\cdot C')^{r+2}$.

Choose $R$ to be a $C_i$-refined $K_q^r$-multi-refiner of $X$ with remainder $X'$ and refinement family $\mathcal{H}$ such that $i\in [r]_0$ and all of the following hold:

\begin{enumerate}
    \item[$(1)_i$] $\Delta_{\mathcal{H}}(T) \le C_i \cdot \left(\frac{v(X)}{|Y|}\right)^{\binom{r}{2}-\binom{i}{2}} \cdot \Delta+ C_i \cdot |Y_i|$ if $T\in \binom{V(X)}{r-1}$ with $|T\setminus Y_i| > i$,
    \item[$(2)_i$] $\Delta_{\mathcal{H}}(T) \le C_i \cdot \left(\frac{v(X)}{|Y|}\right)^{\binom{r}{2}-\binom{i}{2}} \cdot \Delta$ if $T\in \binom{V(X)}{r-1}$ with $|T\setminus Y_i|\le i$,
    \item[$(3)_i$] $R$ is $(r-1)$-flat to $Y_i$ 
    \item[$(4)_i$]   $X'$ is $i$-flat to $Y_{i}$,    
    \item[$(5)_i$] $\Delta(X')\le C_i\cdot \left(\frac{v(X)}{|Y|}\right)^{\binom{r}{2}-\binom{i}{2}} \cdot \Delta$
    \item[$(6)_i$] $X'$ is an $r$-uniform multi-hypergraph with multiplicity at most $M(q,r)$,
\end{enumerate}
and subject to all of the above such that $i$ is minimized, and subject to that $e(X'\cap R)$ is minimized. Note such a choice exists since $i=r$, $R:=\emptyset$ and $X':=X$ satisfy the conditions. Since $e(X'\cap R)$ is minimized, we find that for every $e\in X'\cap R$, there exists $H\in\mathcal{H}$ such that $e\in H$, for otherwise, $R\setminus \{e\}$, $\mathcal{H}$, and $X'\setminus \{e\}$ contradict the choice of $R$, $\mathcal{H}$ and $X'$.

First suppose $i=0$. We claim that $R$ is as desired. To see this, note that (1) follows $(1)_0$ and $(2)_0$; (2) follows from $(3)_0$; and (3) follows from $(5)_0$ and (4) follows from $(4)_0$ and $(6)_0$.

So we assume that $i>0$. Let $\Delta':= \max\{\Delta(X'),~v(X')^{1-\frac{1}{r}}\cdot \log v(X')\}$. Hence by Lemma~\ref{lem:RefineDownOneStep} (applied to $X',Y_i,Y_{i-1})$, there exists a $C'$-refined $K_q^r$-multi-refiner $R'$ of $X'$ with remainder $X''$ and refinement family $\mathcal{H}'$ such that all of the following hold:

\begin{enumerate}
    \item[($1^{\prime}$)] $\Delta_{\mathcal{H'}}(T) \le C' \cdot |Y_{i-1}|$ if $T\in \binom{V(X)}{r-1}$ with $|T\setminus Y_i|=i$,
    \item[($2^{\prime}$)] $\Delta_{\mathcal{H'}}(T) \le C' \cdot \max\left\{\left(\frac{v(X)}{|Y_i|}\right)^{i-1} \cdot \Delta',~\left(\frac{v(X)}{|Y_i|}\right)^{\min\{i,r-1\}}\right\}$ if $T\in \binom{V(X)}{r-1}$ with $|T\setminus Y_i|<i$,
    \item[($3^{\prime}$)] $R'$ is $(r-1)$-flat to $Y_{i-1}$ and for each $H\in \mathcal{H}'$, we have that 
    $|V(H)\setminus Y_{i-1}| \le i$,
    \item[($4^{\prime}$)]   $X''$ is $(i-1)$-flat to $Y_{i-1}$,
    \item[($5^{\prime}$)] $\Delta(X'')\le \Delta(X')+C' \cdot \max\left\{\left(\frac{v(X)}{|Y_i|}\right)^{i-1} \cdot \Delta',~\left(\frac{v(X)}{|Y_i|}\right)^{\min\{i,r-1\}}\right\}$,
    \item[($6^{\prime}$)] $X''$ is an $r$-uniform multi-hypergraph with multiplicity at most $M(q,r)$. 
\end{enumerate}          

Note we assume that $R$ and $R'$ are edge-disjoint by construction (this is possible as they are multi-hypergraphs). Let $R'':= R\cup R'$ and $\mathcal{H}'':=\mathcal{H}\cup\mathcal{H}'$.
By Proposition~\ref{prop:Concatenate}, we find that $R''$ is a $(C_i+C')$-refined $K_q^r$-multi-refiner of $X$ with remainder $X''$ and refinement family $\mathcal{H}''$. Since $C_{i-1}\ge C_i+C'$, we find that $R''$ is $C_{i-1}$-refined.

We will show that all of $(1)_{i-1}$-$(6)_{i-1}$ hold for $R''$, $\mathcal{H}''$ and $X''$, contradicting the choice of $R$, $\mathcal{H}$ and $X'$. First note that $(4)_{i-1}$ and $(6)_{i-1}$ hold for $X''$ by $(4')$ and $(6')$ respectively. Similarly $(3)_{i-1}$ holds for $R''$ and $\mathcal{H}''$ since $(3)_i$ holds for $R$ and $\mathcal{H}$ and $(3')$ holds for $R'$ and $\mathcal{H}'$.

So it remains to show that $(1)_{i-1}$, $(2)_{i-1}$ and $(5)_{i-1}$ hold for $R''$, $\mathcal{H}''$ and $X''$. 

First we show that $(5)_{i-1}$ holds for $X''$. First note that since $\Delta(X')\le C_i\cdot \left(\frac{v(X)}{|Y|}\right)^{\binom{r}{2}-\binom{i}{2}} \cdot \Delta$ by $(5)_i$ and $v(X)=v(X')$, we find that 
$$\Delta' \le C_i\cdot \left(\frac{v(X)}{|Y|}\right)^{\binom{r}{2}-\binom{i}{2}} \cdot \Delta.$$
Thus
$$\Delta(X'')\le C_i \cdot \left(\frac{v(X)}{|Y|}\right)^{\binom{r}{2}-\binom{i}{2}} \cdot \Delta+ C' \cdot \max\left\{\left(\frac{v(X)}{|Y_i|}\right)^{i-1} \cdot \Delta',~\left(\frac{v(X)}{|Y_i|}\right)^{\min\{i,r-1\}}\right\}.$$
Substituting the upper bound for $\Delta'$, we find that
\begin{align*}
\Delta(X'') \le~~&C_i \cdot \left(\frac{v(X)}{|Y|}\right)^{\binom{r}{2}-\binom{i}{2}} \cdot \Delta+\\
&C' \cdot \max\left\{\left(\frac{v(X)}{|Y_i|}\right)^{i-1} \cdot C_i \cdot \left(\frac{v(X)}{|Y|}\right)^{\binom{r}{2}-\binom{i}{2}} \cdot \Delta,~\left(\frac{v(X)}{|Y_i|}\right)^{\min\{i,r-1\}}\right\} \\
\le~~&C_{i-1} \cdot \left(\frac{v(X)}{|Y|}\right)^{\binom{r}{2}-\binom{i-1}{2}} \cdot \Delta(X),
\end{align*}
where we used that $v(X)\ge |Y|$, $|Y_i|\ge \frac{|Y|}{(2C')^{i+1}}$ and $C_{i-1} \ge C_i + C' \cdot (2\cdot C')^{i+1}\cdot C_i$ (since $C_{i-1} \ge (2\cdot C')^{r+2} \cdot C_i^2$). Note we also used that $\max\{\binom{r}{2}-\binom{i}{2}+(i-1),~\min\{i,r-1\}\} \le \binom{r}{2}-\binom{i-1}{2}$. Hence $(5)_{i-1}$ holds for $X''$.

Next we show that $(1)_{i-1}$ and $(2)_{i-1}$ hold for $\mathcal{H}''$. To that end, let $T\in \binom{V(X)}{r-1}$. First suppose that $|T\setminus Y_i| > i$. Then $\Delta_{\mathcal{H'}}(T)=0$ since by (3') we have for all $H\in\mathcal{H}'$ that $|V(H)\setminus Y_{i-1}| \le i$. Thus by $(1)_i$ for $\mathcal{H}$, we have that
\begin{align*}\Delta_{\mathcal{H}''}(T) &\le \Delta_{\mathcal{H}}(T) + \Delta_{\mathcal{H'}}(T) = \Delta_{\mathcal{H}}(T) \le  C_i \cdot \left(\frac{v(X)}{|Y|}\right)^{\binom{r}{2}-\binom{i}{2}} \cdot \Delta+ C_i \cdot |Y_i|\\
&\le C_{i-1} \cdot \left(\frac{v(X)}{|Y|}\right)^{\binom{r}{2}-\binom{i-1}{2}} \cdot \Delta+ C_{i-1} \cdot |Y_{i-1}|,
\end{align*}
where we used that $C_i\le C_{i-1}$, $|Y_i|\le |Y_{i-1}|$ and $|Y|\le v(X)$.

Next suppose that $|T\setminus Y|=i$. Then by $(2)_i$ for $\mathcal{H}$ and ($1^{\prime}$) for $\mathcal{H}'$, we find that
\begin{align*}\Delta_{\mathcal{H}''}(T) &\le \Delta_{\mathcal{H}}(T) + \Delta_{\mathcal{H'}}(T) \le  C_i \cdot \left(\frac{v(X)}{|Y|}\right)^{\binom{r}{2}-\binom{i}{2}} \cdot \Delta+ C' \cdot |Y_{i-1}|\\
&\le C_{i-1} \cdot \left(\frac{v(X)}{|Y|}\right)^{\binom{r}{2}-\binom{i-1}{2}} \cdot \Delta+ C_{i-1} \cdot |Y_{i-1}|,
\end{align*}
where we used that $C_{i-1}\ge \max\{C_{i}, C'\}$, $|Y_i|\le |Y_{i-1}|$ and $|Y|\le v(X)$. Combining the two cases above, we find that $(1)_{i-1}$ holds for $\mathcal{H}''$. 

Finally suppose that $|T\setminus Y_i| < i$. Then by $(2)_i$ for $\mathcal{H}$ and by (2') for $\mathcal{H}'$, we find that \begin{align*}
\Delta_{\mathcal{H}''}(T) &\le \Delta_{\mathcal{H}}(T) + \Delta_{\mathcal{H'}}(T) \\
&\le  C_i \cdot \left(\frac{v(X)}{|Y|}\right)^{\binom{r}{2}-\binom{i}{2}} \cdot \Delta+ C' \cdot \max\left\{\left(\frac{v(X)}{|Y_i|}\right)^{i-1} \cdot \Delta',~\left(\frac{v(X)}{|Y_i|}\right)^{\min\{i,r-1\}}\right\}\\
&\le C_{i-1} \cdot \left(\frac{v(X)}{|Y|}\right)^{\binom{r}{2}-\binom{i-1}{2}} \cdot \Delta(X),
\end{align*}
where we have the same calculation as we had before for $\Delta(X'')$. Hence $(2)_{i-1}$ holds for $\mathcal{H}''$.
\end{lateproof}

\section{Embedding Proofs}\label{s:Embedding}

In this section, we conclude the proofs of the main theorems about refiners and omni-absorbers by embedding absorbers and fake edges as follows. As mentioned, these proofs will proceed by the deterministic `avoid the bad sets' style of proof we already used in the proof of Lemma~\ref{lem:SpecialSets}.

\subsection{Proof of the Refined Multi-Omni-Absorber Theorem}

First we prove the Refined Multi-Omni-Absorber Theorem (Theorem~\ref{thm:MultiOmni}).

\begin{lateproof}{thm:MultiOmni}
Let $C$ be chosen large enough as needed throughout the proof. Let $C'$ be is as in Theorem~\ref{thm:MRT} for $q$ and $r$. By Theorem~\ref{thm:MRT} applied to $X$ and $Y$, there exists a $C'$-refined $K_q^r$-multi-refiner $R$ of $X$ with refinement family $\mathcal{H}'$ such that both of the following hold:
\begin{enumerate}
    \item[($1^{\prime}$)] $\Delta(\mathcal{H}') \le C' \cdot \Delta$, and
    \item[($2^{\prime}$)] $R$ is $(r-1)$-flat to $Y$. 
\end{enumerate}

By Theorem~\ref{thm:AbsorberExistence}, there exists a $K_q^r$-absorber $A_H$ for each $H\in \mathcal{H}'$ and $K_q^r$ decompositions $\mathcal{Q}_{1,H}$ of $A_H$ and $\mathcal{Q}_{2,H}$ of $A_H\cup H$. Let $k: = \max\{ v(A_H): H\in \mathcal{H}'\}$. 

\begin{claim}\label{claim:A}
There exists a collection $\mathcal{A}=(A_H': H\in \mathcal{H}')$ where for each $H\in \mathcal{H}'$, $A_H'$ is a copy of $A_H$ containing $H$ such that $A_H'\setminus V(H) \subseteq Y\setminus V(H)$, and letting $A:= R~\cup~ \bigcup_{H\in \mathcal{H}'} A_H'$, we have that $\Delta(A)\le C\cdot \Delta$.    
\end{claim}
\begin{proof}
Choose $\mathcal{H}''\subseteq \mathcal{H}'$ and $\mathcal{A}' = (A_H': H\in \mathcal{H}'')$ where for each $H\in \mathcal{H}''$, $A_H'$ is a copy of $A_H$ containing $H$ such that $A_H'\setminus V(H) \subseteq Y\setminus V(H)$ as follows: Define $A':= R~\cup~ \bigcup_{H\in \mathcal{H}''} A_H'$. For all $\ell\in [r-1]$, $U\in \binom{V(X)}{\ell}$ and $J\in \binom{V(X)\setminus U}{r-1-\ell}$, define 
$$d_{A'}(J,U) := |\{ H\in \mathcal{H}'': U\subseteq V(A_H')\setminus V(H),~J\subseteq V(H)\}|.$$
Now choose $\mathcal{H}''$ and $\mathcal{A}'$ such that for all such $U$ and $J$,
$$d_{A'}(J,U)\le m:=\frac{C}{k\cdot 2^r}\cdot \Delta,$$
and subject to those conditions, $|\mathcal{H}''|$ is maximized. Note such a choice exists since $\mathcal{H}'':=\emptyset$ satisfies the conditions trivially. 

First suppose that $|\mathcal{H}''|=|\mathcal{H}'|$. It follows that for $S\in \binom{V(X)}{r-1}$
$$|A'(S)| \le k\cdot \Delta(\mathcal{H}') + \sum_{J\subsetneq S} k\cdot d_{A'}(J,S\setminus J) \le k\cdot C'\cdot \Delta + k\cdot (2^r-1) \cdot \left(\frac{C}{k\cdot 2^r}\cdot \Delta\right) \le C\cdot \Delta,$$
where we used that $C\ge k\cdot 2^r\cdot C'$ since $C$ is large enough. Hence $\Delta(A')\le C\cdot \Delta$ and so $\mathcal{A}'$ is as desired.

So we assume that $|\mathcal{H}''| < |\mathcal{H}'|$. Hence there exists $H\in \mathcal{H}'\setminus \mathcal{H}''$. Let $Z:= Y\setminus V(H)$. Since $v(X)\ge C$ and $C$ is large enough, we have that $m\ge 2$ and hence $m-1\ge \frac{m}{2} = \frac{C}{2\cdot k \cdot 2^r} \cdot \Delta$.  For $\ell \in [r-1]$, let
$$B_{\ell} := \left\{ U\in \binom{Z}{\ell}: \exists~J \subseteq V(H) \text{ with } |J|=r-1-\ell \text{ such that } d_{A'}(J,U) \ge  m-1 \right\}.$$
Note that for $J\subseteq V(H)$ with $|J|=r-1-\ell$,
\begin{align*}
\left|\left\{ U\in \binom{Z}{\ell}: d_{A'}(J,U)\ge m-1\right\}\right| &\le \frac{|\{H \in \mathcal{H}'': J\subseteq V(H)\}| \cdot \binom{k}{\ell}}{m-1} \le \frac{\Delta(\mathcal{H}')\cdot v(G)^{\ell}\cdot \binom{k}{\ell}}{m-1} \\
&\le \frac{C'\cdot \Delta\cdot \binom{k}{\ell}}{\frac{C}{2\cdot k\cdot 2^r}\cdot \Delta} \cdot v(G)^{\ell} \le \frac{k^r\cdot 8^r\cdot C'}{C}\cdot \binom{|Z|}{\ell},
\end{align*}
where we used that $\Delta(\mathcal{H}'')\le \Delta(\mathcal{H}')\le C'\cdot \Delta$ and that $|Z|\ge \frac{v(X)}{2}-r\ge \frac{v(X)}{4}$ since $v(G)\ge C \ge 4r$. Note there are at most $\binom{v(H)}{r-1-\ell}$ subsets $J$ of $V(H)$ of size $r-1-\ell$ (which is at most $\binom{C'}{r-1-\ell}$ since $R$ is $C'$-refined and hence $v(H)\le C'$). 

Note that 
$$\sum_{\ell \in [r-1]} \frac{1}{C}\cdot \binom{C'}{r-1-\ell} \cdot k^r\cdot 8^r \cdot C'\cdot \binom{k}{\ell} < 1,$$ 
since $C$ is large enough. Hence by Proposition~\ref{prop:NonUniformTuran}, there exists $S\subseteq Z$ with $|S| = k$ such that $S$ contains no element of $\bigcup_{\ell \in [r-1]} B_{\ell}$. Let $A_H'$ be a copy of $A_H$ containing $H$ such that $V(A_H')\subseteq V(H)\cup S$. Let $\mathcal{H}_0:=\mathcal{H}''\cup \{H\}$ and $\mathcal{A}'':=(A_H': H\in \mathcal{H}_0)$. Define $A'':= R\cup \bigcup_{H\in\mathcal{H}_0} A_H'$.

We note that for all $\ell\in [r-1]$, $U\in \binom{V(X)}{\ell}$ and $J\in \binom{V(X)\setminus U}{r-1-\ell}$, we have that $d_{A''}(J,U)=d_{A'}(J,U)$ unless $J\subseteq V(H)$ and $U\subseteq S$ and in that case we have $d_{A''}(J,U)= d_{A'}(J,U)+1$; in the former case, we have that $d_{A'}(J,U)\le m$ by assumption and in the latter case we find that $d_{A'}(J,U)\le m-1$ since $S$ does not contain an element of $\bigcup_{\ell \in [r-1]} B_{\ell}$. It follows that in either case $d_{A''}(J,U)\le m$, contradicting the choice of $\mathcal{H}''$ and $\mathcal{A}'$.
\end{proof}

Let $\mathcal{A}$ and $A$ be as in the claim above. We show that $A$ is a $C$-refined $K_q^r$-multi-omni-absorber for $X$ such that (1) and (2) hold as desired. 

First we define a decomposition family $\mathcal{H}$. Note that $A$ is an $r$-uniform hypergraph with $V(A)=V(X)$. Further note that $A$ and $X$ are edge-disjoint (which is possible since they are multi-hypergraphs). For each $H\in \mathcal{H}'$, let $\mathcal{Q}_{1,H}'$ and $\mathcal{Q}_{2,H}'$ be the $K_q^r$ decompositions of $A_H'$ corresponding to $\mathcal{Q}_{1,H}$ and $\mathcal{Q}_{2,H}$ respectively. Let $\mathcal{H} := \bigcup_{H\in \mathcal{H}'} \mathcal{Q}_{1,H}'\cup \mathcal{Q}_{2,H}'$. 

Now we prove that $A$ is $C$-refined. To that end, let $e\in A\cup X$. First suppose that $e\in A\setminus R = \bigcup_{H\in \mathcal{H}'} A_H'$. Since the $A_H'$ are pairwise edge-disjoint, we have that $e$ is in $A_H'$ for only one $H\in \mathcal{H}'$. Since $A_H'$ is edge-disjoint from $R\cup X$, we find that $|\mathcal{H}(e)|\le 2$ since $e$ is only in $\mathcal{Q}_{1,H}'$ and $\mathcal{Q}_{2,H}'$.

So we assume that $e\in R\cup X$. Since $R$ is $C'$-refined, we find that $|\mathcal{H}'(e)|\le C'$. But then $|\mathcal{H}(e)|\le C'$ since $e$ is only in $\mathcal{Q}_{2,H}'$ for each such $H$. Since $C\ge \max\{2,C'\}$ since $C$ is large enough, we have that $A$ is $C$-refined.

Next we prove that $A$ is a $K_q^r$-multi-omni-absorber. First note that by construction, $A_H\cap X = \emptyset$ for all $H\in \mathcal{H}'$. Then by definition of $K_q^r$-multi-refiner, we have that $|H\cap X|\le 1$ for all $H\in \mathcal{H}'$. Hence $|Q\cap X|\le 1$ for all $Q\in \mathcal{H}$.

Now let $L$ be a $K_q^r$-divisible subgraph of $X$. Since $R$ is a $K_q^r$-refiner of $X$, there exists $\mathcal{H}_L \subseteq \mathcal{H}'$ that decomposes $L\cup R$.
We let
$$\mathcal{Q}_L:= \bigcup_{H\in \mathcal{H}_L} \mathcal{Q}_{2,H}'~~\cup \bigcup_{H\in \mathcal{H'}\setminus \mathcal{H}_L} \mathcal{Q}_{1,H}'.$$
It follows that $\mathcal{Q}_L$ is a $K_q^r$-decomposition of $L\cup A$. Hence $A$ is a $K_q^r$-multi-omni-absorber for $X$ with decomposition family $\mathcal{H}$.

Next we prove that (1) and (2) hold for $A$. By Claim~\ref{claim:A}, (1) holds for $A$. So finally we show that (2) holds for $A$. By $(2^{\prime})$, $R$ is $(r-1)$-flat to $Y$. Recall that for each $H\in \mathcal{H}'$, we have that $V(H)$ is independent in $A_H$ by definition of $K_q^r$-absorber and hence $|e\cap V(H)|\le r-1$ for all $e\in E(A_H)$. Thus we find that $|e\setminus Y| \le r-1$ for all $e\in E(A_H)$. Hence, it follows that $A$ is $(r-1)$-flat to $Y$. Thus (2) holds for $A$. 
\end{lateproof}

\subsection{A Partial Clique Embedding Lemma}\label{ss:PartialClique}

It remains to prove Theorems~\ref{thm:Refiner} and~\ref{thm:Omni}. The first will follow from Theorem~\ref{thm:MRT} by embedding `fake edges'. The second will then follow by embedding private absorbers. Both proofs will follow by proving we may embed disjoint `partial cliques' (wherein we may embed either a fake edge or an absorber). This motivates the following definition.

\begin{definition}
Let $k> r\ge 1$ be integers. Let $G$ be an $r$-uniform hypergraph. If $S\subseteq V(G)$ with $k > |S|\ge r$, then a \emph{partial $K_k^r$ rooted at $S$} in $G$ is a subgraph $T$ of $G$ with $S\subseteq V(T)$, $|T|=k$ and $\binom{V(T)}{r}\setminus \binom{S}{r} \subseteq E(G)$.     
\end{definition}

Using Proposition~\ref{prop:NonUniformTuran}, we prove the following embedding lemma (which allows us to embed the `fake edges' but also absorbers for refiners in the next subsection as well). The proof is similar to that given for Theorem~\ref{thm:MultiOmni} except that we require the partial cliques to actually be edge-disjoint (since the setting is simple hypergraphs rather multi-hypergraphs). Note that in the following lemma $G$ is a simple hypergraph.

\begin{lem}[Partial Clique Embedding Lemma]\label{lem:EmbedMinDegree}
For all integers $h \ge r \ge 1$ and $k\ge 1$, there exist an integer $C\ge 1$ and reals $\varepsilon, \gamma \in (0,1)$ such that the following holds: Let $G$ be an $r$-uniform hypergraph with $v(G)\ge C$ and $\delta(G)\ge (1-\varepsilon)\cdot v(G)$. If $Y$ is an $h$-bounded multi-hypergraph with $V(Y)=V(G)$ and edge-disjoint from $G$ with $\Delta_{r-1}(Y)\le \gamma \cdot v(G)$, and $|e|\geq r$ for all $e \in E(Y)$, then there exist edge-disjoint subgraphs $(T_e: e\in E(Y))$ of $G$ such that $T_e$ is a partial $K_{k+|e|}^r$ rooted at $V(e)$, and letting $T:= \bigcup_{e\in E(Y)} T_e$, we have that $\Delta(T) \le C\cdot \Delta_{r-1}(Y)$.
\end{lem}
\begin{proof}
Let $C$ be chosen large enough as needed throughout the proof. Choose $Y'\subseteq Y$ and $\mathcal{T}' = (T_e\subseteq G: e \in E(Y'))$ such that $T_e$ is a partial $K_{k+|e|}^r$ rooted at $V(e)$ as follows: Define $T':= \bigcup_{e\in E(Y')} T_e$. For all $\ell\in [r-1]$, $U\in \binom{V(G)}{\ell}$ and $J\in \binom{V(G)\setminus U}{r-1-\ell}$, define 
$$d_{T'}(J,U) := |\{ e'\in Y': U\subseteq V(T_{e'})\setminus V(e'),~J\subseteq e'\}|.$$
Now choose $Y'$ and $\mathcal{T}'$ such that for all such $U$ and $J$,
$$d_{T'}(J,U)\le m:=\frac{C}{k\cdot 2^r}\cdot \Delta_{r-1}(Y),$$
and subject to those conditions, $e(Y')$ is maximized. Note such a choice exists since $Y':=\emptyset$ satisfies the conditions trivially. 

First suppose that $e(Y')=e(Y)$. It follows that for $S\in \binom{V(G)}{r-1}$
$$|T'(S)| \le k\cdot \Delta_{r-1}(Y) + \sum_{J\subsetneq S} k\cdot d_{T'}(J,S\setminus J) \le k\cdot \Delta_{r-1}(Y) + k\cdot (2^r-1) \cdot \left(\frac{C}{k\cdot 2^r}\cdot \Delta_{r-1}(Y)\right) \le C\cdot \Delta_{r-1}(Y),$$
where we used that $C\ge k\cdot 2^r$ since $C$ is large enough. Hence $\Delta(T')\le C\cdot \Delta_{r-1}(Y)$ and so $\mathcal{T}'$ is as desired.

So we assume that $e(Y') < e(Y)$. Hence there exists $e\in Y\setminus Y'$. Let $Z:= V(G)\setminus V(e)$. For $\ell \in [r]$, let 
$$A_{\ell} := \left\{U\in \binom{Z}{\ell}: \exists~J\subseteq e \text{ with } |J|=r-\ell \text{ such that } J\cup U \not \in G\setminus T'\right\}.$$
Let $n:= v(G)$. Note that $\Delta(K_n^r\setminus G) \le \varepsilon n$. Hence $\Delta(K_n^r\setminus (G\setminus T')) \le (\varepsilon+C\cdot \gamma)n$. Hence for $J\subseteq e$ with $|J|=r-\ell$, we find that 
\begin{align*}
\left|\left\{ U\in \binom{Z}{\ell}: J\cup U \not\in G\setminus T'\right\}\right| &\le |\{ e' \in K_n^r\setminus (G\setminus T'): J\subseteq e'\}| \le \Delta(K_n^r\setminus (G\setminus T'))\cdot v(G)^{\ell-1} \\
&\le (\varepsilon+C\cdot \gamma) \cdot v(G)^{\ell} \le (\varepsilon+C\cdot \gamma)\cdot 4^r\cdot \binom{|Z|}{\ell},
\end{align*}
where we used that $|Z|\ge \frac{v(G)}{2}-r\ge \frac{v(G)}{4}$ since $v(G)\ge C \ge 4r$. Note there are at most $\binom{|e|}{r-\ell}$ subsets $J$ of $e$ of size $r-\ell$ (which is at most $\binom{h}{r-\ell}$ since $|e|\le h$).

Note we assume that $\Delta_{r-1}(Y)\ge 1$ as otherwise the outcome of the lemma holds trivially. Since $C\ge 2k\cdot 2^r$ as $C$ is large enough, it follows that $m\ge 2$ and hence $m-1\ge \frac{m}{2} = \frac{C}{2\cdot k\cdot 2^r}\cdot \Delta_{r-1}(Y)$.  For $\ell \in [r-1]$, let
$$B_{\ell} := \left\{ U\in \binom{Z}{\ell}: \exists~J \subseteq e \text{ with } |J|=r-1-\ell \text{ such that } d_{T'}(J,U) \ge  m-1 \right\}.$$ 
Note that for $J\subseteq e$ with $|J|=r-1-\ell$,
\begin{align*}
\left|\left\{ U\in \binom{Z}{\ell}: d_{T'}(J,U)\ge m\right\}\right| &\le \frac{|\{e' \in Y': J\subseteq e'\}| \cdot \binom{k}{\ell}}{m-1} \le \frac{\Delta_{r-1}(Y)\cdot v(G)^{\ell}\cdot \binom{k}{\ell}}{m-1} \\
&\le \frac{2\cdot k\cdot 2^r\cdot \binom{k}{\ell}}{C} \cdot v(G)^{\ell} \le \frac{k^r\cdot 8^r}{C}\cdot \binom{|Z|}{\ell},
\end{align*}
where we used that $|Z|\ge \frac{v(X)}{2}-r\ge \frac{v(X)}{4}$ since $v(G)\ge C \ge 4r$. Note there are at most $\binom{|e|}{r-1-\ell}$ subsets $J$ of $e$ of size $r-1-\ell$ (which is at most $\binom{h}{r-1-\ell}$). 

Note that 
$$\sum_{\ell \in [r]} (\varepsilon +C\cdot \gamma)\cdot 4^r\cdot \binom{k}{\ell} + \frac{1}{C}\cdot \binom{h}{r-1-\ell} \cdot k^r\cdot 8^r \cdot \binom{k}{\ell} < 1,$$ 
since $C$ is large enough and $\gamma$ and $\varepsilon$ are small enough. Hence by Proposition~\ref{prop:NonUniformTuran}, there exists $S\subseteq Z$ with $|S| = k$ such that $S$ contains no element of $\bigcup_{\ell \in [r-1]} A_{\ell}\cup B_{\ell}$. Let $T_e := \binom{e\cup S}{r} \setminus \binom{e}{r}$, let $Y'':=Y'\cup e$ and $\mathcal{T}'':=(T_e: e\in Y'')$. Note that $T_e$ is a partial $K_{k+|e|}^r$ rooted at $V(e)$. Since $S$ does not contain an element of $\bigcup_{\ell \in [r-1]} A_{\ell}$, it follows that $T_e$ is a subgraph of $G$ that is edge-disjoint from all elements of $\mathcal{T}'$. Let $T'' := \bigcup_{e\in E(Y'')} T_e$.

We note that for all $\ell\in [r-1]$, $U\in \binom{V(G)}{\ell}$ and $J\in \binom{V(G)\setminus U}{r-1-\ell}$, we have that $d_{T''}(J,U)=d_{T'}(J,U)$ unless $J\subseteq e$ and $U\subseteq S$ and in that case we have $d_{T''}(J,U)= d_{T'}(J,U)+1$; in the former case, we have that $d_{T'}(J,U)\le m$ by assumption and in the latter case we find that $d_{T'}(J,U)\le m-1$ since $S$ does not contain an element of $\bigcup_{\ell \in [r-1]} B_{\ell}$. It follows that in either case $d_{T''}(J,U)\le m$, contradicting the choice of $Y'$ and $\mathcal{T}'$.
\end{proof}

\subsection{Proof of the Refiner Theorem via Embedding Fake Edges}\label{ss:Basic}

Now that we have proved Theorem~\ref{thm:MLRT}, we next aim to use it to prove Theorem~\ref{thm:Refiner}. We need to replace the multiedges in the multi-refiner with `fake edges'. To that end, we need the following basic gadgets.

\begin{definition}\label{def:BasicGadget}
Let $q>r\ge 1$ be integers. Let $S$ be a set of vertices of size $r$.
\begin{itemize}
    \item An \emph{anti-edge on $S$}, denoted ${\rm AntiEdge}_q^r(S)$, is a set of new vertices $x_1,\ldots, x_{q-r}$ together with edges $\binom{S\cup \{x_i:i\in[q-r]\} }{r} \setminus \{S\}$.
    \item A \emph{fake edge on $S$}, denoted ${\rm FakeEdge}_q^r(S)$, is a set of new vertices $x_1,\ldots, x_{q-r}$ together with  $\{ {\rm AntiEdge}_q^r(T): T\in \binom{S\cup \{x_i:i\in[q-r]\} }{r} \setminus \{S\} \}$.
\end{itemize}    
\end{definition}

Note that all the graphs in Definition~\ref{def:BasicGadget} have at most $(q-r)\cdot \binom{q}{r} \le q^{r+1}$ vertices. Next we collect basic facts about the divisibility of the degrees in these basic gadgets as follows.

\begin{fact}\label{fact:Div}
Let $q> r \ge 1$ be integers. The following hold:
\begin{itemize}
    \item[(1)] If $F={\rm AntiEdge}_q^r(S)$, then for every $0\le i \le r-1$ and $S'\subseteq S$ with $|S'|=i$, we have $d_F(S') \equiv -1 \mod \binom{q-i}{r-i}$.
    \item[(2)] If $F={\rm FakeEdge}_q^r(S)$, then for every $0\le i \le r-1$ and $S'\subseteq S$ with $|S'|=i$, we have $d_F(S') \equiv +1 \mod \binom{q-i}{r-i}$.
\end{itemize}
\end{fact}

Thus Fact~\ref{fact:Div}(1) says that an anti-edge $F$ on $S$ has the negative divisibility properties of an actual edge on $S$. Similarly Fact~\ref{fact:Div}(2) says that the a fake edge $F$ on $S$ has the same divisibility properties as an actual edge on $S$.  Furthermore, we note all of the above basic gadgets defined in Definition~\ref{def:BasicGadget} have at most $(q-r)\cdot \binom{q}{r} \le q^{r+1}$ vertices and hence at most $q^{r(r+1)}$ edges.


\begin{fact}\label{fact:Div2}
Let $q> r \ge 1$ be integers. The following hold:
\begin{itemize}
    \item[(1)] $M(q,r)$ edge-disjoint copies of ${\rm AntiEdge}_q^r(S)$ is a $K_q^r$-divisible hypergraph. 
    \item[(2)] $M(q,r)$ edge-disjoint copies of ${\rm FakeEdge}_q^r(S)$ is a $K_q^r$-divisible hypergraph. 
\end{itemize}
\end{fact}

We are now prepared to prove the Refiner Theorem (Theorem~\ref{thm:Refiner}) by using Theorem~\ref{thm:MRT} and then embedding fake edges using the Partial Clique Embedding Lemma (Lemma~\ref{lem:EmbedMinDegree}) as follows.

\begin{lateproof}{thm:Refiner}

 Note that $X$ is a simple $r$-uniform hypergraph. Let $k:= q^{r+1}$ and $h:= r$ and let $C',\varepsilon',\gamma$ be as in  $C,\varepsilon,\gamma$ for $h,r,k$ in Lemma~\ref{lem:EmbedMinDegree}. Let $C''$ be as Theorem~\ref{thm:MRT} for $q$ and $r$. Let $C:= \max\left\{ \left(\frac{(C')^2}{\gamma}\right)^{r^2},~(1+C''\cdot C') \cdot C'',~ q^{r(r+1)}\cdot C'',~\frac{2}{\varepsilon'}\right\}$ and $\varepsilon := \varepsilon' - \frac{1}{C}$. Note that as $C\ge \frac{2}{\varepsilon'}$ we have that $\varepsilon \ge \frac{\varepsilon'}{2} > 0$. Further note that $$\delta(G\setminus E(X)) \ge \delta(G) - \Delta(X) \ge (1-\varepsilon')\cdot v(G)$$ 
 since $\Delta(X) \le \frac{v(X)}{C}$ and $\varepsilon = \varepsilon'-\frac{1}{C}$.

Note we assume $\Delta(X)\ge 1$ as otherwise there is nothing to show (as otherwise $X$ is empty and the empty refiner satisfies the theorem trivially). Since $\Delta(X)\le \frac{v(X)}{C}$, this implies that $v(X)\ge C$.

Let $\Delta' := \max\left\{\Delta(X),~v(X)^{1-\frac{1}{r}}\cdot \log v(X)\right\}$. Since $v(X)\ge C\ge C''$, we have by Theorem~\ref{thm:MRT} that there exists a $C''$-refined $K_q^r$-multi-refiner $R'$ of $X$ with refinement family $\mathcal{H}'$ such that $\Delta(\mathcal{H}') \le C'' \cdot \Delta'$, and $R'$ is $(r-1)$-flat to $Y$. Note that by definition of multi-refiner, we have for each $H\in \mathcal{H}'$ that $|H\cap X|\le 1$.

 Note since $R$ is $C''$-refined, it follows that
 $$\Delta(R') \le C''\cdot \Delta(\mathcal{H}') \le (C'')^2 \cdot \Delta'\le (C'')^2 \cdot \Delta,$$ 
as $\Delta\ge \Delta'$. Since $v(X)\ge C\cdot \Delta(X)$ and $C$ is large enough, we find that $\Delta \le v(G)$. Then since $C\ge \frac{(C'')^2}{\gamma}$ we find that $\Delta(R')\le \gamma\cdot v(G)$.

By Lemma~\ref{lem:EmbedMinDegree} applied to $Y:=R'$, there exist edge-disjoint subgraphs $(T_e: e\in E(Y))$ of $G\setminus E(X)$ such that $T_e$ is a partial $K_{k+|e|}^r$ rooted at $V(e)$, and letting $T:= \bigcup_{e\in E(Y)} T_e$, we have that $\Delta(T) \le C'\cdot \Delta(Y)$. 

For each $e\in E(Y)$, embed the hypergraph ${\rm FakeEdge}_q^r(V(e))$ inside $T_e$ and denote it by $U_e$. Let $R:= \bigcup_{e\in E(Y)} U_e$. Note that $R$ is a simple $r$-uniform hypergraph and $R\subseteq G\setminus E(X)$. 

For each $H\in\mathcal{H}'$, let $$\phi(H) := (H\cap X) ~\cup \bigcup_{e\in E(H)\setminus X} U_e.$$ Let $\mathcal{H}:= \{\phi(H):H\in\mathcal{H}'\}$. Now we prove that $R$ is a $K_q^r$-refiner of $X$ with refinement family $\mathcal{H}$ such that $\Delta(\mathcal{H})\le C\cdot \Delta$ as desired.

First we show $R$ is $C$-refined. By definition of $C''$-refined, we find that $v(H)\le C''$ for all $H\in\mathcal{H}'$. Moreover, since $v(U_e)\le q^{r+1}$ for all $e\in E(Y)$, it follows that $v(\phi(H)) \le q^{r+1}\cdot C''$ for all $H\in \mathcal{H'}$. Since $C\ge q^{r(r+1)}\cdot C''$, we find that $\max \{v(H),e(H)\} \le C$ for all $H\in\mathcal{H}$. 

Let $f\in X\cup R$. First suppose $f\in R$. Since the $U_e$ are pairwise edge-disjoint by construction, it follows that $f$ is only in one $U_e$ for some $e\in E(Y)$. But then  $f\in \phi(H)$ if and only if $e\in H$. Hence $|\mathcal{H}(f)|\le |\mathcal{H}'(e)| \le C$ since $R'$ is $C$-refined. 

So we assume $f\in X$. Since $R'$ is $C$-refined, we have that $|\mathcal{H}'(f)|\le C$. Thus we find that $|\mathcal{H}(f)|= |\mathcal{H}'(f)|\le C$. Altogether, we have that $R$ is $C$-refined. 

Next we show that $R$ is a $K_q^r$-refiner of $X$ with refinement family $\mathcal{H}$. First note that since $H$ is $K_q^r$-divisible, it follows from Fact~\ref{fact:Div} that $\phi(H)$ is $K_q^r$-divisible. Furthermore, by construction and since $R'$ is a $K_q^r$-multi-refiner of $X$, we find that $|\phi(H)\cap X|=|H\cap X|\le 1$ for all $H\in \mathcal{H}'$. Hence $|H\cap X|\le 1$ for all $H\in \mathcal{H}$.

Let $L\subseteq X$ be $K_q^r$-divisible. By definition, there exists $\mathcal{H}'_L \subseteq \mathcal{H}'$ whose elements are edge-disjoint and decomposes $L\cup R'$. Let $\mathcal{H}_L := \bigcup_{H\in\mathcal{H}'_L} \phi(H)$. It follows that $\mathcal{H}_L$ decomposes $L\cup R$. Hence $R$ is a $K_q^r$-refiner of $X$ with refinement family $\mathcal{H}$. 

Thus it remains to show that $\Delta(\mathcal{H})\le C\cdot \Delta$. If $S \in \binom{V(\phi(H))}{r-1}$ for some $H\in \mathcal{H}'$, then either $S\in \binom{V(H)}{r-1}$ or $S\in E(T_H)$. Hence, it follows that
\begin{align*}
\Delta(\mathcal{H}) &\le \Delta(\mathcal{H}') + \Delta(T) \le (1+C''\cdot C')\cdot \Delta(\mathcal{H}') \le (1+C''\cdot C') \cdot C'' \cdot \Delta'\\
&\le C \cdot \Delta,
\end{align*}
as desired since $C\ge (1+C''\cdot C') \cdot C''$. 
\end{lateproof}

\subsection{Proof of the Refined Omni-Absorber Theorem via Embedding Absorbers}

Finally, we are ready to prove the Refined Omni Absorber Theorem (Theorem~\ref{thm:Omni}) by using Theorem~\ref{thm:Refiner} and then embedding absorbers using the Partial Clique Embedding Lemma (Lemma~\ref{lem:EmbedMinDegree}) as follows.

\begin{lateproof}{thm:Omni}
Note that $X$ is a simple $r$-uniform hypergraph. Let $h, \varepsilon'$ be as $C,\varepsilon$ for $q$ and $r$ in Theorem~\ref{thm:Refiner}. By Theorem~\ref{thm:AbsorberExistence}, it follows that there exists $k$ such that for every $K_q^r$-divisible graph $H$ with $v(H)\le h$, there exists an $K_q^r$-absorber $A_H$ with $\max\{e(A_H),v(A_H)\}\le k$. Let $C',\varepsilon'',\gamma$ be as $C, \varepsilon, \gamma$  for $h,r,k$ in Lemma~\ref{lem:EmbedMinDegree}. Let $\Delta' := \max\left\{\Delta(X), v(X)^{1-\frac{1}{r}}\cdot \log v(X)\right\}$. Let $C:= \max \left\{ \left(\frac{h}{\gamma}\right)^{r^2},~h(h+C'),~k, \left(\frac{4\cdot h^2}{\varepsilon''}\right)^{r^2} \right\}$ and $\varepsilon := \min \left\{ \varepsilon',~\varepsilon'' - \frac{1}{C}-h^2\cdot \Delta'\right\}$. Note that $\varepsilon \ge \min\left\{\varepsilon',\frac{\varepsilon''}{4}\right\} > 0$ since $C\ge \frac{2}{\varepsilon''}$ and $C\ge \left(\frac{4\cdot h^2}{\varepsilon''}\right)^{r^2} $. 

Since $C\ge h$ and $\varepsilon\le \varepsilon'$ we have by Theorem~\ref{thm:Refiner} that there exists an $h$-refined $K_q^r$-refiner $R$ of $X$ such that $\Delta(\mathcal{H}) \le h \cdot \Delta'$. Since $R$ is $h$-refined, it follows that $\Delta(R)\le h\cdot \Delta(\mathcal{H}) \le h^2 \cdot \Delta'$. Hence we have that
$$\delta(G\setminus E(X\cup R)) \ge \delta(G) - \Delta(X) - \Delta(R) \ge (1-\varepsilon'')\cdot v(G)$$ 
 since $\Delta(X) \le \frac{v(X)}{C}$ and $\varepsilon \le \varepsilon''-\frac{1}{C} - h^2\cdot \Delta'$.

Let $\mathcal{H}'$ be the refinement family of $R$. By definition of $h$-refined, we find that $v(H)\le h$ for all $H\in\mathcal{H}$. Let $Y$ be the $h$-bounded multi-hypergraph with $V(Y) :=V(G)$ and $E(Y) := \{ V(H): H \in \mathcal{H}'\}$. Note that  $\Delta_{r-1}(Y)=\Delta(\mathcal{H})$. Hence it follows that $\Delta_{r-1}(Y)\le h\cdot \Delta'$. Since $C$ is large enough, it follows that $\Delta'\le v(X)$. Then since $\Delta(X) \le \frac{v(X)}{C}$ and $C\ge\frac{h}{\gamma}$, we find that $\Delta_{r-1}(Y)\le \gamma\cdot v(G)$.

By Theorem~\ref{thm:AbsorberExistence}, there exists a $K_q^r$-absorber $A_H$ for each $H\in \mathcal{H}'$ and $K_q^r$ decompositions $\mathcal{Q}_{1,H}$ of $A_H$ and $\mathcal{Q}_{2,H}$ of $A_H\cup H$. By definition of $k$, we have that $\max\{ v(A_H): H\in \mathcal{H}'\} \le k$. By Lemma~\ref{lem:EmbedMinDegree}, there exist edge-disjoint subgraphs $(T_e: e\in E(Y))$ of $G\setminus (X\cup R)$ such that $T_e$ is a partial $K_{k+|e|}^r$ rooted at $V(e)$, and letting $T:= \bigcup_{e\in E(Y)} T_e$, we have that $\Delta(T) \le C'\cdot \Delta_{r-1}(Y)$.

For each $H\in E(Y)$, embed the hypergraph $A_H$ inside $T_H$ and denote it by $U_H$. Let $A:= R\cup \bigcup_{H\in E(Y)} U_H$. Note that $A$ is a simple $r$-uniform hypergraph edge-disjoint from $X$. Let $\mathcal{H} := \bigcup_{H\in E(Y)} \mathcal{Q}_{1,H}\cup \mathcal{Q}_{2,H}$.

Now we prove that $A$ is a $C$-refined $K_q^r$-omni-absorber for $X$ such that $\Delta(A)\le C\cdot \Delta$.

First we prove that $A$ is $C$-refined. To that end,  let $e\in X\cup A$. First suppose that $e\in A\setminus R = \bigcup_{H\in E(Y)} U_H$. Since the $U_H$ are pairwise edge-disjoint, we have that $e$ is in $U_H$ for only one $H\in E(Y)$. But then $|\mathcal{H}(e)|\le 2$ since $e$ is only in $\mathcal{Q}_{1,H}$ and $\mathcal{Q}_{2,H}$.

So we assume that $e\in R\cup X$. Since $R$ is $h$-refined, we find that $|\mathcal{H}'(e)|\le h$. But then $|\mathcal{H}(e)|\le h$ since $e$ is only in $\mathcal{Q}_{2,H}$ for each such $H$. Since $C\ge \max\{2,h\}$, we have that $A$ is $C$-refined.

Next we prove that $A$ is a $K_q^r$-omni-absorber. First note that every $H\in \mathcal{H}$ is isomorphic to $K_q^r$. Next note that since $R$ is a $K_q^r$-refiner and $A$ is edge-disjoint from $X$, we find that $|H\cap X| \le 1$ for every $H\in \mathcal{H}$.

Now let $L$ be a $K_q^r$-divisible subgraph of $X$. Since $R$ is a $K_q^r$-refiner of $X$, there exists $\mathcal{H}_L \subseteq \mathcal{H}'$ that decomposes $L\cup R$.
We let
$$\mathcal{Q}_L:= \bigcup_{H\in \mathcal{H}_L} \mathcal{Q}_{2,H}~~\cup \bigcup_{H\in \mathcal{H'}\setminus \mathcal{H}_L} \mathcal{Q}_{1,H}.$$
We note that elements of $\mathcal{Q}_L$ are edge-disjoint by construction. It follows that $\mathcal{Q}_L$ is a $K_q^r$-decomposition of $L\cup A$. Hence $\mathcal{Q}_L$ is a decomposition function for $A$. Hence $A$ is a $K_q^r$-omni-absorber with decomposition family $\mathcal{H}$.

Thus it remains to show $\Delta(A)\le C\cdot \Delta$. We calculate that 
\begin{align*}
\Delta(A) &\le \Delta(R)+\Delta(T) \le h^2 \cdot \Delta'+C'\cdot \Delta_{r-1}(Y) \le h(h+C') \cdot \Delta' \\
&\le C\cdot \Delta    
\end{align*}
as desired since $C\ge h(h+C')$ and $\Delta\ge \Delta'$ as $C\ge h$.
\end{lateproof}

\section{Proof of the Existence Conjecture}\label{s:Existence}

We are now almost prepared to provide our new alternate proof of the Existence Conjecture (Theorem~\ref{thm:Existence}).

\subsection{More Preliminaries}

First we prove Lemma~\ref{lem:RandomX} via a standard application of the Chernoff bound~\cite{AS16}.

\begin{lateproof}{lem:RandomX}
We choose $\varepsilon > 0$ small enough as needed throughout the proof. Let $X$ be the spanning subhypergraph of $G$ obtained by choosing each edge of $G$ independently with probability $p$.

For an $(r-1)$-set $S\in \binom{v(G)}{r-1}$, we have that $|G(S)|\le v(X)$ and hence
$$\Expect{|X(S)|}=p\cdot |G(S)|\le p\cdot v(X).$$
Let $A_S$ be the event that $|X(S)| > 2p\cdot v(X)$.
By the Chernoff bound~\cite{AS16}, we find that 
$$\Prob{A_S} \le e^{-p\cdot v(X) / 3} < \frac{1}{2\cdot v(G)^{r-1}}$$
since $p\ge 3r\cdot \left(\frac{\log v(G)}{v(G)}\right)$ as $p\ge v(G)^{-\varepsilon}$ and $v(G)\ge \frac{1}{\varepsilon}$ and $\varepsilon$ is small enough. Hence $$\Prob{ \bigcup_{S\in\binom{V(G)}{r-1}} A_S} < \frac{1}{2}.$$

For $T\subseteq V(G)$ with $r-1\le |T|\le q-1$, define 
$$N_G(T):= \left\{ v\in V(G)\setminus T: v\cup S\in G \text{ for all } S\in \binom{T}{r-1}\right\}.$$ 
Since $\delta(G)\ge (1-\varepsilon)\cdot v(G)$, it follows that 
$$|N_G(T)| \ge \left(1-\binom{i}{r-1}\cdot \varepsilon\right)\cdot v(G) \ge \frac{v(G)}{2}$$
since $\varepsilon < \frac{1}{2\cdot \binom{q}{r-1}}$ as $\varepsilon$ is small enough. 
Similarly let
$$N_X(T):= \left\{ v\in V(G)\setminus T: v\cup S\in X \text{ for all } S\in \binom{T}{r-1}\right\}.$$
For $i\in \{r-1,\ldots, q-1\}$, define
$$\mu_i:= p^{\binom{i}{r-1}}\cdot \frac{v(G)}{2}.$$
Since $p\ge v(G)^{-\varepsilon}$ and $v(G)\ge \frac{1}{\varepsilon}$ and $\varepsilon$ is small enough, we find for all $i\in \{r-1,\ldots,q-1\}$ that 
$$\mu_i \ge 8 \log(2q\cdot v(G)^{q-1}).$$
For $v\in N_G(T)$, we have that $\Prob{v\in N_X(T)} = p^{\binom{|T|}{r-1}}$. Thus by Linearity of Expectation, we find that
$$\Expect{|N_X(T)|} = \sum_{v\in N_G(T)} \Prob{v\in N_G(T)} = p^{\binom{|T|}{r-1}}\cdot |N_G(T)|\ge \mu_{|T|}.$$
Note then that $|N_X(T)|$ is the sum of independent Bernoulli $\{0,1\}$-random variables. Let $B_T$ be the event that $|N_X(T)| < \frac{\mu_{|T|}}{2}$.  By the Chernoff bound, we find that
$$\Prob{B_T} \le e^{-\mu_{|T|}/ 8} < \frac{1}{2q\cdot v(G)^{q-1}}.$$
Hence 
$$\Prob{ \bigcup_{T} B_T} < \frac{1}{2}.$$

Hence it follows that with some positive probability that none of the $A_S$ or $B_T$ happen. Since none of the $A_S$ happen, it follows that $\Delta(X)\le 2p\cdot v(G)$. For $e\in G$, since none of the $B_T$ happen it follows that the number of $K_q^r$s in $X\cup \{e\}$ containing $e$ is at least
$$\frac{1}{(q-r)!}\cdot \prod_{i=r}^{q-1} \frac{\mu_i}{2} \ge \frac{1}{4^{q-r}\cdot (q-r)!} \cdot p^{\binom{q}{r}-1}\cdot v(G)^{q-r} \ge \varepsilon\cdot p^{\binom{q}{r}-1}\cdot v(G)^{q-r},$$
as desired since $\varepsilon \le \frac{1}{4^{q-r}\cdot (q-r)!}$ as $\varepsilon$ is small enough.
\end{lateproof}

\subsection{Proof of the Existence Conjecture}

Here we prove the stronger minimum degree version of the $\lambda=1$ case of the Existence Conjecture as follows (which was also already proved for all constant $\lambda$ earlier by Keevash~\cite{K14} and Glock, K\"{u}hn, Lo, and Osthus~\cite{GKLO16}).

\begin{thm}[Existence Conjecture - Minimum Degree Version - $\lambda=1$ case]\label{thm:ExistenceMinDegree}
For all integers $q > r \geq 2$, there exist reals $n_0$ and $\varepsilon > 0$ such that for all $K_q^r$-divisible $r$-uniform hypergraphs $G$ with $v(G)\ge n_0$ and $\delta(G) \ge (1-\varepsilon)\cdot v(G)$, there exists a $K_q^r$-decomposition of $G$.
\end{thm}
\begin{proof}
We choose $n_0$ large enough as needed throughout the proof. Let $\varepsilon_1$, $\varepsilon_2$, $\varepsilon_3$ be as $\varepsilon$ in Theorem~\ref{thm:Omni} and Lemmas~\ref{lem:RandomX} and~\ref{lem:RegBoost}, respectively. Let $\varepsilon := \min \left\{\varepsilon_1,\varepsilon_2,\frac{\varepsilon_3}{2}\right\}$.

Let $C$ be as in Theorem~\ref{thm:Omni} for $q$ and $r$. Let $\beta:=\frac{1}{4(q-r)}$. Let $D_{\beta}$ and $\alpha$ be as in Theorem~\ref{thm:NibbleReserves} for $\binom{q}{r}$ and $\beta$. Let $\sigma = \min\left\{ \varepsilon_2, \frac{\alpha}{r\cdot \binom{q}{r}}\right\}$ and let $p:=v(G)^{-\sigma}$. Let $\Delta:= 2p\cdot v(G)$. 

Since $\varepsilon\le \varepsilon_2$, we have that $\delta(G)\ge (1-\varepsilon_2)\cdot v(G)$. Since $\sigma \le \varepsilon_2$, we have that $p\ge v(G)^{-\varepsilon_2}$. We also have that $v(G)\ge n_0\ge \frac{1}{\varepsilon_2}$ since $n_0$ is large enough. Thus by Lemma~\ref{lem:RandomX}, there exists a spanning subhypergraph  $X\subseteq G$ such that $\Delta(X)\le 2p\cdot v(G) = \Delta$ and for all $e\in G$, there exist at least $\varepsilon_2 \cdot p^{\binom{q}{r}-1}\cdot v(G)^{q-r}$ $K_q^r$'s in $X\cup \{e\}$ containing $e$.

Since $\varepsilon\le \varepsilon_1$, we have that $\delta(G) \ge (1-\varepsilon_1)\cdot v(G)$. Note that $\Delta \le \frac{v(X)}{C}$ since $v(G)\ge n_0$ is large enough. Similarly note that $\Delta \ge v(X)^{1-\frac{1}{r}}\cdot \log v(X)$ since $\Delta \ge v(X)^{1-\frac{1}{2r}}$ and $v(X)=v(G)\ge n_0$ is large enough. Thus by Theorem~\ref{thm:Omni}, there exists a $C$-refined $K_q^r$-omni-absorber $A$ for $X$ such that $A\subseteq G$ and $\Delta(A)\le C \cdot \Delta$.

Let $J:= G\setminus (X\cup A)$. Now $$\delta(J) \ge \delta(G) - \Delta(A) - \Delta(X) \ge (1-\varepsilon)\cdot v(G) - (1+C)\cdot \Delta \ge (1-\varepsilon_3)\cdot v(J),$$
where we used that $\varepsilon\le \frac{\varepsilon_3}{2}$ and that $v(J)=v(G)\ge n_0$ is large enough. Thus by Lemma~\ref{lem:RegBoost}, there exists a subhypergraph $\mathcal{D}'$ of ${\rm Design}_{K_q^r}(J)$ such that $d_{\mathcal{D}'}(v) = (1 \pm v(J)^{-(q-r)/3}) \cdot \frac{1}{2(q-r)!} \cdot v(J)^{q-r}$ for all $v\in V(\mathcal{D}')$.

Let $G_1:= \mathcal{D}'$ and let $G_2$ be such that 
$$V(G_2) := E(J)\cup E(X),~~E(G_2) := \{ S\subseteq E(J)\cup E(X): |S\cap E(J)|=1,~ S \text{ is isomorphic to } K_q^r\}.$$ Let $$G:= G_1\cup G_2.$$ 
Note that $G_1$ and $G_2$ are $\binom{q}{r}$-uniform and hence so is $G$. Let $A:=E(J)$ and $B:=E(X)$ and note that $G_2$ is a bipartite hypergraph with parts $A$ and $B$.
Also note that $G_1$ and $G_2$ are edge-disjoint, that $V(G_1)=E(J)=A$ and hence that $V(G_1)\cap V(G_2)=E(J)=A$. 

We seek to apply Theorem~\ref{thm:NibbleReserves} to find a $V(G_1)$-perfect matching of $G$. To that end, let 
$$D:= (1 + v(J)^{-(q-r)/3}) \cdot \frac{1}{2(q-r)!} \cdot v(J)^{q-r}.$$  
Since $v(J)$ is large enough as $n_0$ is large enough, we have that $D\ge D_{\beta}$. Now we check the codegrees of $G$ and the degrees of $G_1$ and $G_2$.

First we check the codegrees of $G$. Note that the codegrees of $G$ are at most 
$$v(J)^{q-(r+1)} \le D^{1-\beta},$$
where the inequality follows since $\beta < \frac{1}{q-r}$ and $v(J)\ge n_0$ is large enough. 

Next we check the degrees of $G_1$. Now every vertex of $G_1$ has degree at most $D$ in $G_1$. Similarly every vertex of $G_1$ has degree at least 
$$(1 - 2\cdot v(J)^{-(q-r)/3})\cdot D \ge D\cdot (1-D^{-\beta}),$$ 
where the inequality follows since $\frac{q-r}{3} > \frac{1}{4} \ge \beta$ and $v(J)\ge n_0$ is large enough. 

Finally we check the degrees of $G_2$. Note that every edge of $G_2$ contains at least two vertices of $B$ since $\binom{q}{r}\ge 3$ as $q> r\ge 2$. It follows from that fact that every vertex in $G_2$ has degree in $G_2$ at most 
$$v(X)^{q-r-1}\cdot \Delta(X) \le D$$ 
where the inequality follows since $\Delta(X) \le \frac{v(X)}{2(q-r)!}$ by the choice of $p$. Meanwhile, every vertex of $G_1$ has degree in $G_2$ at least 
$$\varepsilon_2\cdot p^{\binom{q}{r}-1} \cdot v(G)^{q-r} \ge D^{1-\alpha},$$
where the inequality follows since $\sigma \le \frac{\alpha}{r\binom{q}{r}}$ and hence $p^{\binom{q}{r}-1} \le v(G)^{-\alpha/2}$ (and we used that $v(G)$ is large enough).

By Theorem~\ref{thm:NibbleReserves}, we have that there exists a $V(G_1)$-perfect matching $\mathcal{Q}_1$ of $G$. Note that $\mathcal{Q}_1$ is also then a $K_q^r$-packing of $G$. Let $Q_1 := \bigcup \mathcal{Q}_1$. Thus we have that $E(J)\subseteq E(Q_1)$. Note that $Q_1$ is $K_q^r$-divisible since it admits the $K_q^r$ decomposition $\mathcal{Q}_1$. 

Let $L := X\setminus Q_1$. Note that $A$ is $K_q^r$-divisible as it admits a $K_q^r$-decomposition by the definition of refined omni-absorber. Since $G$ is $K_q^r$-divisible by assumption and $A$ and $Q_1$ are edge-disjoint $K_q^r$-divisible subgraphs of $G$, it follows that $L = G\setminus (Q_1\cup A)$ is $K_q^r$-divisible . 

By the definition of refined omni-absorber, it follows that $L\cup A$ admits a $K_q^r$-decomposition $\mathcal{Q}_2$. But then $\mathcal{Q}_1\cup\mathcal{Q}_2$ is a $K_q^r$ decomposition of $G$ as desired.
\end{proof}

We also note the proof of Theorem~\ref{thm:ExistenceMinDegree} given above easily generalizes to prove the existence of designs with parameters $(n,q,r,\lambda)$ provided $n$ is large enough (that is for all constant $\lambda$ and not just $\lambda=1$). Namely, for constant $\lambda$, we may still take $X$, $A$ and $A\cup X$ to be simple. Then when we invoke nibble we do so for the remaining multi-hypergraph where we also delete from the edges of the design hypergraph any $q$-sets used by the decomposition family $\mathcal{H}$ of $A$. This concludes our proof of the Existence Conjecture.

\subsection{Required Ingredients for Glock, K\"{u}hn, Lo, and Osthus' Absorber Proof}

The informed reader might point out the proof of the existence of absorbers by Glock, K\"{u}hn, Lo, and Osthus (namely of Theorem~\ref{thm:AbsorberExistence}) requires not only that the Existence Conjecture hold for smaller uniformities but also some additional stronger results. While they prove the existence of absorbers inside a `supercomplex', which is a harder embedding theorem, even their implicit proof of `abstract'/`unembedded' absorbers requires more. In particular, they inductively require that the following holds for smaller uniformities (i.e.~for all $r' < r$): 

\begin{thm}\label{thm:CogirthAbsorber}
Let $q>r\ge 1$ be integers. If $L$ is a $K_q^r$-divisible $r$-uniform hypergraph, then there exists a $K_q^r$-absorber $A$ of $L$ and $K_q^r$ decompositions $\mathcal{Q}_1$ of $A$ and $\mathcal{Q}_2$ of $A\cup L$ such that for all $Q_1\in \mathcal{Q}_1$ and $Q_2\in\mathcal{Q}_2$, we have $|V(Q_1)\cap V(Q_2)|\le r$.
\end{thm}

We note that Theorem~\ref{thm:CogirthAbsorber} would follow from proving the minimum degree version of the Existence Conjecture with the following additional outcome about the decomposition of $G$.

\begin{thm}\label{thm:MinCogirthSteiner}
For all integers $q > r \geq 1$, there exist reals $n_0$ and $\varepsilon > 0$ such that the following holds: Let $G$ be a $K_q^r$-divisible $r$-uniform hypergraphs $G$ with $v(G)\ge n_0$ and $\delta(G) \ge (1-\varepsilon)\cdot v(G)$. If $\mathcal{Q}_1$ is a $K_q^r$-packing of $G$, then there exists a $K_q^r$-decomposition $\mathcal{Q}_2$ of $G$ such that for all $Q_1\in \mathcal{Q}_1$ and $Q_2\in\mathcal{Q}_2$, we have $|V(Q_1)\cap V(Q_2)|\le r$.
\end{thm}

It is relatively straightforward to adapt our proof of Theorem~\ref{thm:ExistenceMinDegree} to guarantee a $K_q^r$-decomposition $\mathcal{Q}_2$ as desired. Namely, there are two important modifications but first we require a better construction of absorbers which follows easily from Theorem~\ref{thm:AbsorberExistence}.

\begin{lem}\label{lem:BetterAbsorber}
Let $q > r\ge 1$ be integers. If $L$ is a $K_q^r$-divisible hypergraph, then there exists a $K_q^r$-absorber $A$ for $L$ such that $A\cup L$ admits a $K_q^r$ decomposition $Q$ such that $|V(Q)\cap V(L)|\le r$ for all $Q\in \mathcal{Q}$ and with equality only if $V(Q)\cap V(L) = V(e)$ for some $e\in L$.
\end{lem}
\begin{proof}
Let $L'$ be obtained from $L$ as follows: for each edge $e\in L$, add $M(q,r)$ partial $K_q^r$'s rooted at $V(e)$ labeled $P_{1,e},\ldots, P_{M(q,r),e}$ (in such a way that all new vertices are distinct from each other and disjoint from $V(L)$).

Let $\mathcal{Q}'$ be the $K_q^r$-packing of $L'$ defined as $\{ \{e\}\cup P_{1,e}: e\in L\}$. Let $L_1 := L'\setminus L$ and $L_2:= L'\setminus \bigcup \mathcal{Q}'$. 

By Theorem~\ref{thm:AbsorberExistence}, there exists an absorber $A_1$ of $L_1$ with $K_q^r$ decompositions $\mathcal{Q}_{1,1}$ of $A_1$ and $\mathcal{Q}_{1,2}$ of $A_1\cup L_1$. Similarly by Theorem~\ref{thm:AbsorberExistence}, there exists an absorber $A_2$ of $L_2$ with $K_q^r$ decompositions $\mathcal{Q}_{2,1}$ of $A_2$ and $\mathcal{Q}_{2,2}$ of $A_2\cup L_2$.

Now let $A:= L_1\cup A_1\cup A_2$. Note that $V(L)\subseteq V(A)$ is independent in $A$ since $V(L)$ is independent in $L_1$ and $V(L)$ is independent in $A_1\cup A_2$ by the definition of absorber.

Now $\mathcal{Q}_{1,2}\cup \mathcal{Q}_{2,1}$ is a $K_q^r$ decomposition of $A$. Meanwhile, $\mathcal{Q} = \mathcal{Q}'\cup \mathcal{Q}_{1,1}\cup \mathcal{Q}_{2,2}$ is a $K_q^r$ decomposition of $L\cup A$ such that $|V(Q)\cap V(L)|\le r$ for all $Q\in \mathcal{Q}$ and with equality only if $V(Q)\cap V(L) = V(e)$ for some $e\in L$, as desired.
\end{proof}

Now here are the two modifications to prove Theorem~\ref{thm:MinCogirthSteiner}. First, when embedding private absorbers to prove Theorem~\ref{thm:Omni}, here we have to avoid partial cliques that use $q$-sets that intersect elements of $\mathcal{Q}_1$ in more than $r$ vertices. Using the better absorbers of Lemma~\ref{lem:BetterAbsorber}, it follows that there are only $O(v(X)^{k-1})$ such $k$-sets for a given $e\in Y$ to further avoid. So we simply use a modified form of Lemma~\ref{lem:EmbedMinDegree} for the proof of Theorem~\ref{thm:Omni} whose partial cliques avoid these as $k$-sets as well. 

Second, in the use of nibble we have to avoid using $q$-sets that intersect elements of $\mathcal{Q}_1$ in more than $r$ vertices. Since there are only $O(v(X)^{q-r+1})$ such $q$-sets containing a fixed edge of $J$, we may just delete all such $q$-sets from $\mathcal{D}'$ (that is, even after regularity boosting $J$). We omit the details.


\begin{remark}
The interested reader might find it helpful to know that implicit in Keevash's work~\cite{K14} is an alternate proof of the existence of `abstract' absorbers (Theorem~\ref{thm:AbsorberExistence}). Of course, this is not surprising since such a theorem follows easily from the minimum degree version of the Existence Conjecture (Theorem~\ref{thm:ExistenceMinDegree}) which as we mentioned Keevash also proved via his method of random algebraic constructions. But in fact only a small portion of Keevash's work is required to prove Theorem~\ref{thm:AbsorberExistence}, roughly equal in length to that of the implicit proof of Theorem~\ref{thm:AbsorberExistence} in the work of Glock, K\"{u}hn, Lo, and Osthus. Thus there are two different approaches (one combinatorial, one algebraic) to proving Theorem~\ref{thm:AbsorberExistence} which we use as the base of our approach.    
\end{remark}

\section{Concluding Remarks and Further Directions}\label{s:Conclusion}

To summarize, in this paper, we introduce our novel method of refiners and refined absorbers. We prove key theorems about these objects, in particular our Refined Omni-Absorber Theorem, Theorem~\ref{thm:Omni}. We use this theorem to provide a new proof of the Existence Conjecture for Combinatorial Designs modulo the proof of the existence of absorbers. The extreme efficiency of our omni-absorbers allows us to use a one-level edge vortex. 

As mentioned in the introduction, our new proof structure is quite robust in its use of Theorem~\ref{thm:Omni}. In particular, since the omni-absorber is $C$-refined, it will be possible to ``boost'' the cliques in the decomposition family for other applications.

We proceed with some of these applications in a series of follow-up papers as follows. In one follow-up paper~\cite{DPII}, we construct \emph{girth boosters} and use these with Theorem~\ref{thm:Omni} to prove the existence of high girth omni-absorbers. We combine those with our high girth nibble from~\cite{DP22} to prove the High Girth Existence Conjecture, the common generalization of the Existence Conjecture and Erd\H{o}s' Conjecture on the Existence of High Girth Steiner Triple Systems (which was only recently proved by Kwan, Sah, Sawhney, and Simkin in~\cite{KSSS22}). 

In another follow-up paper~\cite{DKPIII} - joint with Tom Kelly - we apply our new proof structure and Theorem~\ref{thm:Omni} to make significant progress on conjectures about the existence of designs in $G(n,p)$ and random-regular graphs. This requires the existence of sparser absorbers but also embedding them and fake edges into $G(n,p)$ or $G_{n,d}$ respectively.

In a separate follow-up paper~\cite{DKPIV} - also joint with Tom Kelly -  we apply our new proof structure and Theorem~\ref{thm:Omni} to make significant progress on a conjecture on the threshold for $(n,q,2)$-Steiner systems. This requires proving the existence of a spread distribution on omni-absorbers which is accomplished via the use of \emph{spread boosters}.

Altogether, we believe this new proof structure but also the ideas used in the construction of refiners herald much potential for future applications.

\section*{Acknowledgements}

The authors would like to thank Tom Kelly for helpful comments when preparing this manuscript that improved the exposition; in particular, we thank Tom for pointing out that our construction in Proposition~\ref{prop:RMHG} could be viewed as a robustly matchable hypergraph somewhat akin to Montgomery's RMBG. 

\bibliographystyle{plain}
\bibliography{mdelcourt, ref}

\end{document}